\documentclass[10pt]{article}

\usepackage{amsfonts,amsopn,amsthm,amsmath}
\usepackage[utf8]{inputenc}
\usepackage[english]{babel}
\usepackage{graphicx,epstopdf}
\usepackage{tabularx,booktabs}
\usepackage{color,url}
\usepackage{multirow}
\usepackage{algorithm,algorithmic}
\usepackage{caption}
\usepackage{bm}
\usepackage[numbers]{natbib}
\graphicspath{{./pics/}}

\newtheorem{remark}{Remark}
\newtheorem{lemma}{Lemma}
\newtheorem{proposition}{Proposition}
\newtheorem{theorem}{Theorem}

\ifpdf
  \DeclareGraphicsExtensions{.eps,.pdf,.png,.jpg}
\else
  \DeclareGraphicsExtensions{.eps}
\fi

\begin{document}

\title{A Parameter Choice Rule for Tikhonov Regularization\\ Based on Predictive Risk}

\author{Federico Benvenuto\thanks{Dipartimento di Matematica, Universit\`{a} degli Studi di Genova, Via Dodecaneso 35, 16146, Genova, Italy. \texttt{benvenuto@dima.unige.it}} \and Bangti Jin\thanks{Department of Computer Science, University College London, Gower Street, London WC1E 6BT, UK. \texttt{bangti.jin@gmail.com,b.jin@ucl.ac.uk}}}

\maketitle

\begin{abstract}
In this work, we propose a new criterion for choosing the regularization parameter in Tikhonov
regularization when the noise is white Gaussian. The criterion minimizes a lower bound of the
predictive risk, when both data norm and noise variance are known, and the parameter choice
involves minimizing a function whose solution depends only on the signal-to-noise ratio. Moreover,
when neither noise variance nor data norm is given, we propose an iterative algorithm which
alternates between a minimization step of finding the regularization parameter and an estimation
step of estimating signal-to-noise ratio. Simulation studies on both small- and large-scale datasets
suggest that the approach can provide very accurate and stable regularized inverse solutions and,
for small sized samples, it outperforms discrepancy principle, balancing principle,
unbiased predictive risk estimator, L-curve method generalized cross validation, and
quasi-optimality criterion, and achieves excellent stability hitherto unavailable.
\end{abstract}


\section{Introduction}

In this work, we study discrete linear inverse problems of recovering an unknown object from noisy indirect measurements.
When the matrix is ill-conditioned, the solution given by the generalized inverse is generally unsatisfactory,
especially in the presence of data noise. Then one often employs a regularization method, which provides a one-parameter
family of candidate solutions and a criterion for selecting the optimal parameter (and a corresponding solution)
\cite{Tikhonov,Engl,ItoJin:2015}. In the case of white Gaussian noise, popular linear regularization methods include
truncated SVD \cite{hansen1987truncatedsvd}, Tikhonov regularization \cite{Tikhonov} and Landweber method
\cite{landweber1951iteration}, and they all require specifying  one scalar parameter, i.e., regularization parameter,
which is notoriously challenging. A large number of choice rules have been proposed in the literature. Most popular
criteria include discrepancy principle (DP) \cite{Morozov:105020}, Unbiased Predictive Risk Estimator (UPRE)
\cite{Stein:1981} and balancing principle (BP) / Lepskii's principle \cite{Lepskii:1990,MathePereverzev:2003,BauerHohage:2005,Mathe:2006}
(see also \cite{HamarikRuas:2009} and references therein), when the noise level is given,
and Generalized Cross Validation (GCV) \cite{golub1979generalized,wahba1990spline}, L-Curve (LC) method \cite{hansen1992analysis,
HansenOLeary:1993} (see also \cite{Reginska:1996,ItoJinTakeuchi:2011} for an algebraic variant), quasi-optimality criterion (QOC)
\cite{TikhonovGlasko:1965,TikhonovGlasko:1979} and Hanke-Raus rule \cite{HankeRaus:1996}, when the noise level is unknown.
The heuristic rules (without requiring a knowledge of the noise level) can categorically grouped into the general context
of choice rules based on functional minimization. See the review \cite{HamarikPalmRaus:2009} for detailed experimental
comparison and in-depth discussions for various linear regularization strategies; see also the work \cite{Kindermann:2011}.
However, none of the aforementioned methods can consistently find a near-optimal regularization parameter for all test problems
and noise realizations in simulation studies.

DP, UPRE and GCV have a common drawback in that they depend heavily on the
sample realization and thus can lead to instable solutions. In particular, DP selects a regularized solution
with a predictive error not smaller than the given noise level \cite{Morozov:105020}. Thus, when the predictive error
is larger (smaller) than the true noise level or the noise level is very small (large), it leads to under-regularized
(over-regularized) solutions, and the error of the approximate solutions may be arbitrarily large. Indeed, in practice,
DP often includes an additional tuning parameter \cite{Engl}. BP is
theoretically better than DP \cite{MathePereverzev:2003,HamarikPalmRaus:2009}: the convergence holds if the ratio of
the actual and given noise is bounded (which may be larger than one), thereby avoiding the convergence issue of DP.
GCV provides an asymptotically unbiased estimator of the predictive risk, but for small sized samples
it can give unstable solutions \cite{wahba1990spline, lukas2012performance}, due to the flatness of the GCV curve.
Finally, UPRE can fail to approximate the predictive error when the data size is not large enough
or the data norm is too small \cite{Craven1978}. Further, the LC method still defies a complete
convergence analysis even though it performs extremely well in many applications \cite{vogel1996}. QOC
is very prominent among heuristic rules: it is easy to implement, gives excellent empirical performance, and further,
both convergence and convergence rates have also been established for a wide variety of noise models, including the infinite-dimensional setting \cite{KindermannNeubauer:2008,BauerKindermann:2008,JinLorenz:2010,Kindermann:2011}. However, it fails almost surely for severely
ill-posed problems with Gaussian white noise \cite{KindermannPereverzyev:2018}.

In this work we propose a novel method for choosing the crucial Tikhonov regularization parameter, which
falls within the realm of rules based on functional minimization. The proposed strategy is inspired by the
predictive risk, an idea that has gained increasing attention in recent years (see the works \cite{Lucka:2018,
LiWerner:2018} and references therein). However, instead of minimizing the predictive risk or its unbiased estimators
directly, it minimizes a lower bound of the predictive risk, when both data norm and noise level are known. The lower
bound is constructed only for the approximation error (a.k.a. bias) and does not change the noise amplification
(a.k.a. variance). This construction ensures that the method chooses a regularization parameter that is an
upper bound of the optimal parameter (with respect to the predictive risk). The minimizer of the lower bound depends
only on the ratio between the data norm and noise level, i.e. signal-to-noise ratio (SNR), and the procedure is
termed as predictive risk optimization (\texttt{PRO}). In practice, this lower bound is quite tight, and thus the choice
rule enjoys excellent accuracy and stability. We discuss its utility under three different information availability scenarios:
(i) SNR is known, (ii) the data noise level is known, but the true data norm is unknown and (iii) both noise level and data
norm are unknown. In case (i), \texttt{PRO} can be applied directly. In case (ii), we employ an unbiased estimator of
the data norm to efficiently approximate \texttt{PRO}. In case (iii), we propose an iterative algorithm, each iteration
of which is composed of two steps: given estimated data norm and noise variance, one uses \texttt{PRO} to find the regularization
parameter, and given the estimated regularization parameter, one re-estimates the data norm and noise variance. The procedure
is termed as iterative \texttt{PRO} (or \texttt{I-PRO}). Theoretically, we prove a number of properties of the rule for Tikhonov
regularization, e.g., convergence to zero as the noise level tends to zero, upper bound property and the monotone convergence of
the \texttt{I-PRO} procedure. Further, to demonstrate their performance, we carry out extensive simulation studies on benchmark
problems from two public software packages, i.e., Regularization tools \cite{hansen1999regularization} and AIR tools \cite{Hansen20122167}, which involve
small-scale 1d applications and large-scale 2d tomography problems, respectively, and give a detailed comparative study
with several popular existing choice rules.

The rest of the paper is organized as follows. In Section \ref{sec:background}, we describe the setup of a discrete inverse
problem in a Gaussian framework and introduce the \texttt{PRO} criterion. In Section \ref{sec:pro}, we describe optimization
strategies for three different information availability scenarios. In Section \ref{sec:Tikh}, we investigate the \texttt{PRO}
criterion for Tikhonov regularization, establish the well-posedness of the rule and bounds on the chosen parameter, analyze the convergence
of the \texttt{I-PRO} scheme, and describe their efficient implementation for large-scale problems. In Section \ref{sec:numer}, we
report simulation results and discuss the pros and cons of \texttt{PRO} and \texttt{I-PRO}, when compared with existing choice rules.
Section \ref{sec:concl} contains a summary of the work.

\section{Background and motivation}\label{sec:background}
Consider the linear inverse problem
\begin{equation}\label{eq:lsip}
\mathbf{g}^\dag = A \mathbf{f}^\dag,
\end{equation}
where $\mathbf{f}^\dag\in\mathbb{R}^m$,  $\mathbf{g}^\dag\in \mathbb{R}^n$ and $A\in\mathbb{R}^{n\times m}$. The goal is to recover the true signal
$\mathbf{f}^\dag$ from a given noisy measurement of $\mathbf{g}^\dag$, denoted by $\mathbf{g}^\eta$.  Throughout, $\mathbf{g}^\eta := \mathbf{g}^\dag + \boldsymbol{\eta}$, where
$\boldsymbol{\eta}$ is a Gaussian distributed random vector with $\mathbb{E}_{\boldsymbol{\eta}}[\boldsymbol{\eta}]=0$,
$\mathbb{E}_{\boldsymbol{\eta}}[\eta_i \eta_j]=\sigma^2 \delta_{ij}$, with $\delta_{ij}$
denote the standard Kronecker symbol, where $\mathbb{E}_{\boldsymbol{\eta}}[\cdot]$ denotes taking expectation with respect to the
distribution of ${\boldsymbol{\eta}}$. A linear regularization method $R_\alpha$ provides a family of estimates of the solution
\begin{equation}\label{eq:reg}
\mathbf{f}_\alpha^\eta = R_\alpha \mathbf{g}^\eta,
\end{equation}
for any realization $\mathbf{g}^\eta$. Upon ignoring the usual scaling factor $1/n$, the predictive risk $p_\alpha(\mathbf{g}^\eta)$ is defined as
\begin{equation*}
p_\alpha(\mathbf{g}^\eta) =  \mathbb{E}_{{\boldsymbol{\eta}}} [\| \mathbf{g}_\alpha^\eta - \mathbf{g}^\dag \|^2 ],
\end{equation*}
where 
$\|\cdot\|$ denotes the Euclidean norm, and $\mathbf{g}_\alpha^\eta = X_\alpha \mathbf{g}^\eta$ is the predictive data, where
\begin{equation}\label{eq:io}
X_\alpha = A R_\alpha
\end{equation}
is the so-called influence matrix. Since the predicted data $\mathbf{g}^\eta_\alpha$ can be split into
\begin{equation*}
\mathbf{g}^\eta_\alpha = X_\alpha \mathbf{g}^\dag + X_\alpha {\boldsymbol{\eta}},
\end{equation*}
by the standard bias-variance dcomposition, the predictive risk $p_\alpha(\mathbf{g}^\eta)$ is given by
\begin{equation}\label{bias-variance}
p_\alpha(\mathbf{g}^\eta) = \| (X_\alpha - I) \mathbf{g}^\dag\|^2 + \sigma^2 \|X_\alpha\|_F^2,
\end{equation}
where $\|\cdot\|_F$ is the Frobenius norm. Equation \eqref{bias-variance} indicates the predictive risk $p_\alpha(\mathbf{g}^\eta)$
consists of two terms: the bias term $\| (X_\alpha - I) \mathbf{g}^\dag\|^2 $ due to the approximation error, which depends only the regularity of the
exact solution $\mathbf{f}^\dag$, and the variance term $\sigma^2 \|X_\alpha\|_F^2$ due to noise amplification.

Note that the expression in
\eqref{bias-variance} is not directly computable, since the exact data $\mathbf{g}^\dag$ is unknown. However, this
quantity may be estimated using sample data $\mathbf{g}^\eta$ to obtain various predictive risk estimators for parameter choice.
This idea has long been pursued in the literature; see, e.g., \cite{Mallows:1973,Thompson:1991,GalKat:1992,Ramani:2008,DelVaiter:2014,
LevinMeltzer:2017,Lucka:2018,LiWerner:2018} and references therein for various practical applications and theoretical developments.
For example, with $\mathrm{tr}(\cdot)$ denoting the trace of a matrix, the following two estimators are frequently adopted:
\begin{equation*}
 \hat p_\alpha(\mathbf{g}^\eta) = \|A\mathbf{f}^\eta_\alpha-\mathbf{g}^\eta\|^2 - 2\sigma^2\mathrm{tr}(I-X_\alpha)
\end{equation*}
for UPRE and
\begin{equation*}
  \tilde{p}_\alpha(\mathbf{g}^\eta) = \frac{\|A\mathbf{f}^\eta_\alpha-\mathbf{g}^\eta\|^2}{\mathrm{tr}(I-X_\alpha)^2}
\end{equation*}
for GCV, when the noise level $\sigma^2$ is known and unknown, respectively. We refer to \cite{LevinMeltzer:2017}
for further approximations of the predictive risk in parameter choice. However, it is known that they tend to
suffer from the notorious overfitting issue, i.e., the selected parameter tends to be too small \cite{Lucka:2018},
which is also confirmed by the simulation study in Section \ref{sec:numer}. See also the recent work \cite{LiWerner:2018} for
the order optimality of risk estimators for a class of (ordered) filter based regularization methods, where the numerical
challenges are also highlighted \cite[Fig. 1]{LiWerner:2018}.

GCV and UPRE estimate the minimum value of the predictive risk $p_\alpha(\mathbf{g}^\eta)$ using the noisy
sample $\mathbf{g}^\eta$. The accuracy of these estimators depends strongly on the sample size, and
when the size of $\mathbf{g}^\eta$ is small, the estimation can be unsatisfactory. To illustrate this
point, in Fig. \ref{motivation}(a), we show the predictive risk, its bias-variance decomposition
and the optimal $\alpha$ (with respect to predictive risk), for \textit{shaw} from Regularization tools.
In Fig. \ref{motivation}(b), we show the GCV and  UPRE estimators (including their
global minimizers).  Clearly, both curves do not have a unique minimizer, and more importantly,
the global minimizers are smaller than the optimal one by several orders of magnitude.
Consequently, the corresponding reconstructions by GCV and UPRE are hugely corrupted by noise amplification
and completely useless. This kind of behavior can be observed in the majority of test problems
available in the package Regularization tools and actually for each problem therein, its presence has a
non-negligible significant percentage of randomly generated samples $\mathbf{g}^\eta$.

\begin{figure}
\begin{minipage}{0.49\textwidth}
\includegraphics[width=0.91\textwidth]{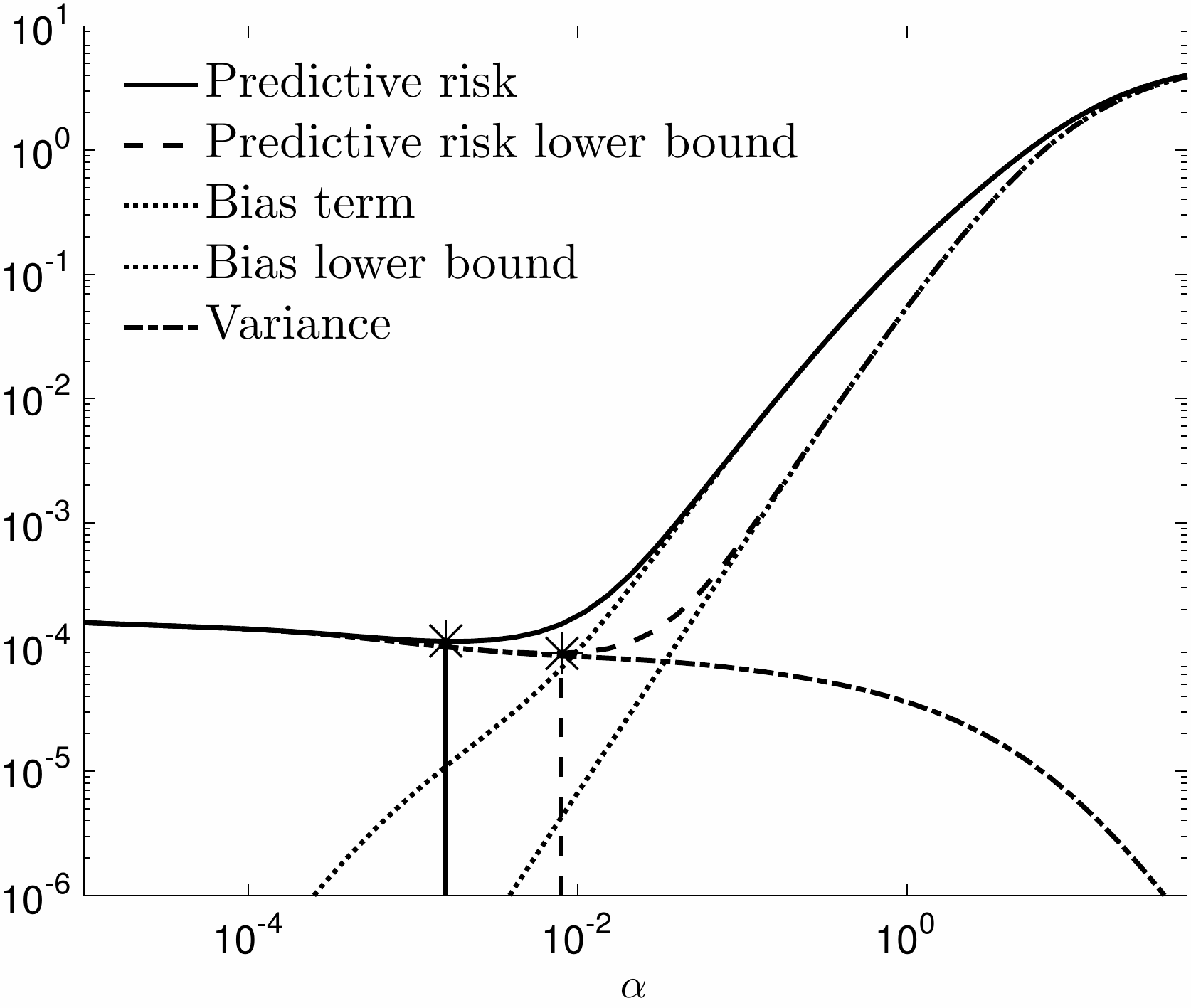}
\end{minipage}
\begin{minipage}{0.49\textwidth}
\includegraphics[width=.99\textwidth]{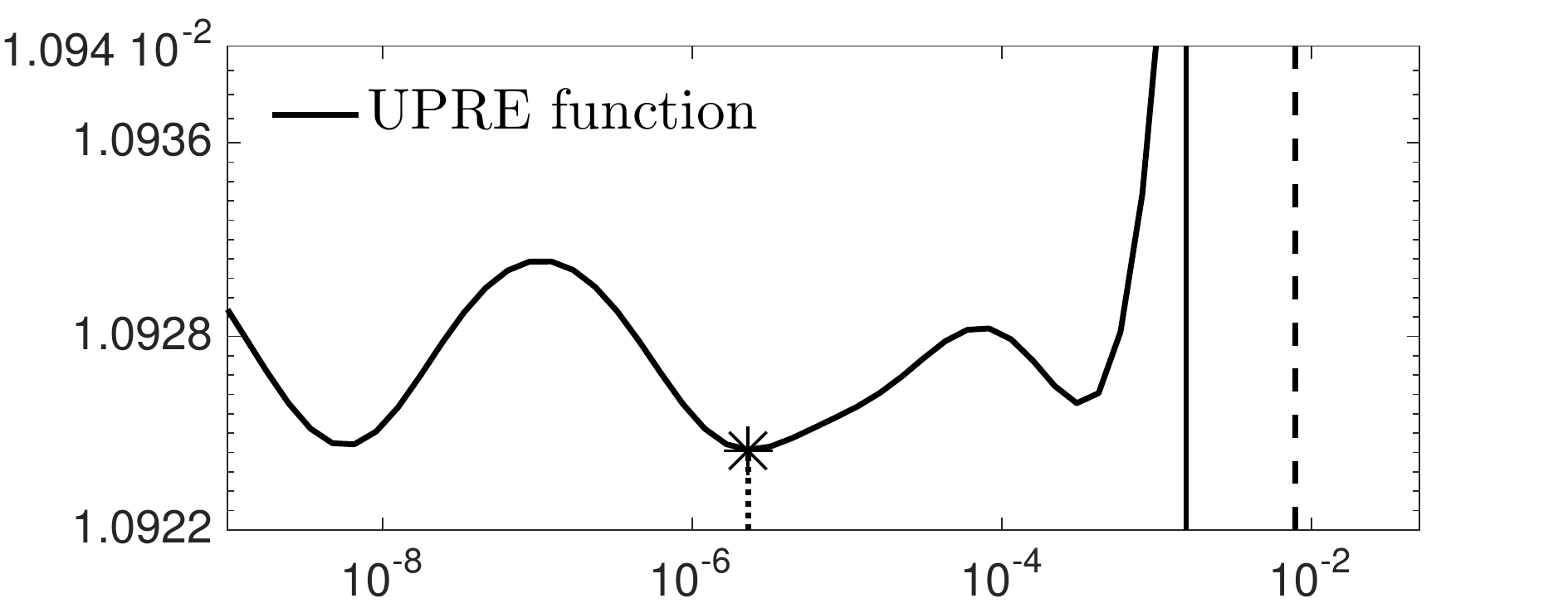}
\includegraphics[width=.99\textwidth]{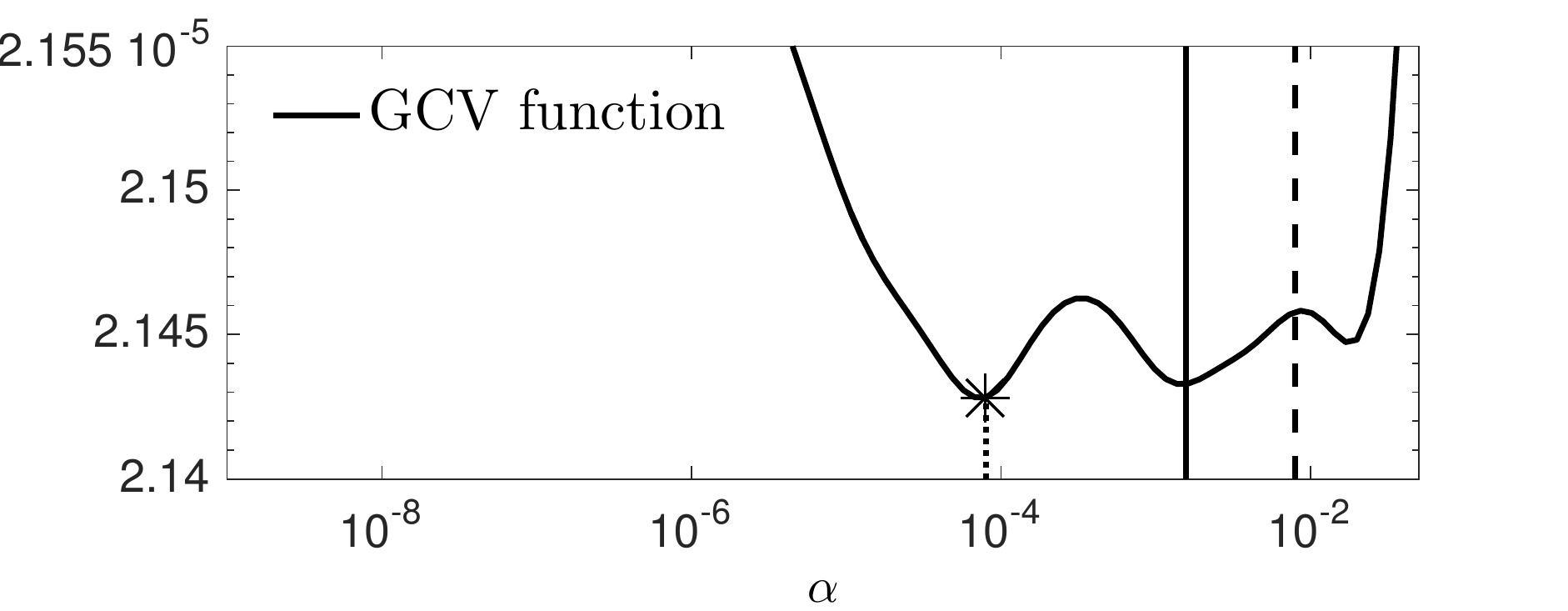}
\end{minipage}
\caption{Behavior of the predictive risk and the lower bound for
\textit{shaw}.\label{motivation}}
\end{figure}

In this work we present a novel parameter choice strategy inspired by predictive risk, which aims
at overcoming the aforementioned drawback of lacking robustness. It gives
an upper bound of the optimal parameter, and thus can provide solutions sufficiently accurate yet
very stable with respect to sample variation. The rule is based on the following simple
observation:
\begin{equation*}
 s_n(I-X_\alpha)^2\|\mathbf{g}^\dag\|^2\leq \|(X_\alpha-I)\mathbf{g}^\dag\|^2\leq s_1(I-X_\alpha)^2\|\mathbf{g}^\dag\|^2,
\end{equation*}
where $s_n(I-X_\alpha)$ and $s_1(I-X_\alpha)$ denote the minimum and maximum singular value
of the matrix $I-X_\alpha$, respectively. Then given the norm of
the true data $\mathbf{g}^\dag$, i.e., $\|\mathbf{g}^\dag\| = \rho$, we can bound the predictive risk $p_\alpha(\mathbf{g}^\eta)$
both from above and below accordingly using the preceding bounds. It turns out that the upper bound
is not useful since it can trivially reduce to a constant. Surprisingly, graphically, the lower bound has the
same shape as the predictive risk and possesses a minimizer close to the optimal one. In
Fig. \ref{motivation}(a), we show the minimizers of the predictive risk (left
one) and lower bound (right one), and the lower bound of the bias term, which empirically motivates
the use of lower bound for parameter choice. Thus, under the hypothesis $\|\mathbf{g}^\dag\|=\rho$, we take the minimum of the
predictive risk
\begin{equation}\label{eq:lower_bound}
T_\rho(\alpha) = \min_{\|\mathbf{g}^\dag\| = \rho} p_\alpha (\mathbf{g}^\eta),
\end{equation}
and then seek the minimizer of $T_\rho(\alpha)$ along $\alpha$ :
\begin{equation}\label{eq:orp}
\alpha^* =\arg\min_{\alpha} T_\rho(\alpha),
\end{equation}
which approximates the optimal parameter and thus can be used as a choice rule.
Note that the function $T_\rho(\alpha)$ depends on the variance $\sigma^2$ of the noise $\eta$ and the (squared)
data norm $\rho^2$. We can rewrite the lower bound of the predictive risk as
\begin{equation}\label{eq:sb}
T_{({\rho^2},{\sigma^2})}(\alpha) =   \rho^2  s_{n}(X_\alpha - I)^{2} + \sigma^2 \|X_\alpha\|_F^2 .
\end{equation}
In the construction of the approximation \eqref{eq:sb}, the variance term remains intact, but the bias term is
subsumed by a lower bound. Since the bias term is increasing and the variance one is decreasing, $\alpha^*$
is an upper bound of the ``optimal'' regularization parameter, which ensures the non-overfitting of the choice
rule; see Proposition \ref{prop:upper} for a proof in the case of Tikhonov regularization. It is worth noting
that problem \eqref{eq:sb} has a nice variational structure similar to the standard Tikhonov regularization
itself: the first term can be viewed as fidelity (with respect to the identity matrix $I$), and the second term
is a penalty, controlling the size). This variational structure lends itself to rigorous analysis.

In \eqref{eq:sb}, we give the explicit dependency of $T_{({\rho^2},{\sigma^2})}(\alpha)$ on the noise variance $\sigma^2$.
Clearly, its minimizer(s) $\alpha^*$ depends only on the SNR $\rho^2/\sigma^2$. Thus, the choice rule
determines a proper level of regularization according to the SNR. Nevertheless, the latter information is not always
available in practice. In Section \ref{sec:pro}, we discuss three different strategies for minimizing the function
$T_{({\rho^2},{\sigma^2})}(\alpha)$ according to the available SNR information.

\section{Optimization strategies}\label{sec:pro}
Now we apply the lower bound $T_{({\rho^2},{\sigma^2})}(\alpha)$ given in \eqref{eq:sb} for parameter choice.
The minimization of $T_{({\rho^2},{\sigma^2})}(\alpha)$ is possible only if both $\rho$ and $\sigma$ are known.
When some of this information is unavailable, as is often the case in practice, we propose to
approximate $T_{({\rho^2},{\sigma^2})}(\alpha)$ (and accordingly the minimizer $\alpha^*$). We describe three
optimization approaches according to the level of available information. The second and third
cases are more realistic in practical applications.

\subsection{$\rho$ and $\sigma$ known}
First, we consider the case of known $\rho$ and $\sigma$. This is the least likely case in practice,
since usually the true data norm $\|\mathbf{g}^\dag\|$ is not known \textit{a priori}, but theoretically it is the easiest
case. Then the function $T_{(\rho^2,\sigma^2)}$ applies directly. Note that minimizers of $T_{(\rho^2,\sigma^2)}$ only depend
on the SNR $\rho^2/\sigma^2$ or equivalently
\begin{equation}\label{eq:snr}
\xi := 10 \log_{10} \left( \frac{\rho^2}{n\sigma^2} \right),
\end{equation}
which is generally known as SNR in decibel (dB), according to ISO 15739:2017, and will also
be called SNR below, by slightly abusing the terminology. Note that given SNR $\xi$, the minimizer does
not depend on noisy data $\mathbf{g}^\eta$ and the choice rule is of {\it a priori} nature \cite{Engl}. For Tikhonov regularization, it
can be shown that the rule is indeed well defined, and the function $T_{(\rho^2,\sigma^2)}$ is relatively well behaved and
numerically tractable; see Section \ref{sec:Tikh} below for further details.

\subsection{$\rho$ unknown and $\sigma$ known}

In many applications, the SNR $\xi$ is unknown and only an estimate of $\sigma^2$ is available. This happens when
the measurement is acquired by an industrial device, and by collecting multiple repetitive measurements of a
zero excitation signal, the read-out noise allows estimating the noise level $\sigma^2$ (e.g., via maximum likelihood).
This is a standard calibration procedure for providing the noise level (measurement precision) $\sigma^2$ of a
device in industry. Using the following unbiased estimator of $\rho^2$
\begin{equation}\label{eq:ume}
\hat{\rho^2} = \| \mathbf{g}^\eta \|^2 - n \sigma^2,
\end{equation}
we can define an $M$-estimator of $\alpha^*$. Specifically, we take the minimizer
\begin{equation}\label{eq:ue}
\alpha^* := \arg \min_{\alpha} \{U^\eta(\alpha):=T_{({\hat{\rho^2}},{\sigma^2})}(\alpha)\}
\end{equation}
where $U^\eta(\alpha)$ is an unbiased estimator of $T_{({\rho^2},{\sigma^2})}(\alpha)$, i.e.
$\mathbb{E}_{\boldsymbol{\eta}}[U^\eta(\alpha)] = T_{({\rho^2},{\sigma^2})}(\alpha).$
Upon slightly abusing terminology, we also refer it to as \texttt{PRO}.

\subsection{$\rho$ and $\sigma$ unknown}\label{blind-section}
When both $\rho$ and $\sigma$ are unknown, \texttt{PRO} cannot be applied directly, and we propose an alternating estimating/minimization strategy. The procedure is summarized in the \texttt{I-PRO} algorithm below.

Formally, the \texttt{I-PRO} algorithm alternates between estimating the SNR (given $\alpha$) and
estimating $\alpha$ (given the SNR), in a manner similar to the classical EM algorithm or
coordinate ascent in variational Bayes. We term it as iterative predictive risk optimization
(\texttt{I-PRO}). Surprisingly, this simple procedure can provide excellent estimates of the
SNR for ill-posed problems, and the obtained solutions are often very close to the ones
corresponding to known $\rho$ and $\sigma$, as confirmed by extensive simulation results in Section \ref{sec:numer}.

\begin{algorithm}
\caption{\texttt{I-PRO}: Iterative Predictive Risk Optimization.\label{alg:ipro}}
\begin{algorithmic}[1]
\STATE{Fix $\epsilon = 10^{-16}$, $\alpha_0\neq0$, $\alpha_1=0$, $k=1$}
\WHILE{$ | \alpha_k - \alpha_{k-1} | > \epsilon \alpha_k $}
\STATE{Update the variance according to formula
\begin{equation}\label{eq:se}
{\sigma^2_{k}}(\alpha) = \frac{1}{n} \|\mathbf{r}^\eta_{\alpha_{k}}\|^2
\end{equation}
where $\mathbf{r}^\eta_\alpha := A \mathbf{f}^\eta_{\alpha} - \mathbf{g}^\eta$ is the residual.}
\STATE{Update the signal norm using
\begin{equation}\label{eq:le}
{\rho^2_{k}}(\alpha_{k}) = \| \mathbf{g}^\eta \|^2 - \|\mathbf{r}^\eta_{a_{k}}\|^2.
\end{equation}}
\STATE\label{line3}{Update $\alpha_k$ by computing minimizer of equation \eqref{eq:sb} given $\rho^2_k$ and $\sigma_k^2$, i.e.,
\begin{equation}
\label{eq:argmink}
\alpha_{k+1} := \arg\min_{\alpha} T_{({\rho_k^2},{\sigma_k^2})}(\alpha) .
\end{equation}}
\ENDWHILE
\RETURN $\alpha_{k}$
\end{algorithmic}
\end{algorithm}

Computationally, \texttt{PRO} is based on two invariants, i.e., the smallest singular value $s_n(X_\alpha-I)$ of
$X_\alpha-I$ and trace of $X_\alpha^T X_\alpha$. They can be both computed by means of SVD of the matrix $A$, whereas $X_\alpha$ can
be written in terms of SVD of $A$, which is very convenient when the problem size is small. Otherwise, they can be
computed without the SVD of $A$, using randomized algorithms; see Section \ref{section:svd-free} for further details.

\section{Tikhonov regularization}\label{sec:Tikh}

Now we consider \texttt{PRO} and \texttt{I-PRO} for standard Tikhonov regularization.
The corresponding influence matrix $X_\alpha\in\mathbb{R}^{n\times n}$ is given by
\begin{equation*}
X_\alpha = A (A^*A+\alpha I) ^{-1}A^*.
\end{equation*}
Clearly, $X_\alpha$ is symmetric and positive semidefinite, with its singular values $s_n(X_\alpha)\leq\ldots s_1(X_\alpha)
< 1$. Let $s_1 \geq ... \geq s_r>0=s_{r+1}=\ldots=s_{\min(m,n)}$ be the
singular values of the matrix $A$, with $r$ being the rank of $A$. Then the singular values $s_i(X_\alpha)$ of $X_\alpha$ are given by
\begin{equation}\label{eq:op_norm}
s_i(X_\alpha) = \left\{\begin{array}{ll}
         \displaystyle \frac{s_i^2}{s_i^2+\alpha},& i = 1,\ldots,r,\\
           0 , & i = r+1,\ldots, n.\\
           \end{array}\right.
\end{equation}
Since $s_{n} (X_\alpha -I) = 1 - s_{1}(X_\alpha) >  0$, a lower bound of the predictive risk $p_\alpha(\mathbf{g}^\eta)$ is given by
\begin{equation}\label{eq:tm}
T_{(\rho^2,\sigma^2)}(\alpha) =   \rho^2 \frac{\alpha^2}{(s_1^2+\alpha)^2} + \sigma^2 \sum_{i=1}^r \frac{s_i^4}{(s_i^2+\alpha)^2}.
\end{equation}
The explicit form of $T_{(\rho^2,\sigma^2)}$ facilitates the analytical study of \texttt{PRO}. It is convenient to introduce
two auxiliary functions. We denote the first and the second terms (involving $\alpha$) of equation \eqref{eq:tm} by
$f_1(\alpha)$ and $f_2(\alpha)$, respectively, i.e.,
\begin{equation*}
  f_1 (\alpha) = \frac{\alpha^2}{(\alpha+s_1^2)^2} \quad\mbox{and}\quad f_2(\alpha) = \sum_{i=1}^r\frac{s_i^4}{(\alpha+s_i^2)^2}.
\end{equation*}
Next, for any fixed $h=\sigma^2/\rho^2$, we define a parameterized function
$$
T_h(\alpha) = f_1(\alpha)+hf_2(\alpha)\quad\mbox{and}\quad \alpha^* = \arg\min_{\alpha\in[0,s_1^2/2]} T_h(\alpha).
$$
The re-parameterization does not change the minimizer $\alpha^*$, but is more convenient for the analysis.
The following identities hold for the function $T_h(\alpha)$ :
\begin{equation*}
  \lim_{\alpha \to0^+} T_h(\alpha) = rh\quad \mbox{and}\quad \lim_{\alpha\to\infty} T_h(\alpha)=1.
\end{equation*}

\subsection{The {\tt PRO} function $T_h(\alpha)$}
The next lemma gives the convexity and monotonicity of $f_1(\alpha)$ and $f_2(\alpha)$.
\begin{lemma}\label{lem:f}
The function $f_1(\alpha)$ is convex over $[0,s_1^2/2]$ and monotonically increasing, and $f_2(\alpha)$ is convex and
monotonically decreasing over $[0,\infty)$.
\end{lemma}
\begin{proof}
Direct computation gives
\begin{eqnarray*}
  f'_1(\alpha) & = \frac{2\alpha}{(\alpha+s_1^2)^2}-\frac{2\alpha^2}{(\alpha+s_1^2)^3}=\frac{2\alpha s_1^2}{(\alpha+s_1^2)^3}>0,\\
  f''_1(\alpha) & = \frac{2s_1^2(\alpha+s_1^2)-6\alpha s_1^2}{(\alpha+s_1^2)^4} = \frac{2s_1^2(s_1^2-2\alpha)}{(\alpha+s_1^2)^4}.
\end{eqnarray*}
Thus, $f_1$ is convex over the interval $(0,s_1^2/2]$. Further,
\begin{equation*}
  f_2'(\alpha) = -\sum_{i=1}^r\frac{2s_i^4}{(\alpha+s_i^2)^3}<0,\quad f''_2(\alpha) = \sum_{i=1}^r\frac{6s_i^4}{(\alpha+s_i^2)^4}.
\end{equation*}
This completes the proof of the lemma.
\end{proof}

\begin{lemma}\label{lem:al}
For any minimizer $\alpha^*(h) \in [0,s_1^2/2]$ to the function $T_h(\alpha)$, it is strictly
monotonically increasing in $h$.
\end{lemma}
\begin{proof}
We prove the assertion by the implicit function theorem. The existence of a unique
$\alpha^*\equiv \alpha^*(h)\in [0,s_1^2/2]$ follows from the strict convexity of
$T_h$ over the interval, cf Lemma \ref{lem:f}. First, we claim that $\alpha^*(h)$
is not $\alpha=0$. Indeed, it follows from straightforward computation that
\begin{equation}\label{eqn:deriv-Th}
  T_h'(\alpha)=\frac{2\alpha s_1^2}{(\alpha+s_1^2)^3}-h \sum_{i=1}^r\frac{2s_i^4}{(\alpha+s_i^2)^3}
\end{equation}
and thus
\begin{equation*}
\lim_{\alpha\to 0^+}T_h'(\alpha)  = -2h\sum_{i=1}^rs_i^{-2}<0.
\end{equation*}
Thus, $T_h$ is differentiable at $\alpha^*(h)$, and the optimal $\alpha^*\equiv \alpha^*(h)$ satisfies the optimality condition
\begin{equation*}
  f'_1(\alpha^*)+hf_2'(\alpha^*) = 0.
\end{equation*}
Then by the chain rule, we have
\begin{equation*}
  f''_1(\alpha^*)\frac{d\alpha^*}{dh} + f_2'(\alpha^*) + hf''_2(\alpha^*)\frac{d\alpha^*}{dh} = 0,
\end{equation*}
i.e.,
\begin{equation*}
  \frac{d\alpha^*}{dh} = \frac{-f_2'(\alpha^*)}{f_1''(\alpha^*)+hf''_2(\alpha^*)}
\end{equation*}
Then in view of Lemma \ref{lem:f}, the denominator is strictly positive for any $\alpha^*>0$,
from which the desired assertion follows directly.
\end{proof}

The next proposition summarizes some properties of the minimizer to $T_h(\alpha)$.
\begin{proposition}\label{lemma}
Let $r$ be the rank of the matrix $A$. Then the following statements hold for the function $T_h(\alpha)$.
\begin{itemize}
\item[(i)] Over the interval $[0,s_1^2/2]$, the function $T_h$ has a unique minimizer $\alpha^*$.
\item[(ii)] If $s_1>0$ and $h\leq (27r)^{-1}$, the
function $T_h(\alpha)$  admits a unique global minimizer $\alpha^*$ in $(0,s_1^2/2]$.
\item[(iii)] For $h\leq \zeta : = s_1^2/\mathrm{tr}(A^*A)$, the minimizer $\alpha^*\equiv \alpha^*(h)$ satisfies
\begin{equation*}
  s_1^2 h \leq  \alpha^* \leq (1-(h/\zeta)^\frac13)^{-1}s_1^2(h/\zeta)^\frac13.
\end{equation*}
\end{itemize}
\end{proposition}
\begin{proof}
Part (i) is already shown in Lemma \ref{lem:al}.

\medskip
\noindent Part (ii).
By Lemma \ref{lem:f}, the function $T_h(\alpha)$ is differentiable and strictly convex in
the interval $\alpha \in [0,s_1^2 / 2]$. Furhter, $T_h$ is increasing in the interval $\alpha \in (s_1^2 / 2,\infty)$ if $h<(27r)^{-1}$.
Indeed, in view of \eqref{eqn:deriv-Th},
\begin{eqnarray*}
  T_h'(\alpha) & \geq \frac{2\alpha s_1^2}{(\alpha+s_1^2)^3}-h\sum_{i=1}^r\frac{2s_1^2}{(\alpha+s_i^2)^2}\\
               & \geq \frac{2\alpha s_1^2}{(\alpha+s_1^2)^3}-\frac{2hrs_1^2}{\alpha^2}=\frac{2s_1^2}{\alpha^2}\Big(\frac{\alpha^3}{(\alpha+s_1^2)^3}-hr\Big).
\end{eqnarray*}
Clearly $\frac{\alpha^3}{(\alpha+s_1^2)^3}$ is an increasing in $\alpha$, and its minimum over $[s_1^2/2,+\infty)$
is achieved at $\alpha=s_1^2/2$ with a minimum value $1/27$. Therefore under the condition $h\leq (27r)^{-1}$, there
exists a unique global minimizer to $T_h(\alpha)$ in $[0,+\infty)$ and it is located within the interval $[0,s_1^2 / 2]$.
Moreover, as argued in Lemma \ref{lem:al} that if $h>0$, $\alpha=0$ cannot be a  minimizer. This shows part (ii).

\medskip
\noindent Part (iii). Clearly, by Lemma \ref{lem:al}, as $h$ tends to zero monotonically,
the minimizer $\alpha^*(h)$ also decreases monotonically. Further, the optimality
condition $\frac{d}{d\alpha} T_h(\alpha^*) = 0$ implies
\begin{equation}\label{eq:alpha-tikh}
\frac{\alpha^* s^2_1}{(s_1^2+\alpha^*)^3} \left( \sum_{i=1}^r \frac{s^4_i}{(s_i^2+\alpha^*)^3} \right)^{-1}  = h.
\end{equation}
Next we bound the quantity on the left hand side. Since $\sum_{i=1}^r \frac{s^4_i}{(s_i^2+\alpha^*)^3}\geq \frac{s_1^4}{(\alpha^*+s_1^2)^3}$, we deduce
\begin{equation*}
  h \leq \frac{\alpha^* s_1^2}{(s_1^2+\alpha^*)^3}\frac{(\alpha^*+s_1^2)^3}{s_1^4} = \frac{\alpha^*}{s_1^2}
\end{equation*}
This shows the first inequality. Meanwhile, from the inequality
\begin{equation*}
\sum_{i=1}^r \frac{s^4_i}{(s_i^2+\alpha^*)^3}\leq \sum_{i=1}^r\frac{s_i^2}{(\alpha^*)^2}=\frac{\mathrm{tr}(A^*A)}{(\alpha^*)^2},
\end{equation*}
we obtain
\begin{equation*}
   h \geq \frac{\alpha^* s_1^2}{(s_1^2+\alpha^*)^3}\frac{(\alpha^*)^2}{\mathrm{tr}(A^*A)}=\frac{s_1^2}{\mathrm{tr}(A^*A)}\frac{(\alpha^*)^3}{(s_1^2+\alpha^*)^3},
\end{equation*}
i.e., $\frac{\alpha^*}{\alpha^*+s_1^2} \leq (\mathrm{tr}(A^*A)h/s_1^2)^\frac{1}{3}$. Solving for the inequality gives the assertion in part (iii).
\end{proof}

\begin{remark}
According to Theorem \ref{lemma}, the function $T_h$ always has a finite positive and unique minimizer within the interval $[0,s_1^2/2]$,
which is also the unique global minimizer over $[0,\infty)$ if the SNR is sufficiently large. In practice, minimizing over $[0,s_1^2/2]$
is sufficient for Tikhonov regularization. Since the function $T_h(\alpha)$ is strictly convex in $\alpha\in [0,s_1^2/2]$, in principle,
 its minimization is numerically tractable. For example, one may apply Newton type methods, which is
guaranteed to converge globally \cite{Polyak:2018}. Thus it is numerically more amenable than choice rules based
on predictive risk, e.g., UPRE and GCV, which are known to suffer from (bad) local minima as well as flatness near global minima;
see Fig. \ref{motivation}.
\end{remark}

Next we turn to the upper bound property of $\alpha^*$.
\begin{lemma}\label{lem:al0}
For any $\mathbf{g}^\dag$, there exists an $\alpha_0>0$, such that there holds for all $\alpha\leq \alpha_0$
\begin{equation*}
  \sum_{i=1}^r \frac{2\alpha s_i^2}{(s_i^2+\alpha)^3}(\mathbf{g}^\dag,\mathbf{u}_i)^2 \geq \frac{2\alpha s_1^2}{(s_1^2+\alpha)^3}\|\mathbf{g}^\dag\|^2,
\end{equation*}
where $\mathbf{u}_i$ denotes the $i$th left singular vector of the matrix $A$.
\end{lemma}
\begin{proof}
By the identity $\sum_{i=1}^r (\mathbf{g}^\dag,\mathbf{u}_i)^2
=\|\mathbf{g}^\dag\|^2=\rho^2$, the assertion is equivalent to
\begin{equation*}
  \sum_{i=1}^r \frac{2\alpha s_i^2}{(s_i^2+\alpha)^3}(\mathbf{g}^\dag,\mathbf{u}_i)^2 \geq \sum_{i=1}^r \frac{2\alpha s_1^2}{(s_1^2+\alpha)^3}(\mathbf{g}^\dag,\mathbf{u}_i)^2.
\end{equation*}
Now we claim that for all sufficiently small $\alpha$, there holds
\begin{equation*}
   \frac{s_i^2}{(s_i^2+\alpha)^3} \geq \frac{s_1^2}{(s_1^2+\alpha)^3},\quad\mbox{i.e.}\quad
  \frac{s_1^2+\alpha}{s_i^2+\alpha} \leq \Big(\frac{s_1^2}{s_i^2}\Big)^\frac13,
\end{equation*}
which, with $\lambda=(\frac{s_1^2}{s_i^2})^\frac13$, is equivalent to
$
\alpha \leq s_1^2\frac{1-\lambda^{-2}}{\lambda-1}=s_1^2\lambda^{-1}(1+\lambda^{-1}).
$
Thus, it suffices to choose $\alpha_0=s_1^\frac{4}{3}s_n^{\frac{2}{3}}$.
\end{proof}
\begin{remark}
The bound $\alpha_0$ given in Lemma \ref{lem:al0} is very loose. In practice, the coefficients $(\mathbf{g}^\dag,\mathbf{u}_i)^2$
can decay very rapidly to zero as $i$ increases, especially for severely ill-posed problems, due to smoothing property of $A$
(and regularity condition on $\mathbf{f}^\dag$, e.g., source conditions \cite{Engl,ItoJin:2015}). Then, the
first few terms in the summation are dominating, and one expects a much larger bound. In practice, it is often
$O(1)$.
\end{remark}
\begin{proposition}\label{prop:upper}
If the minimizer $\alpha^*$ to the function $T_h$ is sufficiently small, then it is an upper bound of a local minimizer
to the predictive risk $p_\alpha(\mathbf{g}^\eta)$.
\end{proposition}
\begin{proof}
Straightforward computation shows that the predictive risk $p_\alpha(\mathbf{g}^\eta)$ is given by
\begin{equation*}
   p_\alpha(\mathbf{g}^\eta) = \sum_{i=1}^r \frac{\alpha^2}{(s_i^2+\alpha)^2}(\mathbf{g}^\dag,\mathbf{u}_i)^2 + \sigma^2\sum_{i=1}^r \frac{s_i^4}{(s_i^2+\alpha)^2},
\end{equation*}
It suffices to show that $\frac{d}{d\alpha}p_\alpha(\mathbf{g}^\eta)|_{\alpha=\alpha^*}>0$. Note that $\frac{d}{d\alpha}p_\alpha(\mathbf{g}^\eta)|_{\alpha=\alpha^*}$ is given by
\begin{equation*}
  \sum_{i=1}^r \frac{2\alpha^* s_i^2}{(s_i^2+\alpha^*)^3}(\mathbf{g}^\dag,\mathbf{u}_i)^2 - \sigma^2\sum_{i=1}^r \frac{2s_i^4}{(s_i^2+\alpha^*)^3}.
\end{equation*}
Since the minimizer $\alpha^*$ of $T_{(\rho^2,\sigma^2)}$ is sufficiently small,
by Lemma \ref{lem:al0}, $\frac{d}{d\alpha}p_\alpha(\mathbf{g}^\eta)|_{\alpha=\alpha^*}>0$.
This and the monotonicity of the sums in $p_\alpha(\mathbf{g}^\eta)$ complete the proof.
\end{proof}

\subsection{Convergence of the iterative scheme}

Now we analyze the convergence of the \texttt{I-PRO} algorithm for Tikhonov regularization. First, by viewing the
estimates $\rho$ and $\sigma$ as functions of $\alpha$, we have the following monotonicity result.

\begin{lemma}\label{lem:mon-sig}
The functions $\rho^2(\alpha)$ and $\sigma^2(\alpha)$ are monotonically decreasing and increasing, respectively, in $\alpha$.
\end{lemma}
\begin{proof}
The variance estimate $\sigma^2(\alpha)$ is proportional to $\| A\mathbf{f}^\eta_\alpha-\mathbf{g}^\eta\|^2$. For any $\alpha_0,\alpha_1>0$,
by the minimizing property of $\mathbf{f}^\eta_{\alpha_0}$ and $\mathbf{f}^\eta_{\alpha_1}$, i.e.,
\begin{eqnarray*}
   \|A\mathbf{f}^\eta_{\alpha_0}-\mathbf{g}^\eta\|^2 + \alpha_0\|\mathbf{f}_{\alpha_0}^\eta\|^2 & \leq \|A\mathbf{f}^\eta_{\alpha_1}-\mathbf{g}^\eta\|^2 + \alpha_0\|\mathbf{f}_{\alpha_1}^\eta\|^2,\\
   \|A\mathbf{f}^\eta_{\alpha_1}-\mathbf{g}^\eta\|^2 + \alpha_1\|\mathbf{f}_{\alpha_1}^\eta\|^2 & \leq \|A\mathbf{f}^\eta_{\alpha_0}-\mathbf{g}^\eta\|^2 + \alpha_1\|\mathbf{f}_{\alpha_0}^\eta\|^2,
\end{eqnarray*}
we deduce that the residual $\|A\mathbf{f}^\eta_\alpha-\mathbf{g}^\eta\| $ is monotonically increasing in
$\alpha$. Then the assertion follows immediately.
\end{proof}

Now we can prove that the sequence $\{\alpha_k\}_{k=1}^\infty$ of regularization parameters generated by
the \texttt{I-PRO} algorithm is actually monotone. Once it is bounded, the result provides a constructive
proof of the existence of a fixed point.
\begin{theorem}
For any $\alpha_0\in[0,s_1^2/2]$, the sequence $\{\alpha_k\}$ generated by the {\tt I-PRO} algorithm is monotone.
\end{theorem}
\begin{proof}
If $\alpha_0<\alpha_1$, then by the monotonicity in Lemma \ref{lem:mon-sig}, $h(\alpha)=\sigma^2(\alpha)/\rho^2(\alpha)$ is
monotonically increasing, and thus $h(\alpha_0)<h(\alpha_1)$. Then by Lemma \ref{lem:al}, $\alpha^*_{h(\alpha_0)}
<\alpha^*_{h(\alpha_1)}$, i.e., $\alpha_1<\alpha_2$. The case $\alpha_0>\alpha_1$ follows similarly. This
shows the monotonicity of the iterates $\{\alpha_k\}_{k=1}^\infty$.
\end{proof}

Note that by the optimality condition to $T_h$, at the fixed point $\alpha^*$ satisfies
\begin{equation}\label{eqn:fixed-point}
   \frac{\alpha s_1^2}{(\alpha+s_1)^3} - \frac{\sigma^2(\alpha)}{\rho^2(\alpha)}\sum_{i=1}^r\frac{s_i^4}{(\alpha+s_i^2)^3} = 0.
\end{equation}
Note that this equation can be rewritten into a (high-degree) polynomial in $\alpha$, and thus the existence
of a root is ensured. Further, by Lemma \ref{lem:mon-sig}, $\|A\mathbf{f}_\alpha^\eta-\mathbf{g}^\eta\|$ decreases monotonically to the least-squares
residual as $\alpha$ tends to zero, and thus the limit $\lim_{\alpha\to0^+}\|A\mathbf{f}_\alpha^\eta-\mathbf{g}^\eta\|>0$ makes sense.
\begin{proposition}\label{prop:ipro}
The following statements hold for the {\tt I-PRO} algorithm.
\begin{itemize}
 \item[(i)]If $\lim_{\alpha\to0^+}\|A\mathbf{f}_\alpha^\eta - \mathbf{g}^\eta\|>0$, then $0$ is not a fixed point of the {\tt I-PRO} algorithm.
 \item[(ii)] If the estimate $\sigma^2(\alpha^*)/\rho(\alpha^*)$ is smaller (greater) than the exact one, then {\tt I-PRO} estimate $\alpha^*$ is smaller (greater) than the {\tt PRO} estimate.
\end{itemize}
\end{proposition}
\begin{proof}
Assertion (i) is direct from \eqref{eqn:fixed-point} and an argument by contradiction, and (ii) follows from Lemma \ref{lem:al}.
\end{proof}

Proposition \ref{prop:ipro}(ii) suggests that one may monitor the quantity $\sigma^2(\alpha_k)/\rho^2(\alpha_k)$
during the iteration as an \textit{a posteriori} check: if it is deemed to be too small, then the iteration should be terminated.

\subsection{SVD-free implementation}\label{section:svd-free}

For large-scale problems, SVD is usually too costly to apply (randomized SVD may be applied instead, when the problem has low-rank structure
\cite{ItoJin:2019}). Instead we implement \texttt{PRO} and \texttt{I-PRO} as follows. It follows from \eqref{eq:op_norm} that
the first term in the function $T_{(\rho^2,\sigma^2)}$ depends on the smallest eigenvalue of the influence matrix $I-X_\alpha$, or the
largest eigenvalue $\lambda_1(A^*A)$ of $A^*A$. Thus, we employ classical algorithms for computing $\lambda_1(A^*A)$, e.g.,
power method, and take
\begin{equation}\label{eq:op_norm_large_scale}
  s_n(X_\alpha - I )^2 = \left(\frac{\alpha}{\lambda_1(A^*A) + \alpha}\right)^2.
\end{equation}
To compute the term $\|X_\alpha\|_F^2$, we employ a randomized trace estimator \cite{hutchinson1990stochastic}:
\begin{eqnarray*}
\|X_\alpha\|^2_{F} &= {\rm tr}( X^*_\alpha X_\alpha ) =  \mathbb{E}_{\mathbf{z}}  [\| X_\alpha \mathbf{z} \|^2] \\
   &\simeq \sum_{i=1}^{p}  \|A (A^*A + \alpha I)^{-1} A^*\mathbf{z}_{(i)} \|^2,
\end{eqnarray*}
where $\mathbf{z}$ follows the standard Gaussian distribution and $\{\mathbf{z}_{(i)}\}_{i=1}^p$
are $p$ i.i.d. samples of $Z$ (other choices are also possible). Each summand involves solving one
linear system: $(A^*A + \alpha I)\mathbf{w}_{(i)} =  A^* \mathbf{z}_{(i)}.$ Once $\mathbf{w}_{(i)}$ are
obtained, we have
\begin{equation}\label{eq:fr_norm_large_scale}
\|X_\alpha\|^2_{F}  \simeq \sum_{i=1}^{p}  \|A\mathbf{w}_{(i)} \|^2.
\end{equation}
In practice, only a few samples are required for an accurate estimation (see \cite{Avron:2011} for relevant
error bounds on randomized trace estimators). The randomized trace estimation is widely used in implementing GCV. The difference lies in the fact
that for \texttt{PRO}, it is applied to the matrix $X_\alpha^*X_\alpha$, instead of $X_\alpha$ in GCV.

\section{Numerical simulations and discussions}\label{sec:numer}
Now we present two sets of simulation studies with synthetic data from popular public software
packages, i.e., Regularization tools and AIR tools. The first aims to show the
robustness of the proposed method for problems with small and moderately sized samples, and
the second to show its performance on large-scale problems.

\subsection{Small and moderately sized samples}
Below we compare the Tikhonov solutions by \texttt{PRO} and \texttt{I-PRO} with that by existing choice rules
in two scenarios: (i) \texttt{PRO} versus DP, BP and UPRE, if $\sigma$ is known; (ii) \texttt{I-PRO}
versus LC, GCV and QOC, if $\sigma$ is unknown. DP is implemented by solving for $\alpha$ in
\begin{equation*}
  \|\mathbf{A}\mathbf{f}_\alpha^\eta-\mathbf{g}^\eta\| = \sqrt{n}\sigma .
\end{equation*}
The implementation of BP follows \cite[equation (1.7)]{HamarikRuas:2009} with the constants $\gamma=1/4$ and $c$ given in \cite[Table 1,
p. 961]{HamarikRuas:2009}. All the rules are computed on an equally distributed grid on a logarithmical scale.
We evaluate them on eight inverse problems (i.e., \textit{baart}, \textit{deriv2}, \textit{foxgood}, \textit{gravity}, \textit{heat},
\textit{i\_laplace}, \textit{phillips}, \textit{shaw}) from the public package Regularization tools (downloaded from
\url{http://people.compute.dtu.dk/pcha/Regutools/index.html}, accessed on July 1, 2019). They are all
one-dimensional linear integral equations of first kind, with different degree of ill-posedness. Each problem refers to
one example, except \textit{heat} and \textit{i\_laplace}. The \textit{heat} family depends on one parameter controlling its conditioning,
and we choose the values 1 and 5 to simulate mildly ill-posed and almost well-posed scenarios, respectively. The
\textit{i\_laplace} family has four cases, of which the first three cases have smooth solutions and thus standard Tikhonov
regularization is suitable. This gives rise to a total of eleven examples. For each example, we take $100$ independent
and identically distributed (i.i.d.) realizations $\{\boldsymbol{\eta}_i\}_{i=1}^{100}$ of the Gaussian noise
with mean zero and covariance $\sigma^2I$, seeded by \textsc{MATLAB} function {\tt rng(i)} with $i=1,...,100$,
leading to 100 replicates of noisy data $\mathbf{g}^{\eta_i}\sim \mathcal N(\mathbf{g},\sigma^2 I)$; see Fig.
\ref{synthetic-data} for typical noisy data for \textit{baart}.

\begin{figure}
\begin{tabular}{ccc}
\includegraphics[width=0.3\textwidth, trim=48 30 48 30,clip]{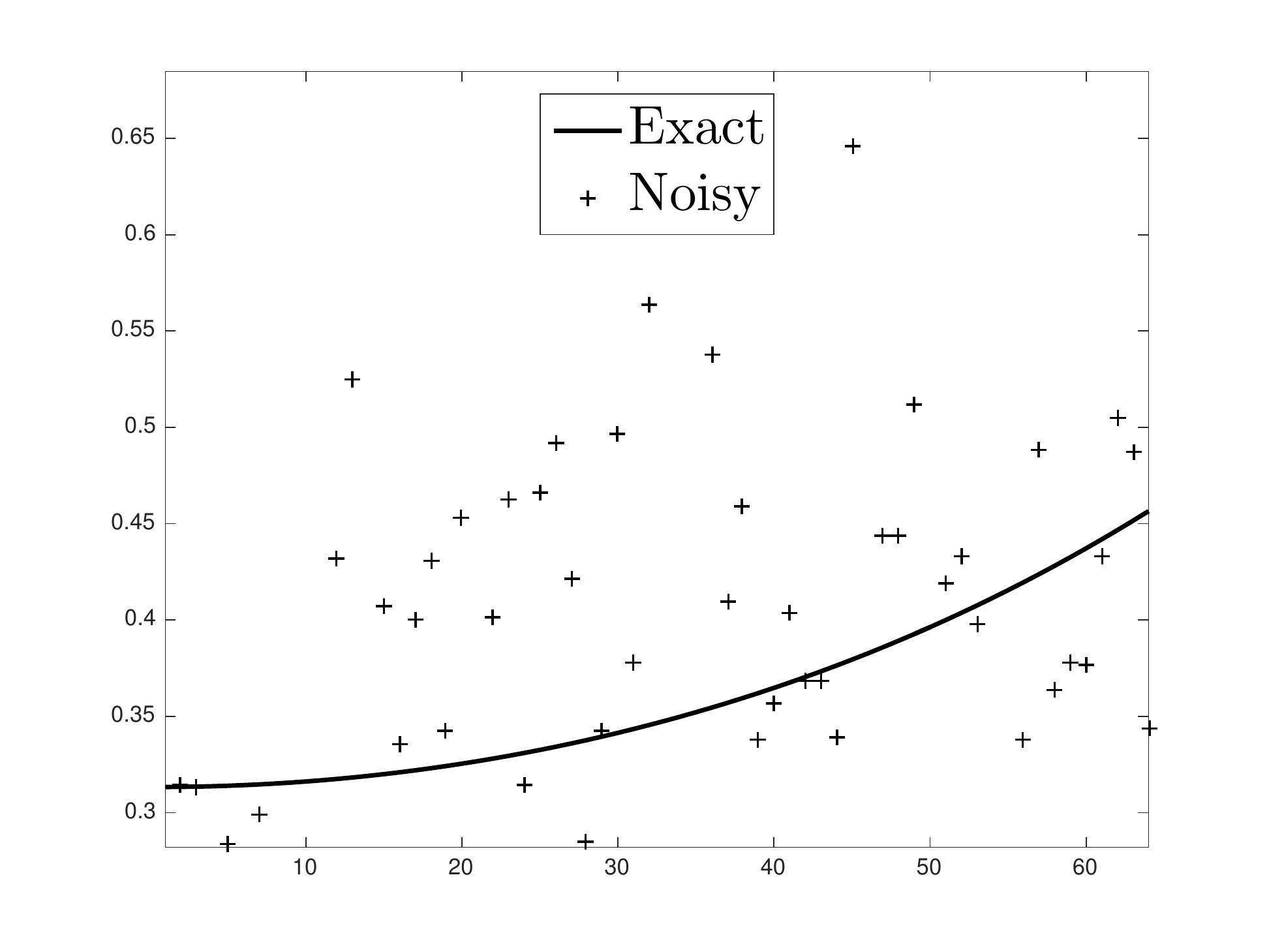}&
\includegraphics[width=0.3\textwidth, trim=48 30 48 30,clip]{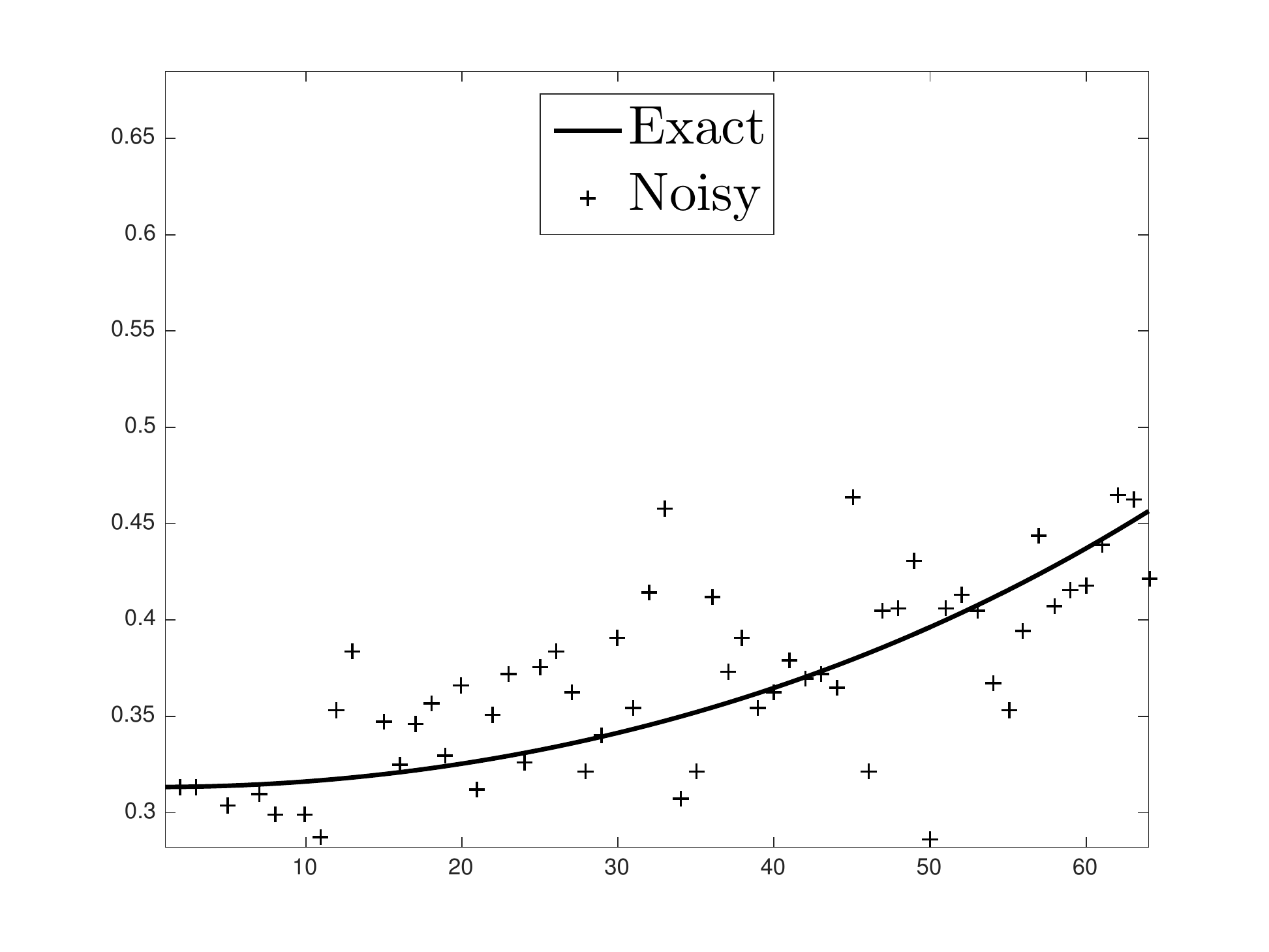}&
\includegraphics[width=0.3\textwidth, trim=48 30 48 30,clip]{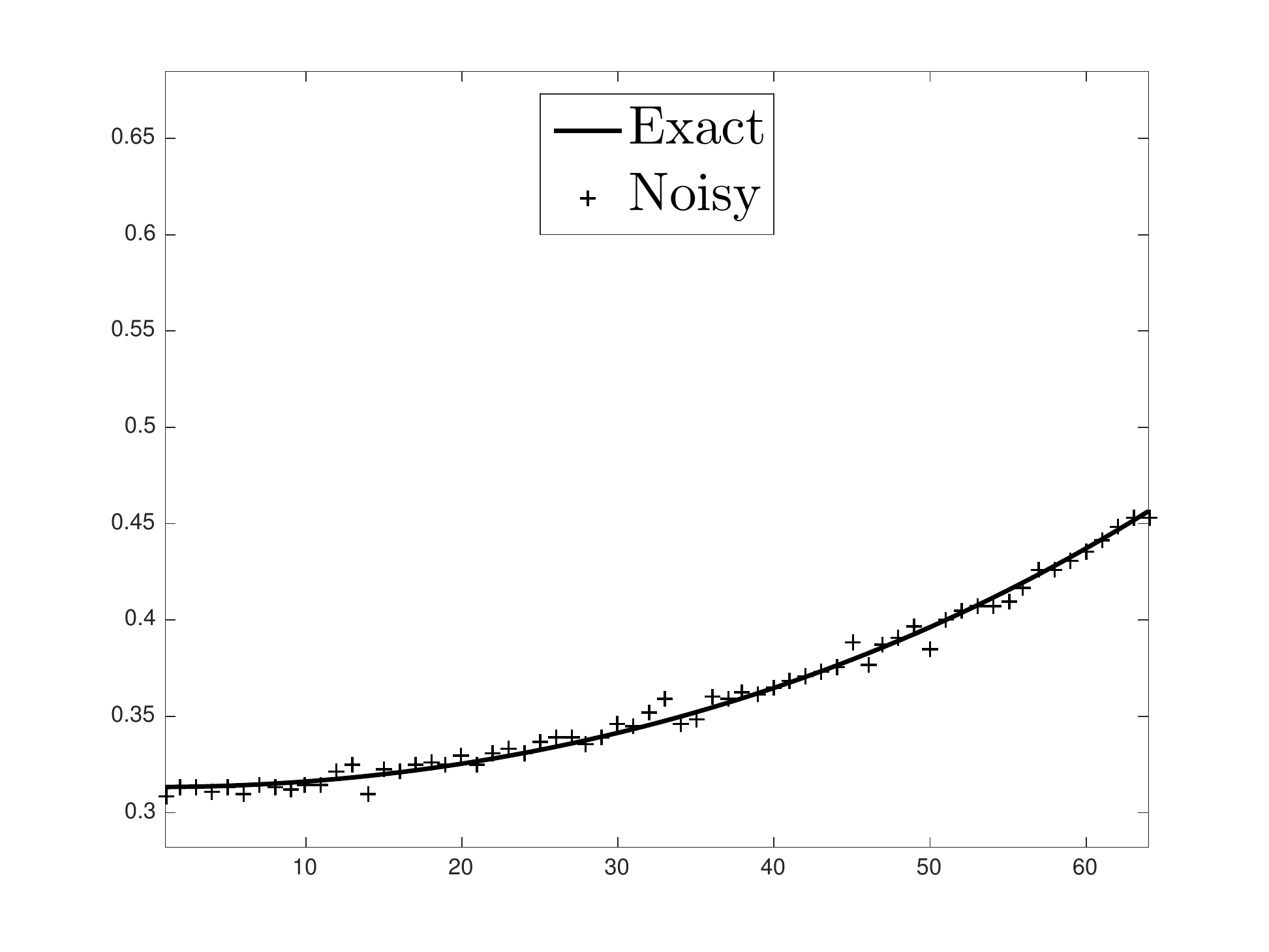}\\ 
(a) $\xi=10$ & (b) $\xi=20$ & (c) $\xi=40$
\end{tabular}
\caption{Synthetic data for \textit{baart} with $n=64$ and random seed $9$, i.e. {\tt rng(9)}.}
\label{synthetic-data}
\end{figure}

To measure the accuracy of a reconstruction $\mathbf{f}_\alpha^\eta$, we use the relative $\ell_2$ norm error.
ratio between the reconstruction $\mathbf{f}_\alpha^\eta$ and the ground truth $\mathbf{f}^\dag$
\begin{equation*}
\label{l2error}
 \mathrm{\varepsilon}(\alpha,\eta) = \frac{ \|\mathbf{f}_\alpha^{\eta} - \mathbf{f}^\dag\| }{ \| \mathbf{f}^\dag \| }
\end{equation*}
Other quality measures can be used, and the observations below remain unchanged.
Further, we define the `oracle'/optimal $\ell_2$ norm error, denoted by $\varepsilon_o(\eta)$, as
the minimum value of $\ell_2$ norm error achieved along the regularization path. Then we measure the efficiency of
a parameter choice rule by computing the ratio between the oracle $\ell_2$ norm error over the $\ell_2$ norm error of the solution given by the rule
\begin{equation}\label{eq:percentage}
 {\rm eff}_{*}(\eta) = \frac{ {\varepsilon}_o }{ {\varepsilon}(\alpha_{*},\eta) },
\end{equation}
where $*=$ DP, UPRE, BP, \texttt{PRO}, LC, GCV, QOC, \texttt{I-PRO}. Finally, to show the overall performance
of a choice rule on a fixed problem setup, we take the median {of its efficiency} over the 100 replicates.

\begin{figure}
\setlength{\tabcolsep}{2pt}
\begin{tabular}{ccc}
\includegraphics[width=0.32\textwidth]{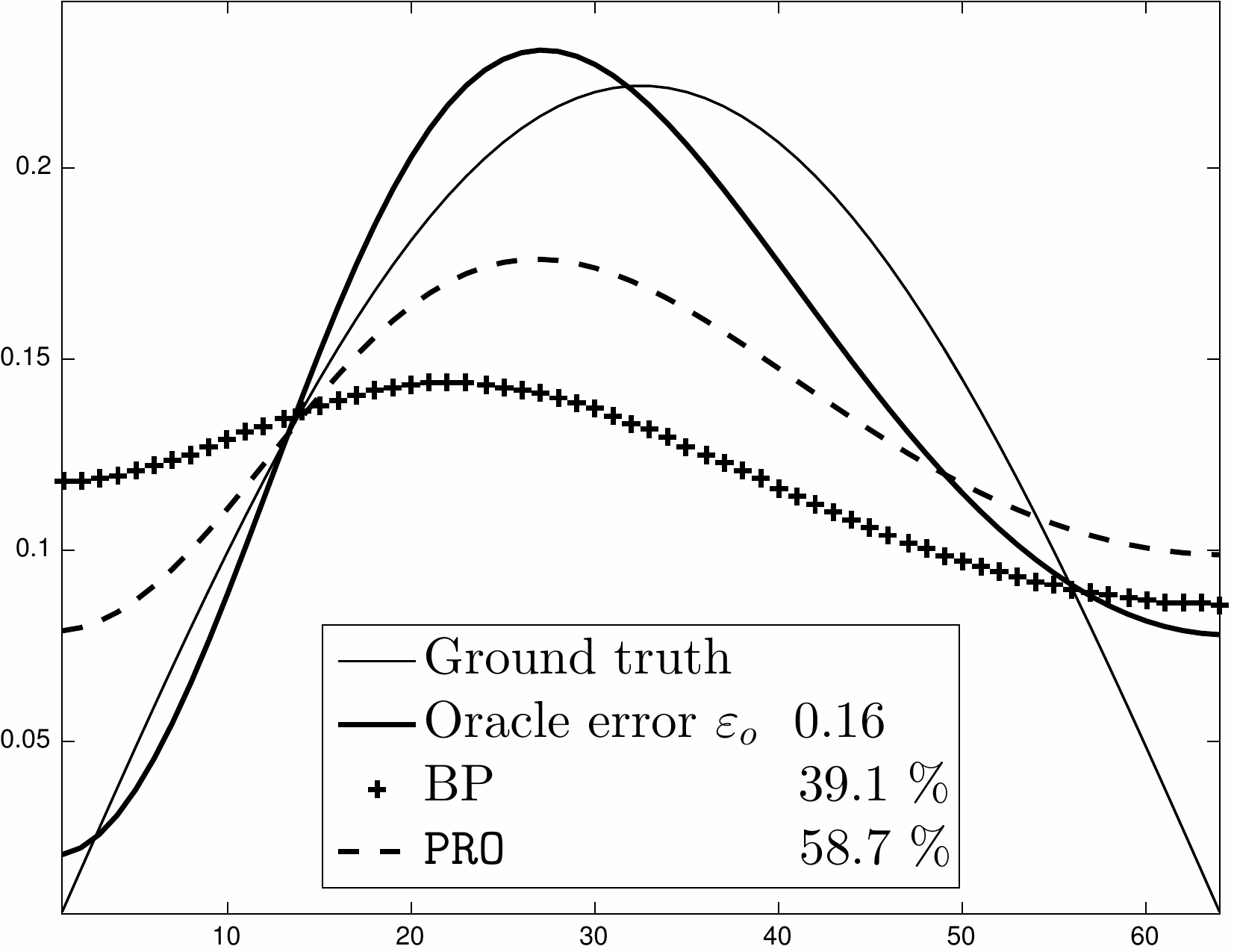}&
\includegraphics[width=0.32\textwidth]{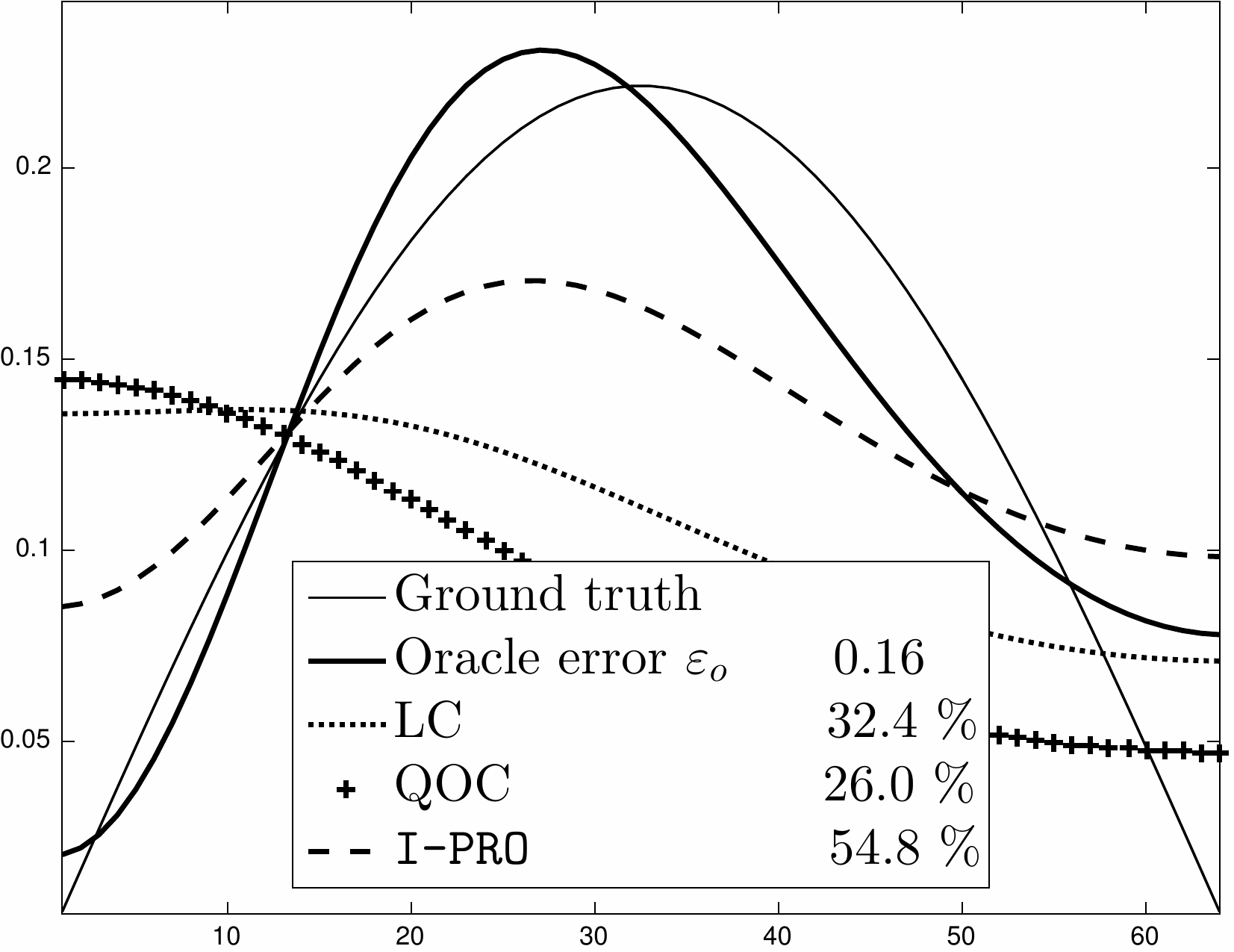}&
\includegraphics[width=0.32\textwidth]{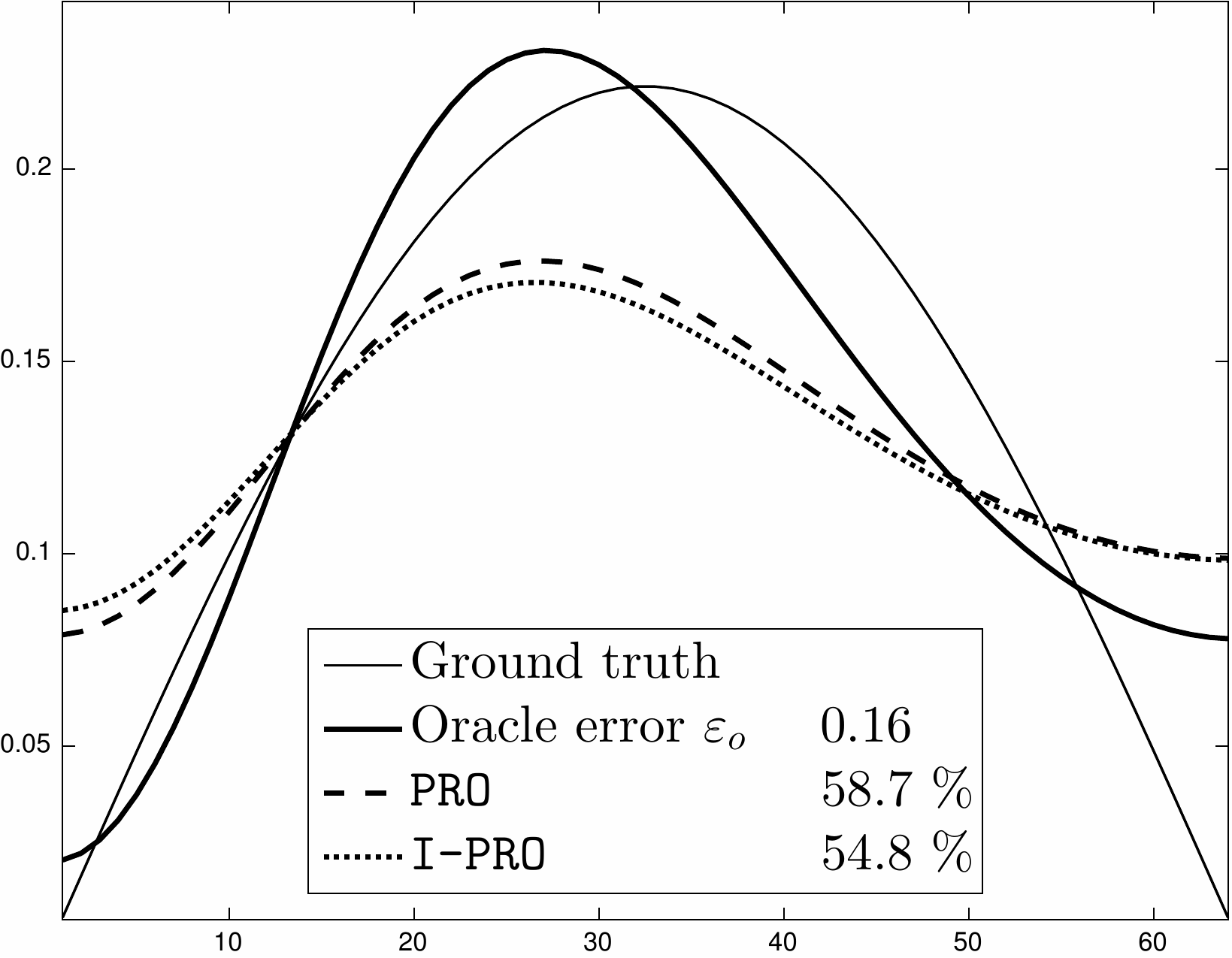}\\
(a) {\tt rng(9)} & (b) {\tt rng(9)} & (c) {\tt rng(9)}\\
\includegraphics[width=0.32\textwidth]{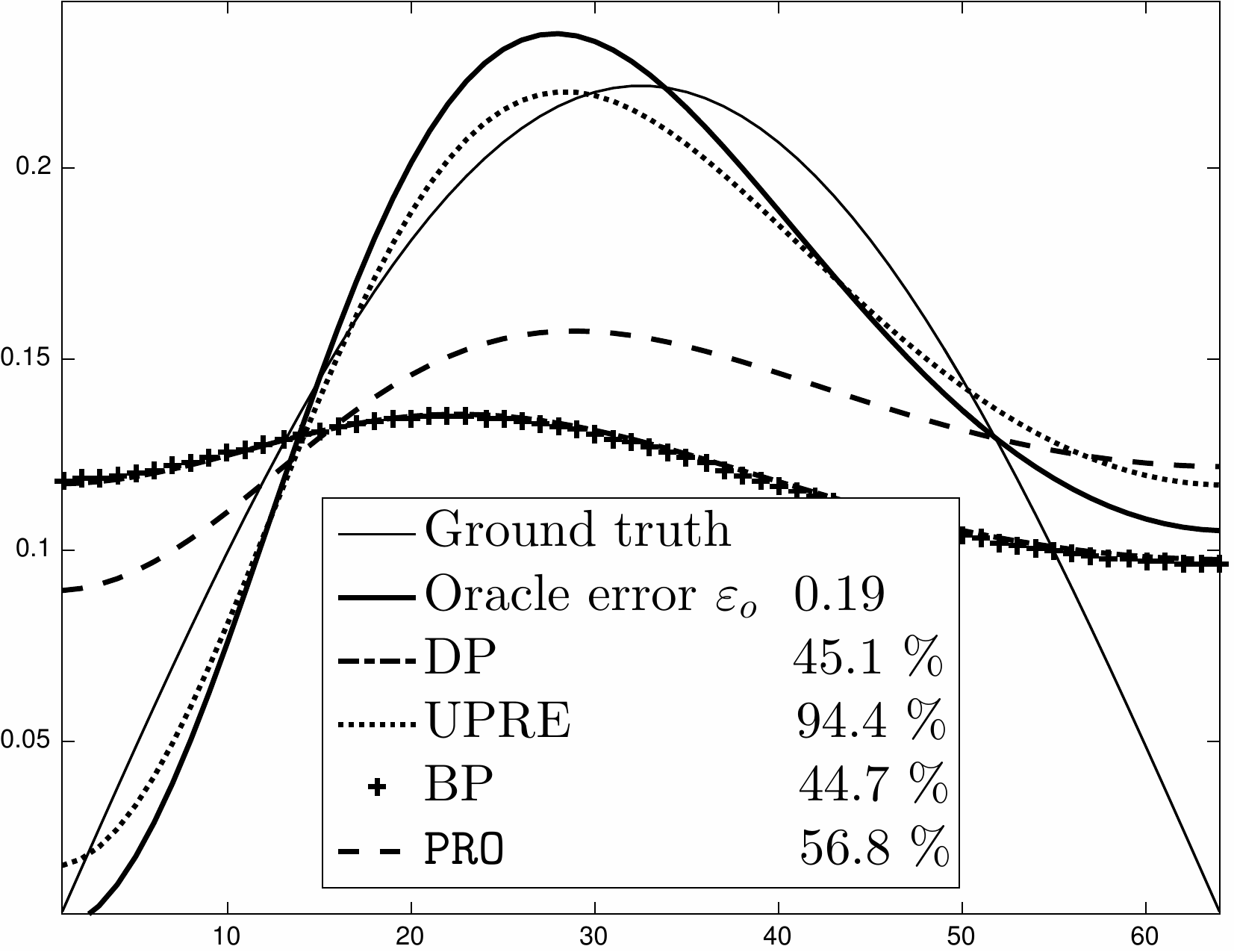}&
\includegraphics[width=0.32\textwidth]{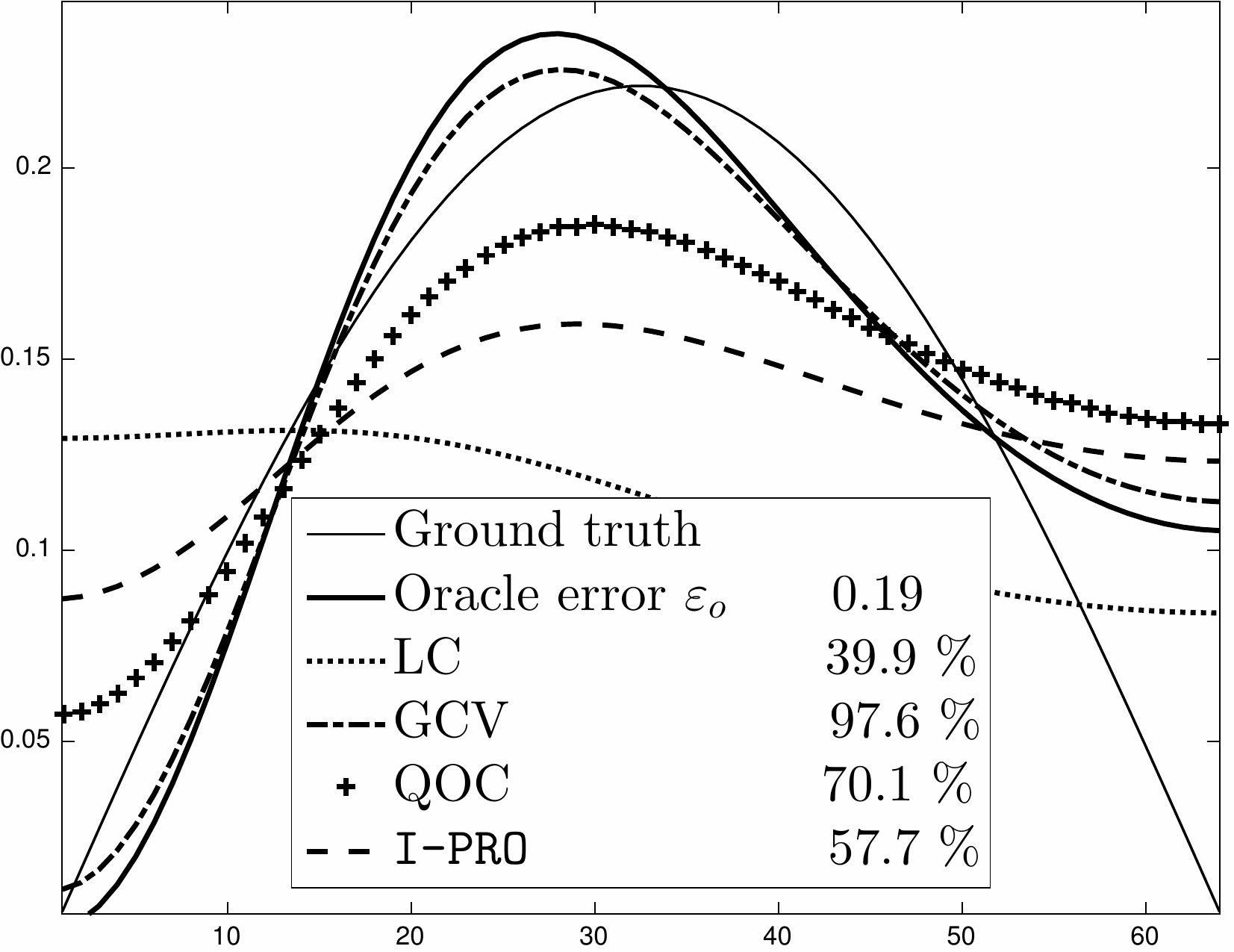}&
\includegraphics[width=0.32\textwidth]{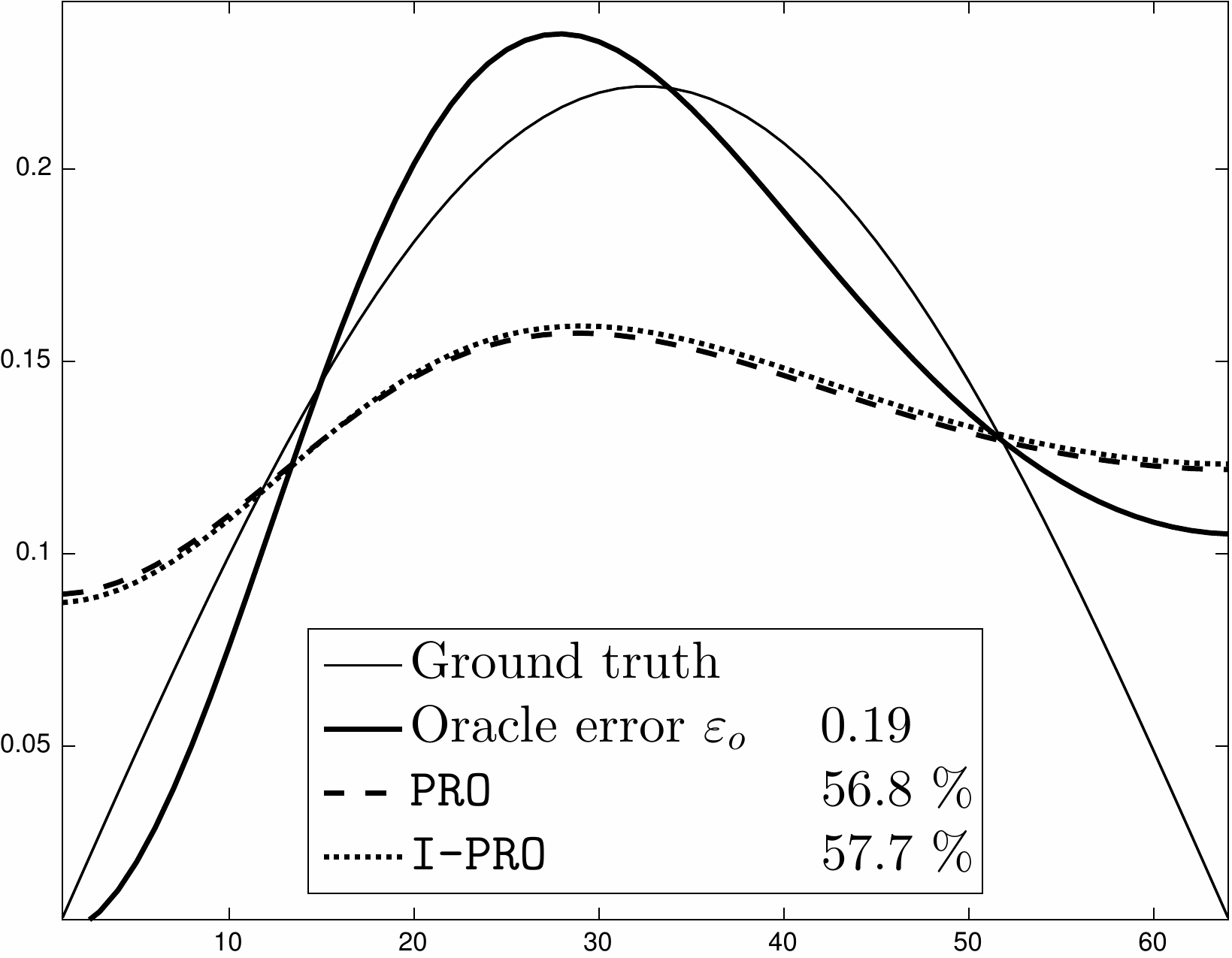}\\
(d) {\tt rng(12)} & (e) {\tt rng(12)} & (f) {\tt rng(12)}
\end{tabular}
\caption{Reconstructions for {\it baart} with $n=64$ and $\xi=10$ and two different noise realizations with {\tt rng(9)} and {\tt rng(12)}.
\label{reconstructions}}
\end{figure}

In Tables \ref{table:1} and \ref{table:2}, we present quantitative results for all the eight choice rules
under consideration at three typical SNR values (i.e., $\xi=10,20,40$) and two sample sizes
(i.e., $n=64$ and $1024$). The results are the median (columns from 3 to 10)
of the empirical efficiency distributions computed from the 100 replicates.
Below we examine the results more closely.

\begin{table}
\centering
\resizebox{0.6\textheight}{!}{
\begin{tabular}{ l c    r r r r   |   r r r r  }
\toprule
&\multicolumn{1}{c}{\,$\ell_2$ error \,} &\multicolumn{4}{c}{$\sigma$ know $\lambda$ unknown} &
\multicolumn{4}{c}{$\sigma$ and $\lambda$ unknowns}\\
\cmidrule(lr){2-2} \cmidrule(lr){3-6} \cmidrule(lr){7-10}
$n=64$ &
\multicolumn{1}{c}{Oracle}
& \multicolumn{1}{c}{DP}
&  \multicolumn{1}{c}{UPRE}
&  \multicolumn{1}{c}{BP}
&  \multicolumn{1}{c}{{\tt PRO}}
&  \multicolumn{1}{c}{LC}
& \multicolumn{1}{c}{GCV}
&  \multicolumn{1}{c}{QOC}
& \multicolumn{1}{c}{{\tt I-PRO}}\\
\midrule
&  \multicolumn{9}{c}{low statistic ($\xi=10$)} \\
\cmidrule(lr){2-2} \cmidrule(lr){3-6} \cmidrule(lr){7-10}
{\it baart} & 0.34 & 43.1$\%$ & 74.5$\%$ & 61.1$\%$ & 78.7$\%$ & 71.1$\%$ & 71.2$\%$ & 79.3$\%$ & 79.5$\%$\\
{\it deriv2} & 0.46 & 76.9$\%$ & 90.2$\%$ & 87.5$\%$ & 99.0$\%$ & 94.4$\%$ & 90.2$\%$ & 84.9$\%$ & 98.8$\%$\\
{\it foxgood} & 0.11 & 12.7$\%$ & 51.5$\%$ & 38.5$\%$ & 77.2$\%$ & 56.0$\%$ & 51.6$\%$ & 60.5$\%$ & 76.4$\%$\\
{\it gravity} & 0.19 & 73.8$\%$ & 76.3$\%$ & 86.9$\%$ & 87.4$\%$ & 93.3$\%$ & 75.2$\%$ & 96.9$\%$ & 86.7$\%$\\
{\it heat(1)} & 0.50 & 87.0$\%$ & 96.9$\%$ & 91.9$\%$ & 89.4$\%$ & 0.1$\%$ & 95.2$\%$ & 82.4$\%$ & 89.1$\%$\\
{\it heat(5)} & 0.38 & 98.5$\%$ & 90.2$\%$ & 80.2$\%$ & 98.5$\%$ & 88.0$\%$ & 88.0$\%$ & 55.0$\%$ & 55.9$\%$\\
{\it i\_laplace(1)} & 0.22 & 63.2$\%$ & 87.4$\%$ & 84.0$\%$ & 96.1$\%$ & 96.6$\%$ & 87.4$\%$ & 96.2$\%$ & 96.4$\%$\\
{\it i\_laplace(2)} & 0.86 & 97.0$\%$ & 99.2$\%$ & 98.3$\%$ & 99.4$\%$ & 99.4$\%$ & 99.3$\%$ & 99.3$\%$ & 99.4$\%$\\
{\it i\_laplace(3)} & 0.28 & 65.8$\%$ & 89.8$\%$ & 84.4$\%$ & 95.9$\%$ & 96.7$\%$ & 87.7$\%$ & 94.9$\%$ & 96.3$\%$\\
{\it phillips} & 0.19 & 77.6$\%$ & 85.4$\%$ & 87.6$\%$ & 95.5$\%$ & 91.6$\%$ & 82.8$\%$ & 97.2$\%$ & 94.4$\%$\\
{\it shaw} & 0.24 & 57.9$\%$ & 88.7$\%$ & 74.1$\%$ & 95.9$\%$ & 95.2$\%$ & 86.3$\%$ & 95.5$\%$ & 96.2$\%$\\
\midrule
&  \multicolumn{9}{c}{medium statistic ($\xi=20$)}  \\
\cmidrule(lr){2-2} \cmidrule(lr){3-6} \cmidrule(lr){7-10}
{\it baart} & 0.21 & 45.7$\%$ & 78.4$\%$ & 35.3$\%$ & 68.8$\%$ & 72.8$\%$ & 72.3$\%$ & 68.4$\%$ & 69.4$\%$\\
{\it deriv2} & 0.39 & 87.7$\%$ & 94.5$\%$ & 93.2$\%$ & 97.5$\%$ & 98.5$\%$ & 92.8$\%$ & 85.1$\%$ & 97.4$\%$\\
{\it foxgood} & 0.06 & 14.5$\%$ & 48.8$\%$ & 32.4$\%$ & 79.5$\%$ & 83.7$\%$ & 48.7$\%$ & 86.9$\%$ & 79.6$\%$\\
{\it gravity} & 0.11 & 84.4$\%$ & 78.2$\%$ & 85.5$\%$ & 90.1$\%$ & 76.0$\%$ & 78.3$\%$ & 97.3$\%$ & 87.5$\%$\\
{\it heat(1)} & 0.33 & 90.4$\%$ & 89.8$\%$ & 96.1$\%$ & 74.2$\%$ & 0.3$\%$ & 86.9$\%$ & 58.6$\%$ & 71.3$\%$\\
{\it heat(5)} & 0.18 & 98.4$\%$ & 92.8$\%$ & 98.5$\%$ & 99.4$\%$ & 94.5$\%$ & 94.5$\%$ & 80.9$\%$ & 80.9$\%$\\
{\it i\_laplace(1)} & 0.16 & 74.8$\%$ & 88.1$\%$ & 82.9$\%$ & 95.1$\%$ & 95.3$\%$ & 88.2$\%$ & 92.1$\%$ & 95.2$\%$\\
{\it i\_laplace(2)} & 0.84 & 97.9$\%$ & 99.0$\%$ & 98.3$\%$ & 98.8$\%$ & 99.4$\%$ & 99.0$\%$ & 99.0$\%$ & 98.8$\%$\\
{\it i\_laplace(3)} & 0.18 & 60.6$\%$ & 86.2$\%$ & 69.0$\%$ & 89.4$\%$ & 92.2$\%$ & 85.8$\%$ & 77.7$\%$ & 89.6$\%$\\
{\it phillips} & 0.09 & 82.9$\%$ & 75.2$\%$ & 80.4$\%$ & 90.3$\%$ & 73.9$\%$ & 73.5$\%$ & 99.0$\%$ & 88.0$\%$\\
{\it shaw} & 0.18 & 70.8$\%$ & 84.4$\%$ & 82.6$\%$ & 97.6$\%$ & 95.8$\%$ & 84.5$\%$ & 95.5$\%$ & 97.5$\%$\\
\midrule
&  \multicolumn{9}{c}{high statistic ($\xi=40$)}  \\
\cmidrule(lr){2-2} \cmidrule(lr){3-6} \cmidrule(lr){7-10}
{\it baart} & 0.15 & 51.7$\%$ & 76.9$\%$ & 24.9$\%$ & 72.2$\%$ & 88.9$\%$ & 71.2$\%$ & 89.8$\%$ & 70.6$\%$\\
{\it deriv2} & 0.26 & 95.0$\%$ & 95.6$\%$ & 97.7$\%$ & 88.0$\%$ & 91.6$\%$ & 95.5$\%$ & 88.9$\%$ & 86.2$\%$\\
{\it foxgood} & 0.02 & 24.7$\%$ & 56.1$\%$ & 10.5$\%$ & 84.9$\%$ & 49.5$\%$ & 41.8$\%$ & 84.7$\%$ & 84.6$\%$\\
{\it gravity} & 0.04 & 86.4$\%$ & 84.8$\%$ & 80.4$\%$ & 95.9$\%$ & 45.6$\%$ & 84.9$\%$ & 96.2$\%$ & 95.5$\%$\\
{\it heat(1)} & 0.13 & 95.4$\%$ & 82.5$\%$ & 23.8$\%$ & 59.7$\%$ & 1.6$\%$ & 78.9$\%$ & 96.1$\%$ & 36.2$\%$\\
{\it heat(5)} & 0.02 & 90.3$\%$ & 99.8$\%$ & 85.3$\%$ & 99.8$\%$ & 16.4$\%$ & 16.4$\%$ & 99.6$\%$ & 99.6$\%$\\
{\it i\_laplace(1)} & 0.10 & 73.2$\%$ & 87.0$\%$ & 83.4$\%$ & 92.3$\%$ & 85.5$\%$ & 86.9$\%$ & 87.8$\%$ & 92.7$\%$\\
{\it i\_laplace(2)} & 0.82 & 98.4$\%$ & 99.1$\%$ & 98.6$\%$ & 97.9$\%$ & 99.6$\%$ & 99.1$\%$ & 98.5$\%$ & 97.7$\%$\\
{\it i\_laplace(3)} & 0.06 & 60.6$\%$ & 78.2$\%$ & 68.0$\%$ & 68.9$\%$ & 69.2$\%$ & 73.4$\%$ & 94.1$\%$ & 67.3$\%$\\
{\it phillips} & 0.03 & 93.5$\%$ & 56.8$\%$ & 90.6$\%$ & 72.0$\%$ & 22.3$\%$ & 51.1$\%$ & 98.0$\%$ & 70.3$\%$\\
{\it shaw} & 0.11 & 64.4$\%$ & 87.0$\%$ & 68.4$\%$ & 79.6$\%$ & 91.8$\%$ & 87.4$\%$ & 69.1$\%$ & 80.0$\%$\\
\bottomrule
\end{tabular}}
\caption{Performance comparison between choice rules DP, UPRE, BP, {\tt PRO}, LC, GCV, QOC, {\tt I-PRO} on examples from Regularization tools.
For each row: column one gives the problem and column two shows the oracle $\ell_2$-norm error.
From the third to tenth column the median efficiency of the choice rule (see equation \ref{eq:percentage}).
The three blocks correspond to three different levels of statistic. The sample size $n$ of each problem is 64.\label{table:1}}
\end{table}

\begin{table}
\centering
\resizebox{0.6\textheight}{!}{
\begin{tabular}{ l c    r r r r   |   r r r r  }
\toprule
&
\multicolumn{1}{c}{\,$\ell_2$ error \,} &
\multicolumn{4}{c}{$\sigma$ know $\lambda$ unknown} &
\multicolumn{4}{c}{$\sigma$ and $\lambda$ unknowns}\\
\cmidrule(lr){2-2} \cmidrule(lr){3-6} \cmidrule(lr){7-10}
$n=1024$ &
\multicolumn{1}{c}{Oracle}
& \multicolumn{1}{c}{DP}
&  \multicolumn{1}{c}{UPRE}
&  \multicolumn{1}{c}{BP}
&  \multicolumn{1}{c}{{\tt PRO}}
&  \multicolumn{1}{c}{LC}
& \multicolumn{1}{c}{GCV}
&  \multicolumn{1}{c}{QOC}
& \multicolumn{1}{c}{{\tt I-PRO}}\\
\midrule
&  \multicolumn{9}{c}{low statistic ($\xi=10$)} \\
\cmidrule(lr){2-2} \cmidrule(lr){3-6} \cmidrule(lr){7-10}
{\it baart} & 0.21 & 30.1$\%$ & 86.7$\%$ & 41.9$\%$ & 69.8$\%$ & 41.6$\%$ & 83.3$\%$ & 74.5$\%$ & 70.3$\%$\\
{\it deriv2} & 0.37 & 70.6$\%$ & 96.4$\%$ & 67.4$\%$ & 95.6$\%$ & 78.0$\%$ & 96.3$\%$ & 90.1$\%$ & 95.9$\%$\\
{\it foxgood} & 0.06 & 4.3$\%$ & 55.6$\%$ & 20.0$\%$ & 84.4$\%$ & 29.7$\%$ & 49.3$\%$ & 88.8$\%$ & 84.2$\%$\\
{\it gravity} & 0.09 & 40.0$\%$ & 77.8$\%$ & 42.3$\%$ & 90.1$\%$ & 95.5$\%$ & 76.1$\%$ & 97.4$\%$ & 90.4$\%$\\
{\it heat(1)} & 0.30 & 67.0$\%$ & 91.7$\%$ & 50.8$\%$ & 71.8$\%$ & 56.8$\%$ & 90.2$\%$ & 51.3$\%$ & 71.5$\%$\\
{\it heat(5)} & 0.20 & 97.4$\%$ & 62.2$\%$ & 99.9$\%$ & 83.5$\%$ & 0.1$\%$ & 60.9$\%$ & 0.1$\%$ & 81.8$\%$\\
{\it i\_laplace(1)} & 0.26 & 81.8$\%$ & 96.4$\%$ & 76.4$\%$ & 93.8$\%$ & 97.0$\%$ & 96.4$\%$ & 98.0$\%$ & 93.8$\%$\\
{\it i\_laplace(2)} & 0.81 & 96.6$\%$ & 99.4$\%$ & 96.4$\%$ & 99.5$\%$ & 98.2$\%$ & 99.4$\%$ & 99.0$\%$ & 99.5$\%$\\
{\it i\_laplace(3)} & 0.20 & 45.7$\%$ & 90.2$\%$ & 47.7$\%$ & 78.1$\%$ & 65.0$\%$ & 90.7$\%$ & 92.1$\%$ & 78.2$\%$\\
{\it phillips} & 0.08 & 44.4$\%$ & 64.4$\%$ & 39.0$\%$ & 86.3$\%$ & 98.3$\%$ & 65.5$\%$ & 99.1$\%$ & 86.8$\%$\\
{\it shaw} & 0.18 & 41.1$\%$ & 87.1$\%$ & 54.0$\%$ & 97.7$\%$ & 86.6$\%$ & 83.2$\%$ & 97.6$\%$ & 97.9$\%$\\
\midrule
&  \multicolumn{9}{c}{medium statistic ($\xi=20$)}  \\
\cmidrule(lr){2-2} \cmidrule(lr){3-6} \cmidrule(lr){7-10}
{\it baart} & 0.17 & 29.4$\%$ & 80.8$\%$ & 28.7$\%$ & 67.8$\%$ & 58.8$\%$ & 79.6$\%$ & 81.5$\%$ & 68.0$\%$\\
{\it deriv2} & 0.31 & 76.5$\%$ & 95.8$\%$ & 71.6$\%$ & 92.0$\%$ & 85.5$\%$ & 95.4$\%$ & 89.6$\%$ & 91.9$\%$\\
{\it foxgood} & 0.04 & 6.3$\%$ & 49.1$\%$ & 16.0$\%$ & 91.5$\%$ & 52.1$\%$ & 45.9$\%$ & 88.2$\%$ & 91.2$\%$\\
{\it gravity} & 0.05 & 45.7$\%$ & 81.1$\%$ & 46.0$\%$ & 91.7$\%$ & 97.5$\%$ & 81.6$\%$ & 96.8$\%$ & 91.2$\%$\\
{\it heat(1)} & 0.19 & 74.0$\%$ & 86.8$\%$ & 59.5$\%$ & 64.1$\%$ & 84.7$\%$ & 86.3$\%$ & 95.5$\%$ & 62.7$\%$\\
{\it heat(5)} & 0.10 & 91.1$\%$ & 53.7$\%$ & 98.8$\%$ & 74.3$\%$ & 0.2$\%$ & 52.6$\%$ & 0.2$\%$ & 72.3$\%$\\
{\it i\_laplace(1)} & 0.24 & 87.0$\%$ & 97.5$\%$ & 83.5$\%$ & 93.0$\%$ & 97.3$\%$ & 97.6$\%$ & 95.8$\%$ & 93.0$\%$\\
{\it i\_laplace(2)} & 0.79 & 97.1$\%$ & 99.5$\%$ & 96.8$\%$ & 98.9$\%$ & 98.4$\%$ & 99.5$\%$ & 99.0$\%$ & 99.0$\%$\\
{\it i\_laplace(3)} & 0.12 & 46.5$\%$ & 91.3$\%$ & 45.3$\%$ & 72.2$\%$ & 76.2$\%$ & 91.3$\%$ & 95.1$\%$ & 71.8$\%$\\
{\it phillips} & 0.04 & 53.7$\%$ & 54.4$\%$ & 41.8$\%$ & 72.8$\%$ & 90.0$\%$ & 52.8$\%$ & 99.1$\%$ & 73.2$\%$\\
{\it shaw} & 0.15 & 47.9$\%$ & 87.6$\%$ & 68.9$\%$ & 95.1$\%$ & 93.7$\%$ & 85.9$\%$ & 91.6$\%$ & 95.1$\%$\\
\midrule
&  \multicolumn{9}{c}{high statistic ($\xi=40$)}  \\
\cmidrule(lr){2-2} \cmidrule(lr){3-6} \cmidrule(lr){7-10}
{\it baart} & 0.12 & 30.1$\%$ & 78.6$\%$ & 19.3$\%$ & 68.6$\%$ & 78.1$\%$ & 76.1$\%$ & 76.2$\%$ & 68.4$\%$\\
{\it deriv2} & 0.21 & 88.4$\%$ & 95.4$\%$ & 78.3$\%$ & 82.1$\%$ & 98.7$\%$ & 95.1$\%$ & 89.6$\%$ & 81.6$\%$\\
{\it foxgood} & 0.02 & 11.8$\%$ & 56.9$\%$ & 7.1$\%$ & 88.1$\%$ & 90.1$\%$ & 55.8$\%$ & 72.3$\%$ & 88.7$\%$\\
{\it gravity} & 0.02 & 54.6$\%$ & 82.3$\%$ & 49.4$\%$ & 96.9$\%$ & 79.3$\%$ & 80.9$\%$ & 94.0$\%$ & 96.8$\%$\\
{\it heat(1)} & 0.07 & 87.7$\%$ & 77.3$\%$ & 12.4$\%$ & 53.1$\%$ & 85.6$\%$ & 77.1$\%$ & 99.0$\%$ & 48.5$\%$\\
{\it heat(5)} & 0.02 & 71.8$\%$ & 39.0$\%$ & 92.2$\%$ & 54.8$\%$ & 0.5$\%$ & 37.7$\%$ & 0.5$\%$ & 51.7$\%$\\
{\it i\_laplace(1)} & 0.21 & 90.8$\%$ & 98.3$\%$ & 88.3$\%$ & 91.2$\%$ & 98.5$\%$ & 98.3$\%$ & 95.6$\%$ & 91.0$\%$\\
{\it i\_laplace(2)} & 0.77 & 98.3$\%$ & 99.7$\%$ & 97.9$\%$ & 98.4$\%$ & 99.2$\%$ & 99.7$\%$ & 98.9$\%$ & 98.4$\%$\\
{\it i\_laplace(3)} & 0.05 & 47.6$\%$ & 86.8$\%$ & 43.9$\%$ & 59.0$\%$ & 97.1$\%$ & 87.4$\%$ & 95.9$\%$ & 58.6$\%$\\
{\it phillips} & 0.02 & 76.2$\%$ & 65.5$\%$ & 61.0$\%$ & 83.7$\%$ & 50.9$\%$ & 63.6$\%$ & 93.9$\%$ & 83.1$\%$\\
{\it shaw} & 0.06 & 38.2$\%$ & 87.9$\%$ & 36.6$\%$ & 56.7$\%$ & 92.3$\%$ & 87.1$\%$ & 80.6$\%$ & 56.2$\%$\\
\bottomrule
\end{tabular}}
\caption{See caption of Table \ref{table:1}, but with the sample size $n$ of each problem being 1024. \label{table:2}}
\end{table}

First, the oracle $\ell_2$ norm error, $\varepsilon_o$, in the second column depends crucially on the problem type,
condition number and ground truth. When $n=64$ (see Table \ref{table:1}), except for one case, $\varepsilon_o$
ranges from 0.11 to 0.50, from 0.06 to 0.39, and from 0.02 to 0.26, for low, medium and high statistic, respectively. In the
exceptional case, i.e., \textit{i\_laplace(2)}, the ground truth object $\mathbf{f}^\dag$ to be recovered is actually almost invisible
within the meaning of the inverse problems theory \cite{bertero1998introduction}, and so $\varepsilon_o$ fails to reach barely $0.82$.
For all choice rules with $\sigma^2$ known, the efficiency of any parameter choice rule is almost always larger than
$60\%$. The efficiency depends on both the choice rule and specific problem. For example, \textit{baart} and \textit{foxgood} are
fairly challenging for all eight choice rules. In particular, on these two problems, DP does not perform as well as the others.
This may be related to the well known fact that DP is sensitive to the estimation accuracy of the noise variance $\sigma^2$.
Despite a strong dependency on the problem, the percentage generally increases with the SNR, indicating the convergence
of these rules (at least on average). It is worth noting that LC is very effective in most cases but it fails spectacularly on
\textit{heat}, for which all other methods work reasonably well. Moreover, the efficiency distributions for DP, UPRE, GCV (and
occasionally also LC) have very low median, and thus in practice, these choice rules can fail to choose a suitable $\alpha$ value.
The observations on GCV is consistent with prior empirical study in \cite{Lucka:2018}. Surprisingly, \texttt{PRO} and \texttt{I-PRO}
(and BP and QOC) can consistently provide satisfactory solutions for all noise realizations. This property is highly desirable in
practice, and it is attributed to the fact that \texttt{PRO} and \texttt{I-PRO} are designed to achieve the minimum error in the
mean squared sense (at least approximately!), and thus are fairly robust with respect to noise realization. This
contrasts sharply with other existing choice rules, which always suffer from a significant number of failures. One example of
failure is given in Fig. {\ref{reconstructions}}: whereas in (d) and (e) all choice rules yield good estimates of the ground truth,
in (a) and (b) existing choice rules, e.g., UPRE or GCV, are highly inefficient and the corresponding solutions are omitted since
they extend over several orders of magnitude above the ground truth.

In practice, it is also important to check the variance of the efficiency distribution. By considering each
problem separately (given the SNR $\xi$ and the size $n$), the simulation study indicates that the empirical
efficiency distribution by \texttt{PRO} has a very small variance. This is due to the fact that given $\sigma^2$, the estimation
$\hat{\rho^2}$ of the data norm $\rho^2$ given in \eqref{eq:ume} is relatively insensitive to the noise realization, especially
when the sample size $n$ is not too small. In contrast, the empirical efficiency distributions by DP, UPRE
and GCV (and sometimes also LC) spread out, and thus these methods do not enjoy uniform performance with respect
to noise realization. This behavior is clearly visible in Fig. \ref{figure:boxplots}: \texttt{PRO} and \texttt{I-PRO} plots
show distributions with small variances, whereas that of the other methods have broad tails; see also Fig. \ref{reconstructions}
for the illustration of the excellent stability of the \texttt{PRO} and \texttt{I-PRO} solutions. Overall, BP works fairly well,
but it is below par on the example \textit{baart}. Fig. \ref{figure:boxplots} indicates that among all eight rules,
\texttt{PRO}, \texttt{I-PRO} and QOC emerge as the strongest contenders.
\begin{figure}
\includegraphics[width=0.33\textwidth]{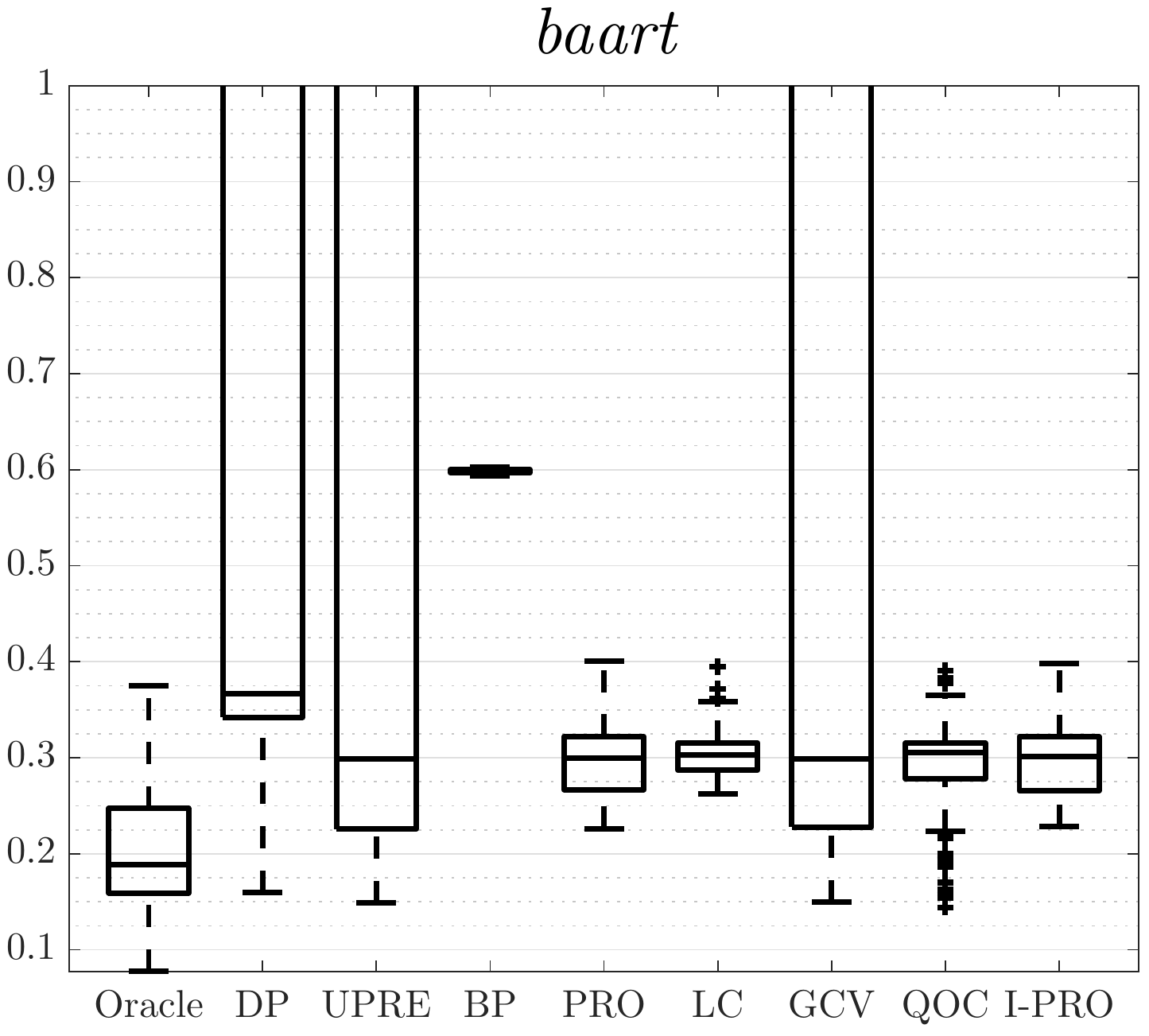}
\includegraphics[width=0.33\textwidth]{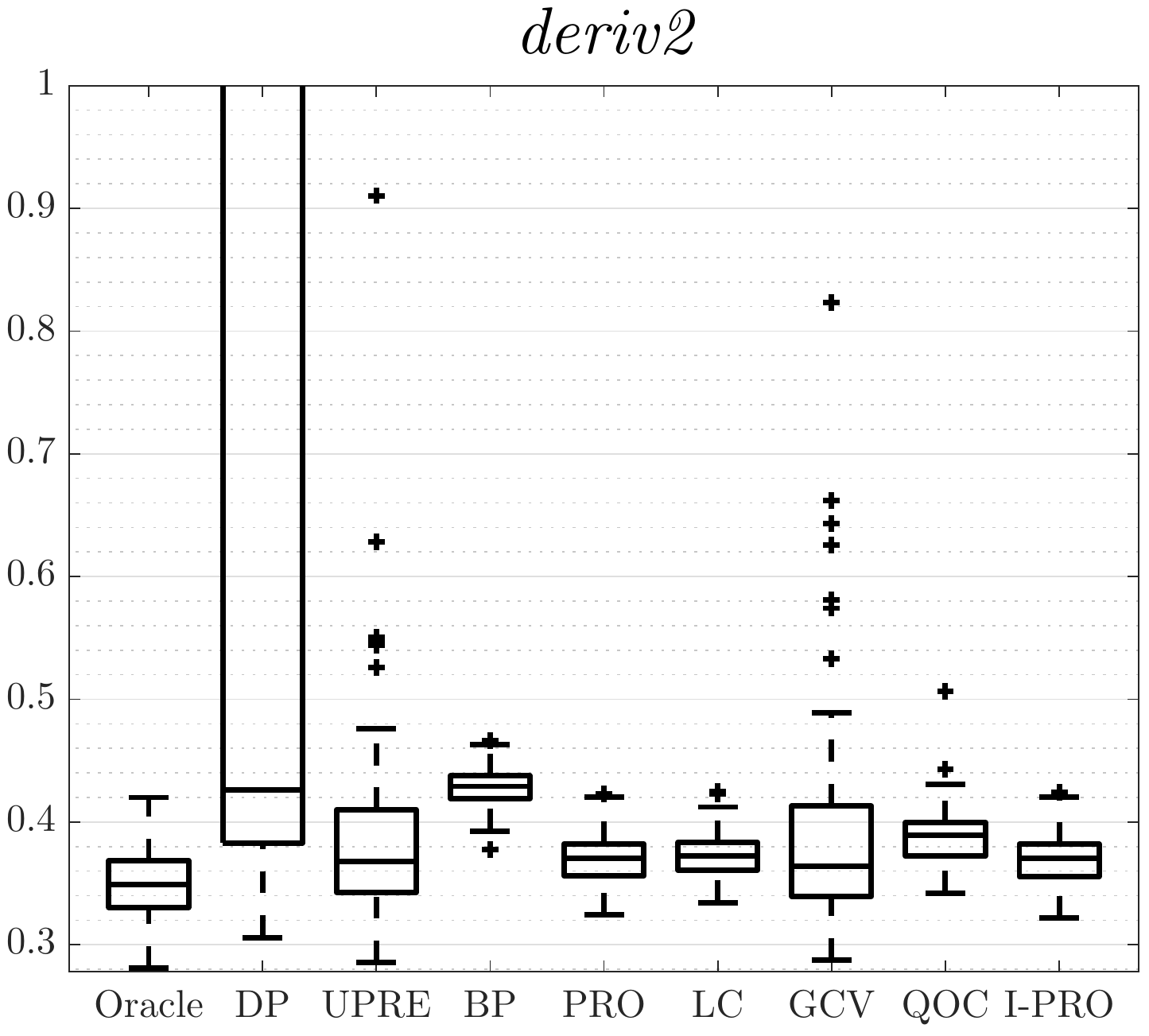}
\includegraphics[width=0.33\textwidth]{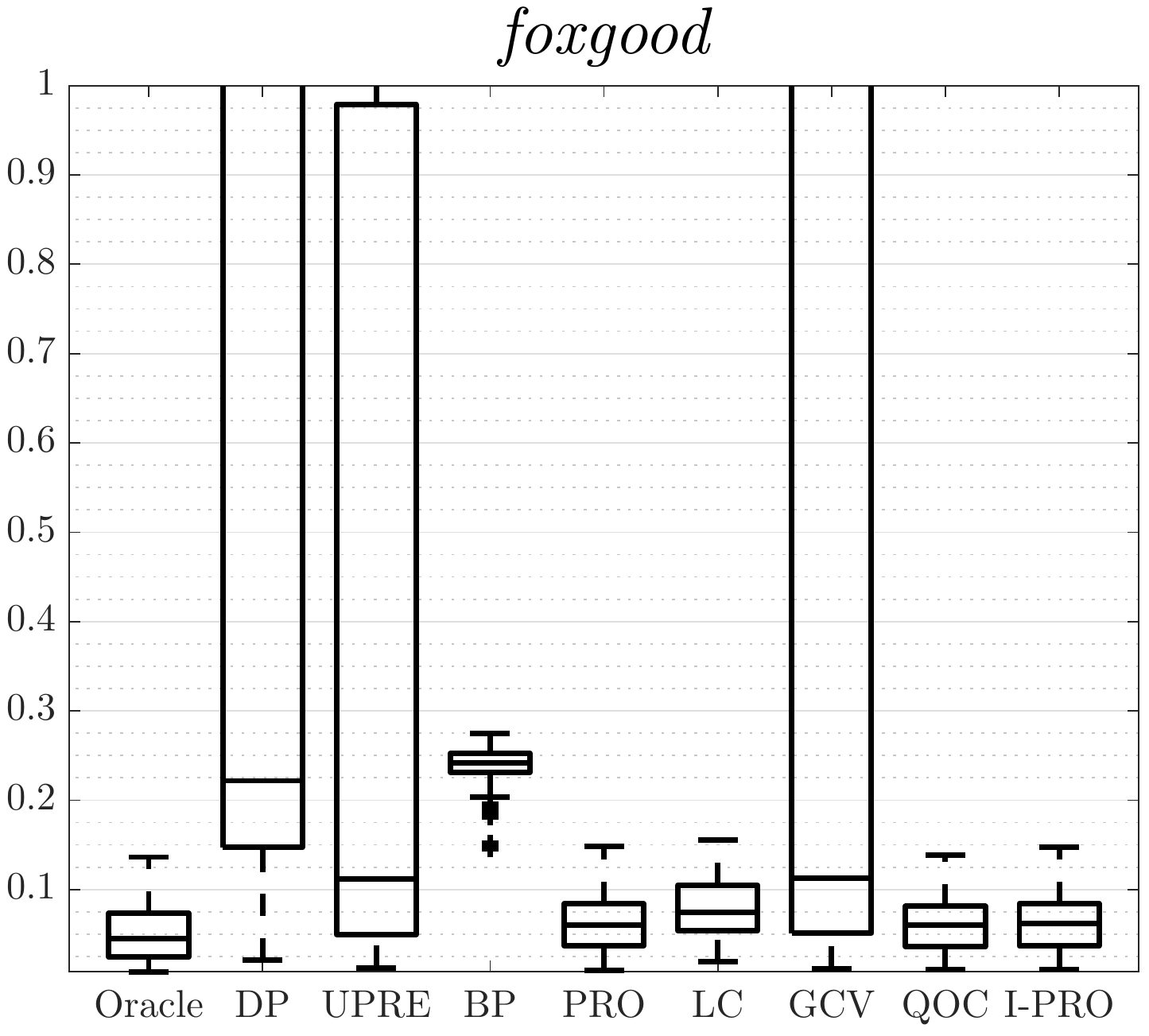}\\[2pt]
\includegraphics[width=0.33\textwidth]{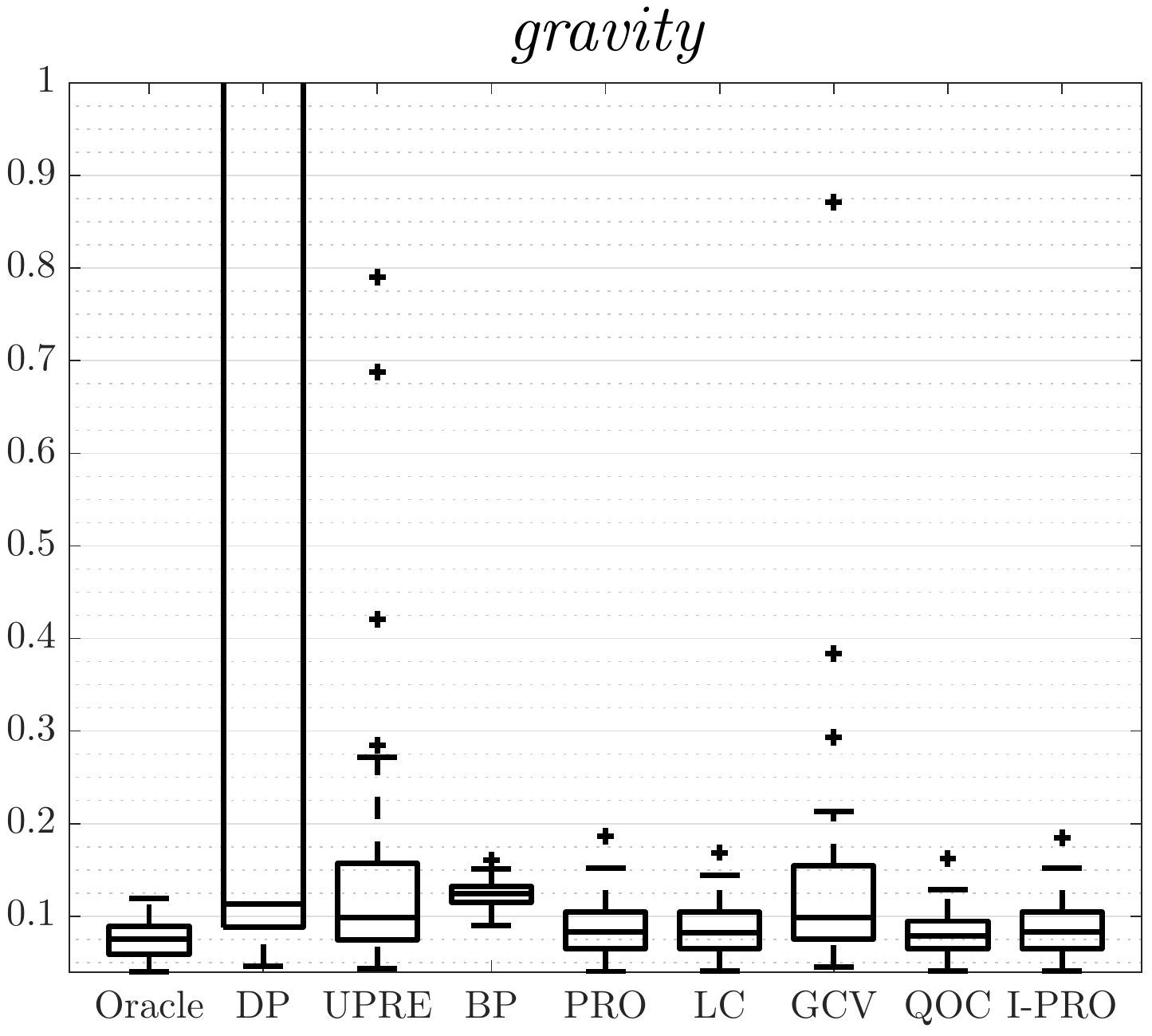}
\includegraphics[width=0.33\textwidth]{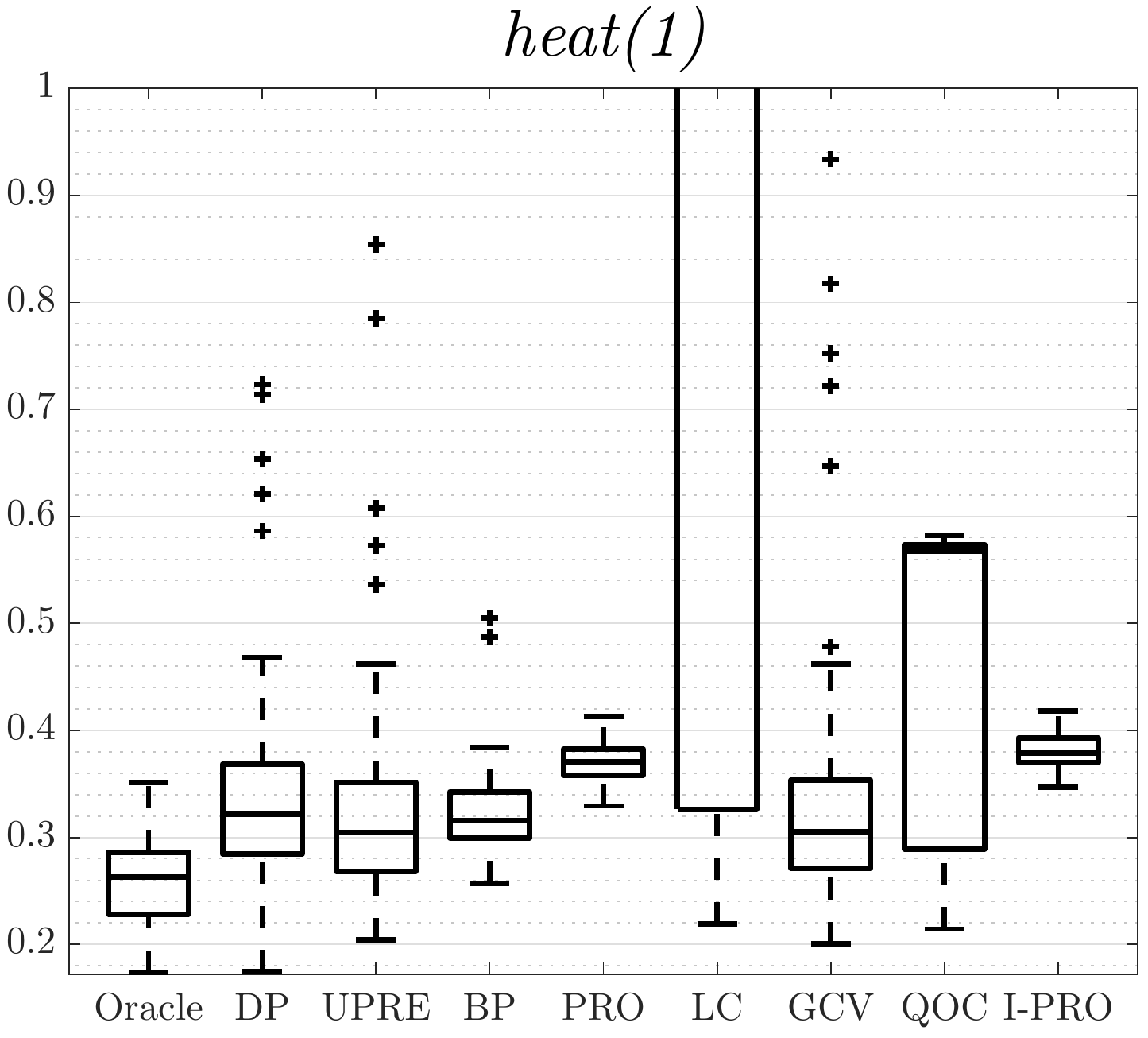}
\includegraphics[width=0.33\textwidth]{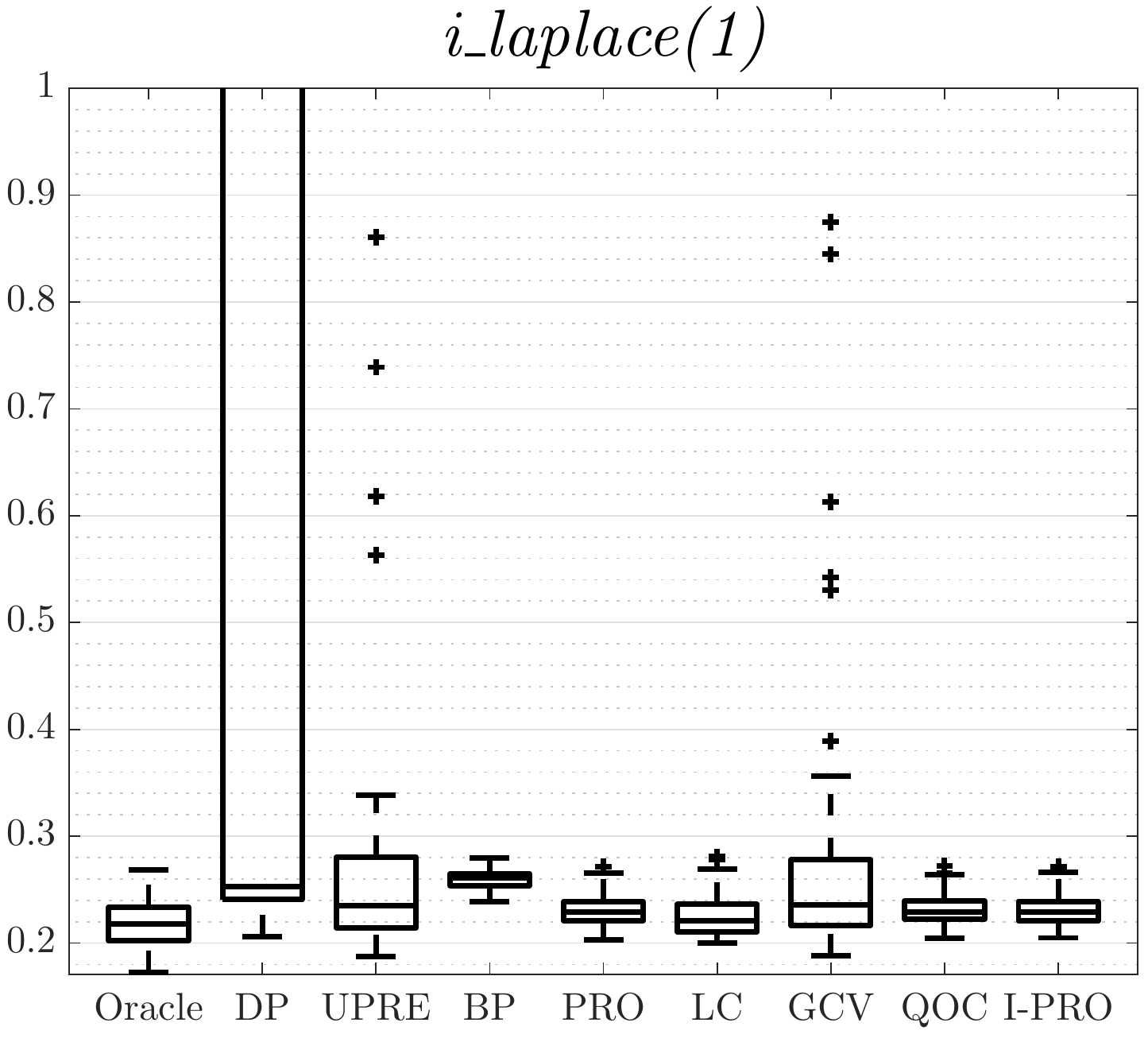}\\[2pt]
\includegraphics[width=0.33\textwidth]{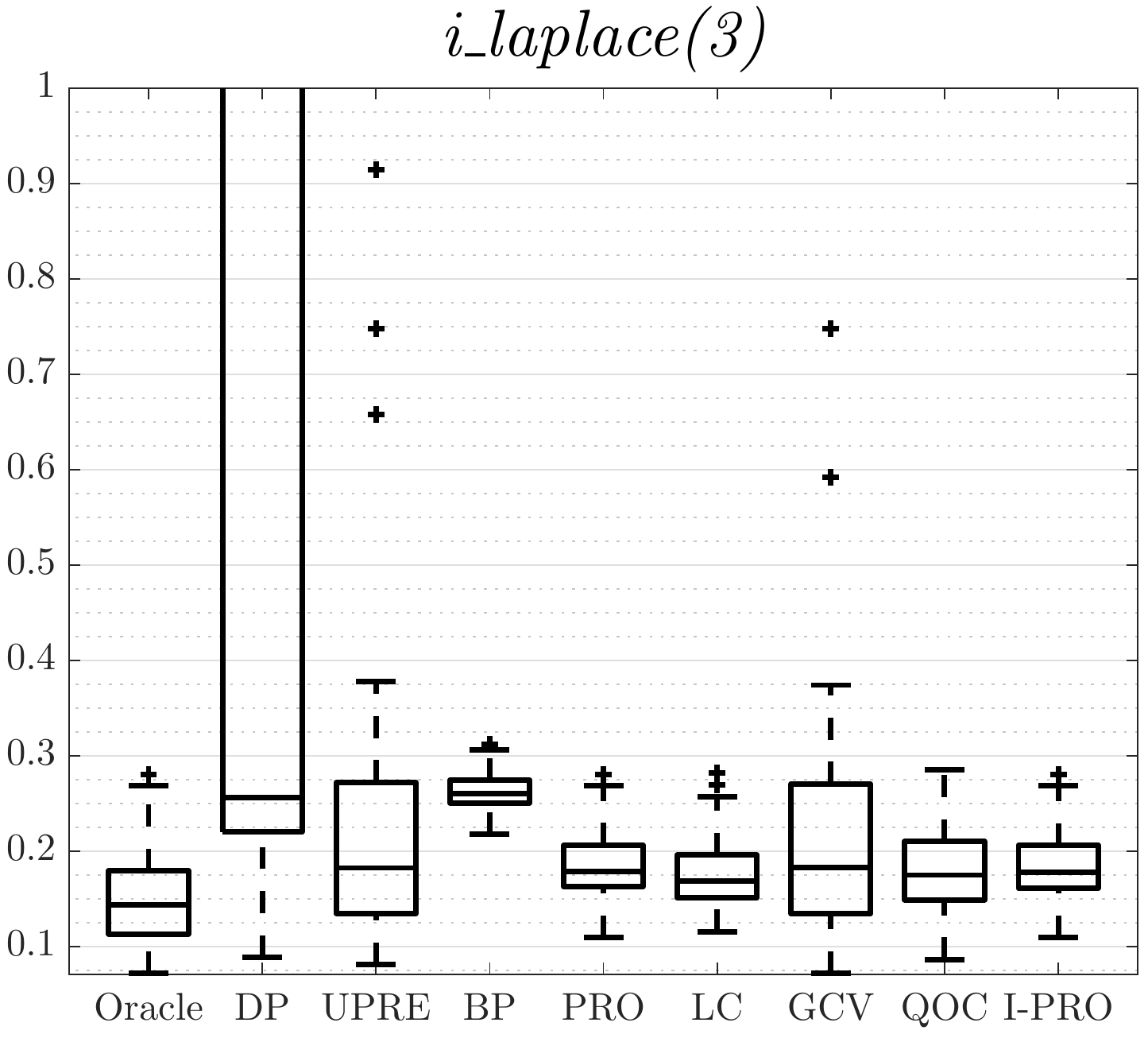}
\includegraphics[width=0.33\textwidth]{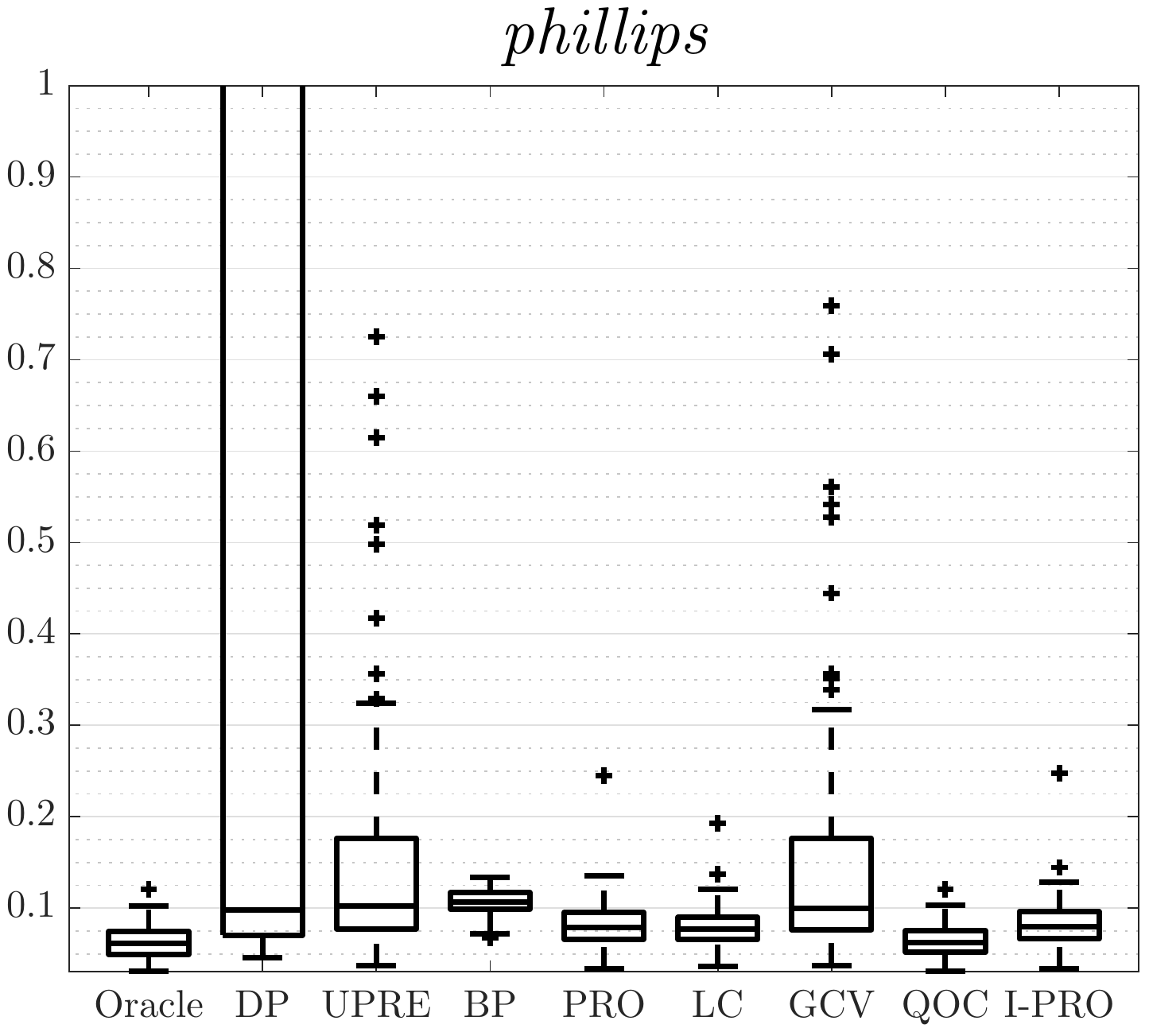}
\includegraphics[width=0.33\textwidth]{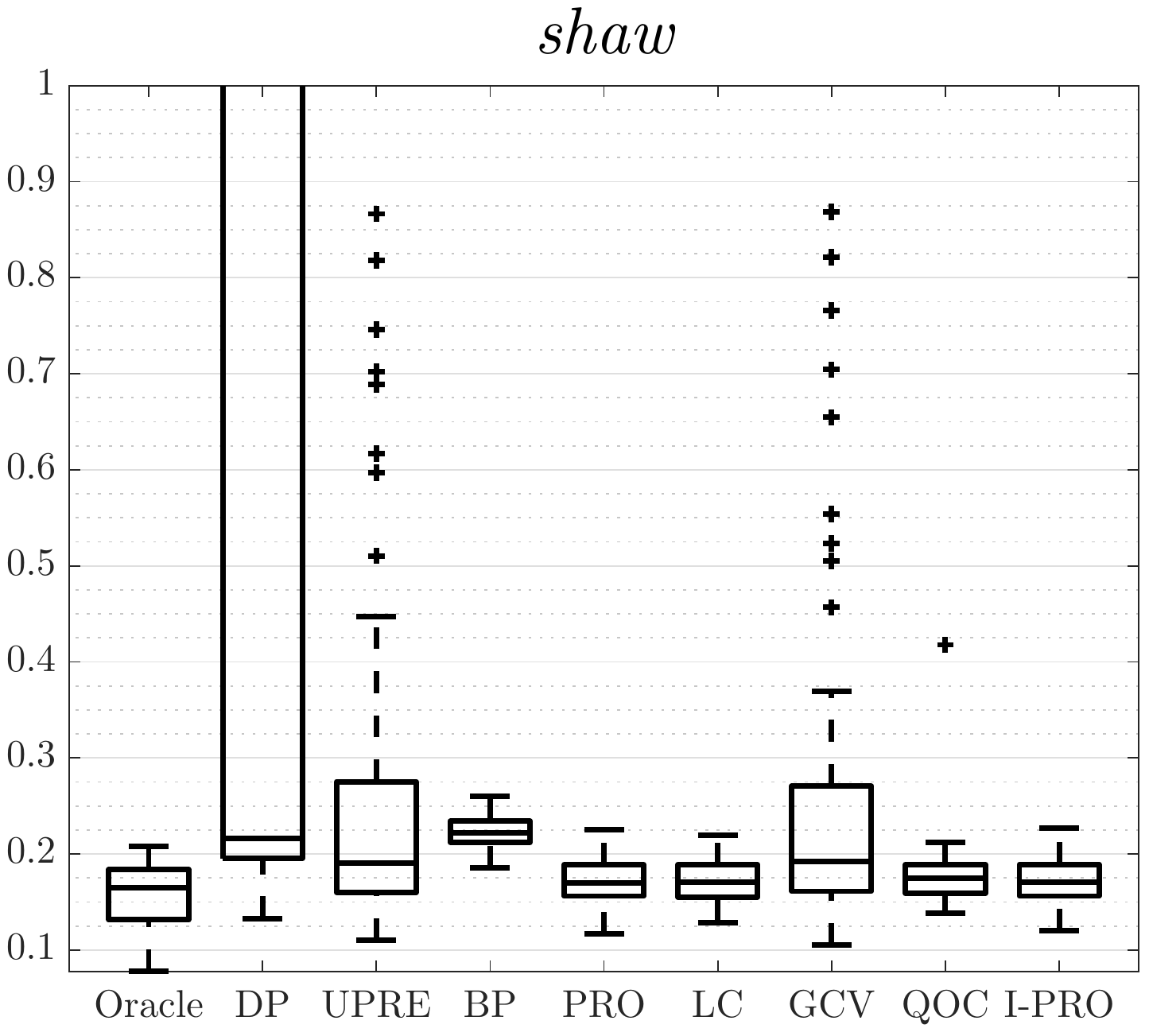}
\caption{Box plots of the $\ell_2$-norm error distributions over 100 samples for nine 1d problems,
with SNR $\xi = 20$ and $n=256$.}\label{figure:boxplots}
\end{figure}

Now we provide numerical insights into the fixed point iteration in the \texttt{I-PRO} algorithm, when
$\sigma$ is unknown. Numerically, it always converges rapidly, and the corresponding fixed point equation
has few fixed points, if not unique globally. (It was proved to be unique within the interval $[0,s_1^2/2]$
in Proposition \ref{lemma}.) However, the convergence of the algorithm is necessary but not
sufficient for ensuring the good performance of \texttt{I-PRO}: the converged solution can potentially
be very bad! Remarkably, \texttt{I-PRO} estimates are always very close to that by \texttt{PRO},
which is consistently observed on all benchmark examples. This is evident by comparing the 7-th column
with the 11-th column of Tables \ref{table:1} and \ref{table:2}: the reconstruction
errors of \texttt{PRO} and \texttt{I-PRO} are nearly identical, except for \textit{heat(5)}. In the
exceptional case, i.e., \textit{heat(5)} with  $\xi=10$ and $n=64$, the condition number is very small
(around $3$) so the problem is actually well posed. In this case, the efficiency of \texttt{PRO} is about $99\%$,
while the efficiency of \texttt{I-PRO} is only about $56\%$. This difference originates from the fixed point
iteration: a close inspection shows that the fixed point iteration converges to an $\alpha$ value very close
to zero. Practically, this seems not a serious restriction, since reconstructions actually do not change
much when varying $\alpha$ near zero, due to the well-conditioned  nature of the specific problem. Although not
proved, the fixed point iteration is observed to work well for all ill-conditioned problems.

\begin{figure}
\setlength{\tabcolsep}{2pt}
\begin{tabular}{cccc}
\includegraphics[width=0.32\textwidth]{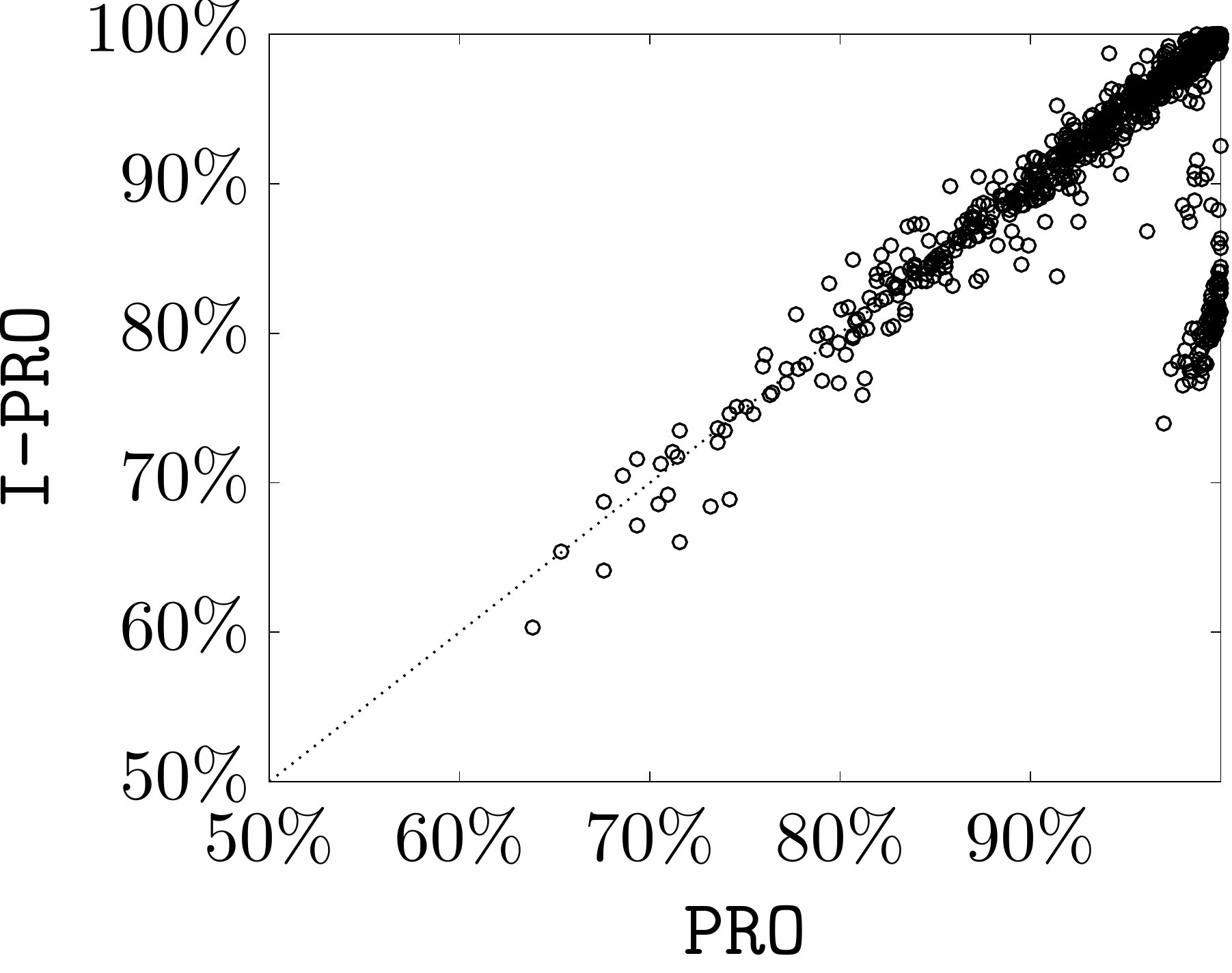}&
\includegraphics[width=0.32\textwidth]{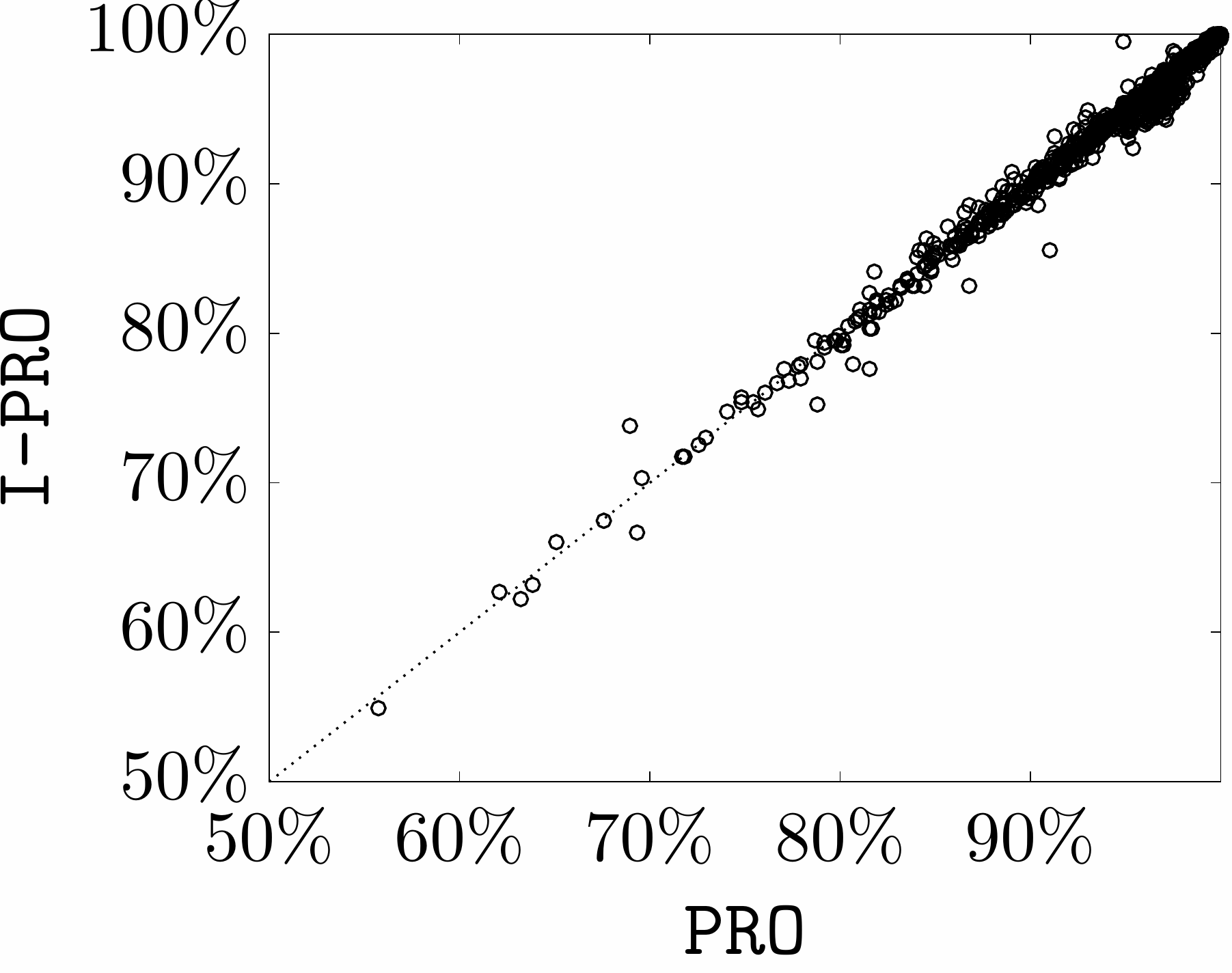}&
\includegraphics[width=0.32\textwidth]{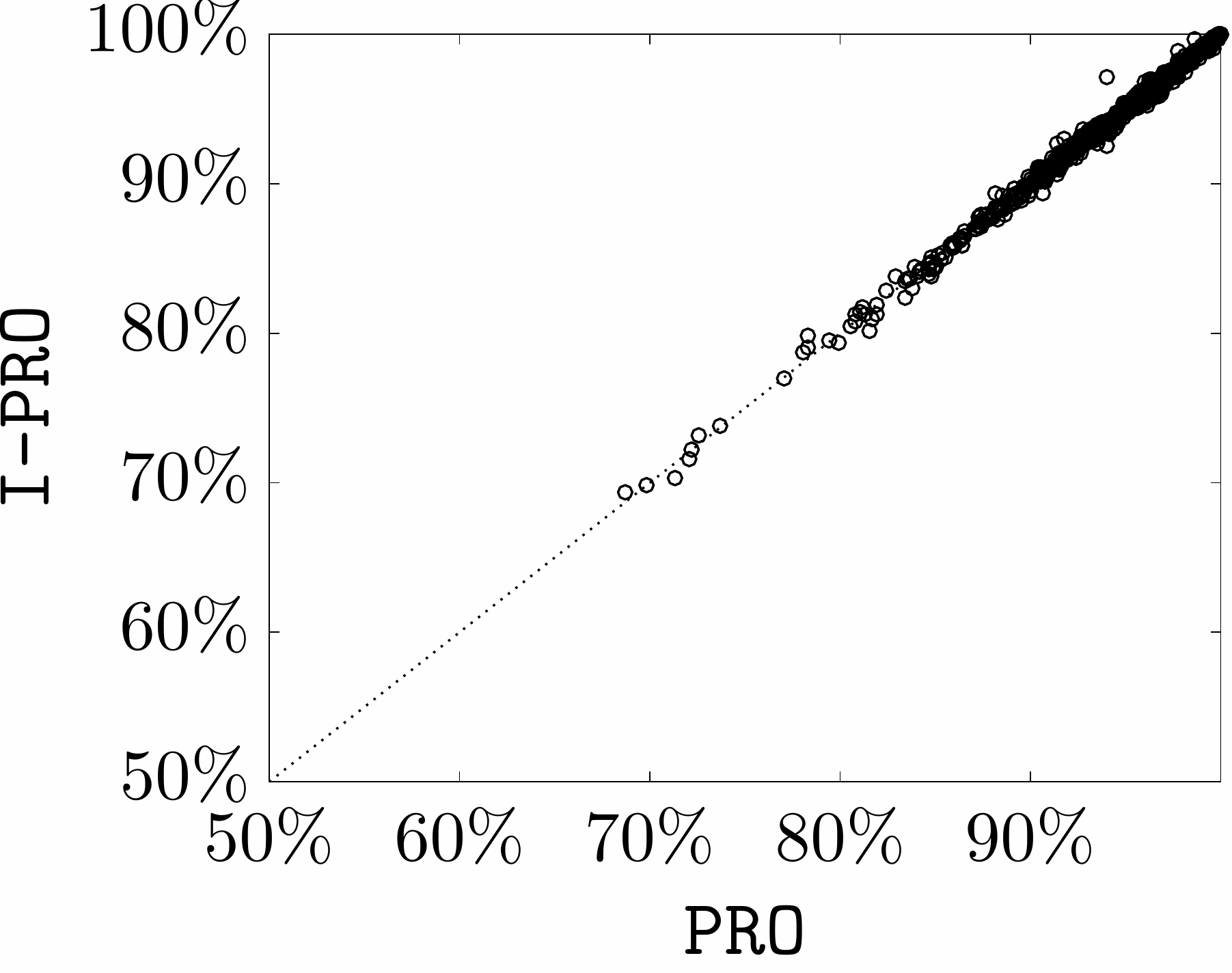}\\[4pt]
\includegraphics[width=0.32\textwidth]{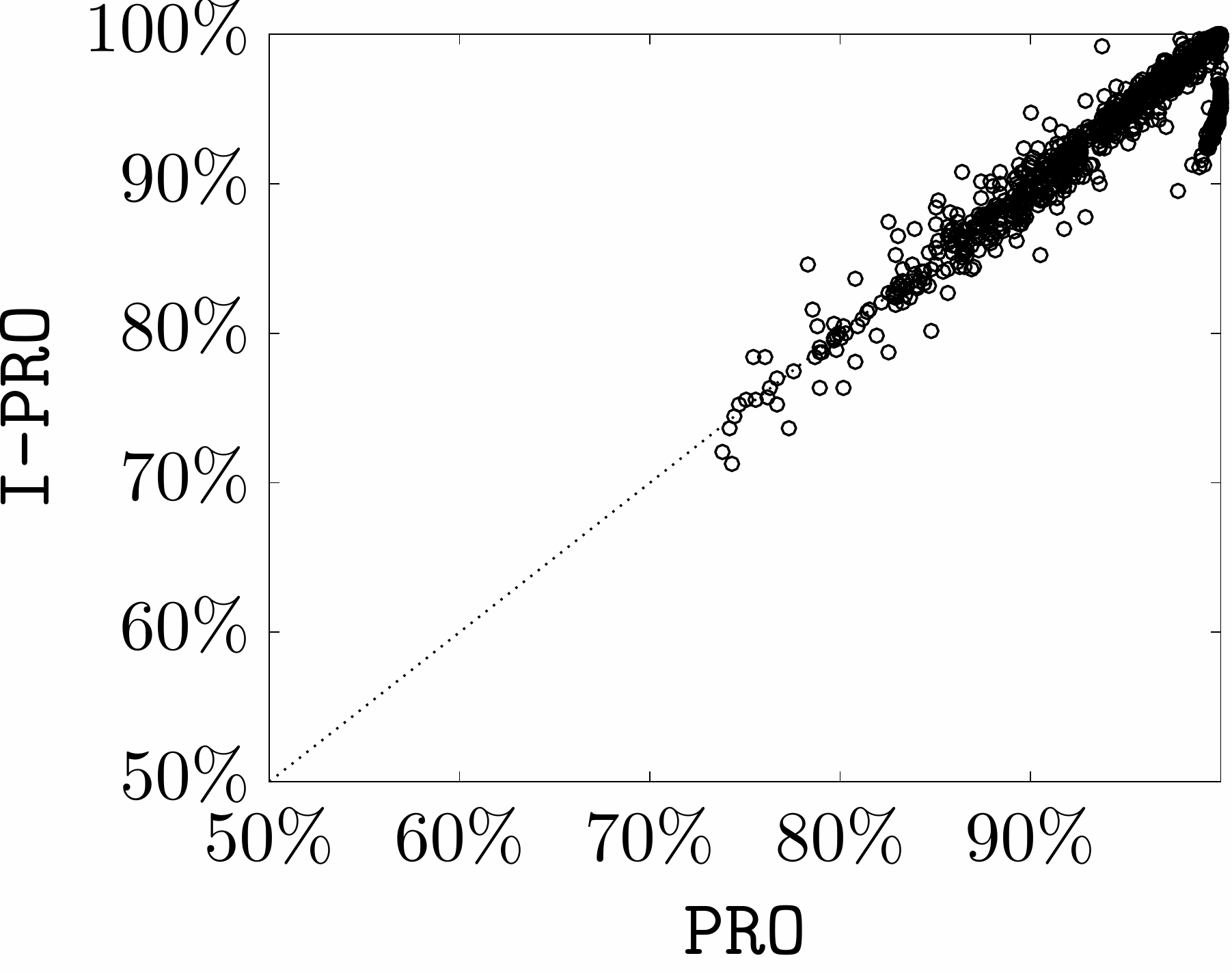}&
\includegraphics[width=0.32\textwidth]{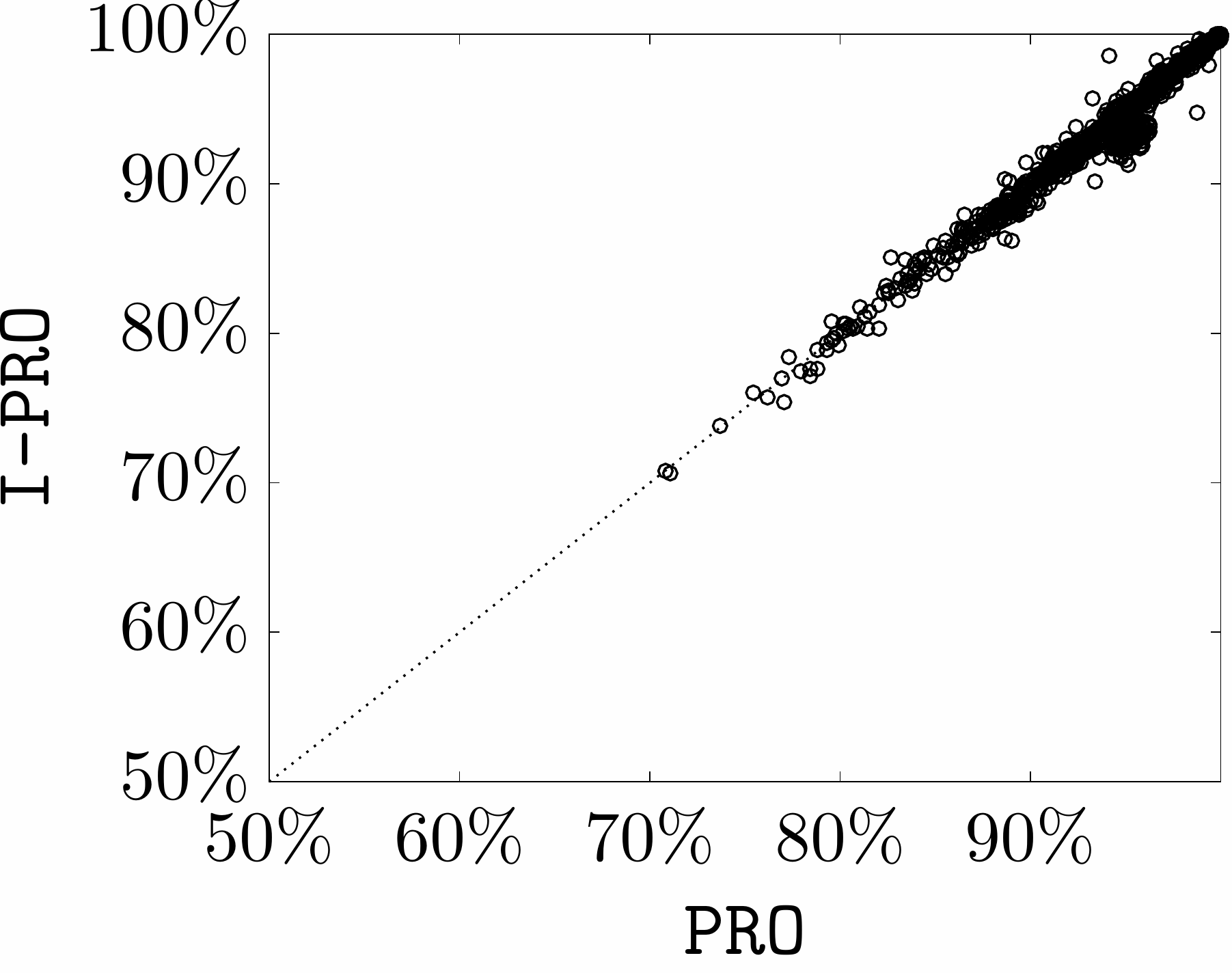}&
\includegraphics[width=0.32\textwidth]{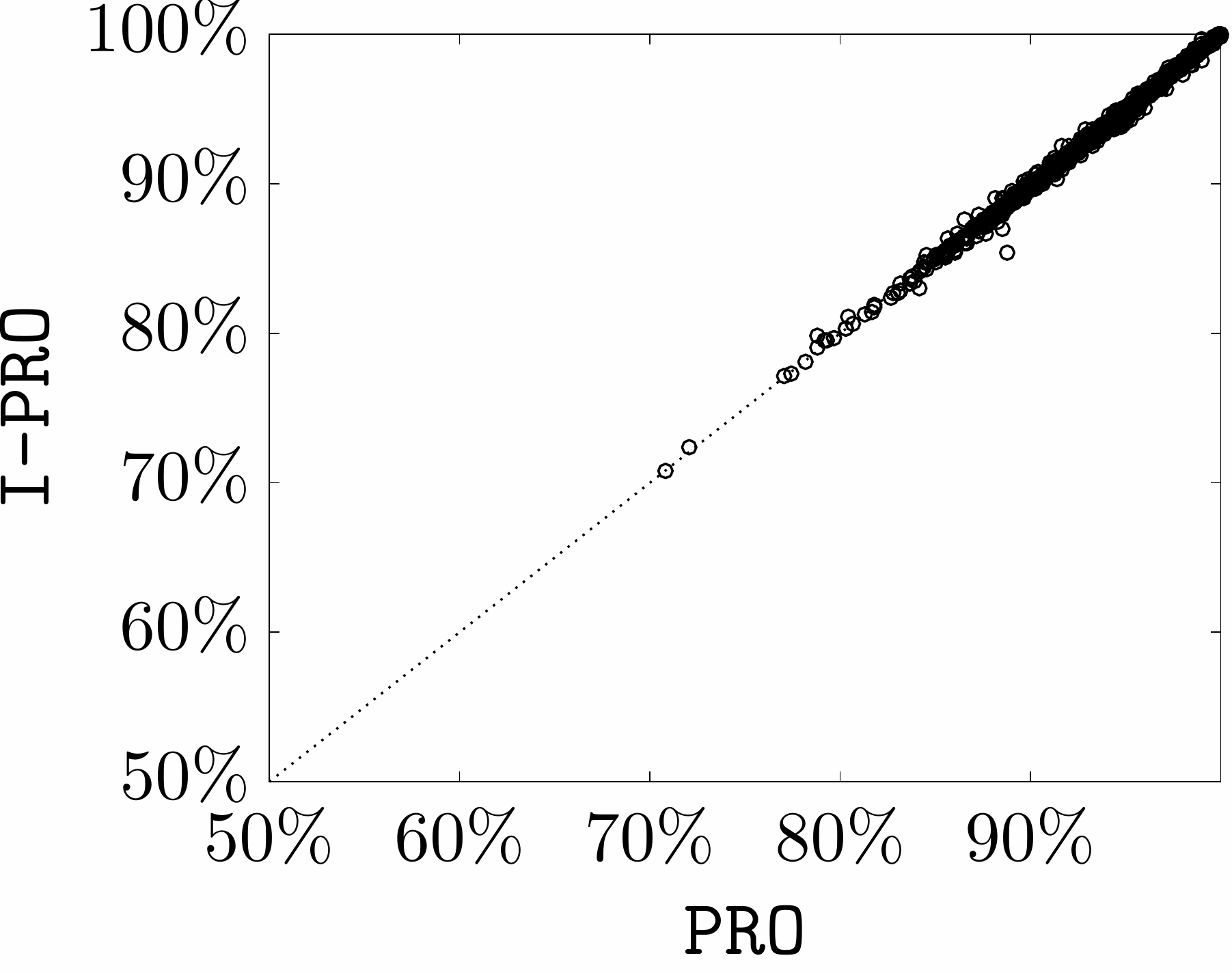}\\[4pt]
\includegraphics[width=0.32\textwidth]{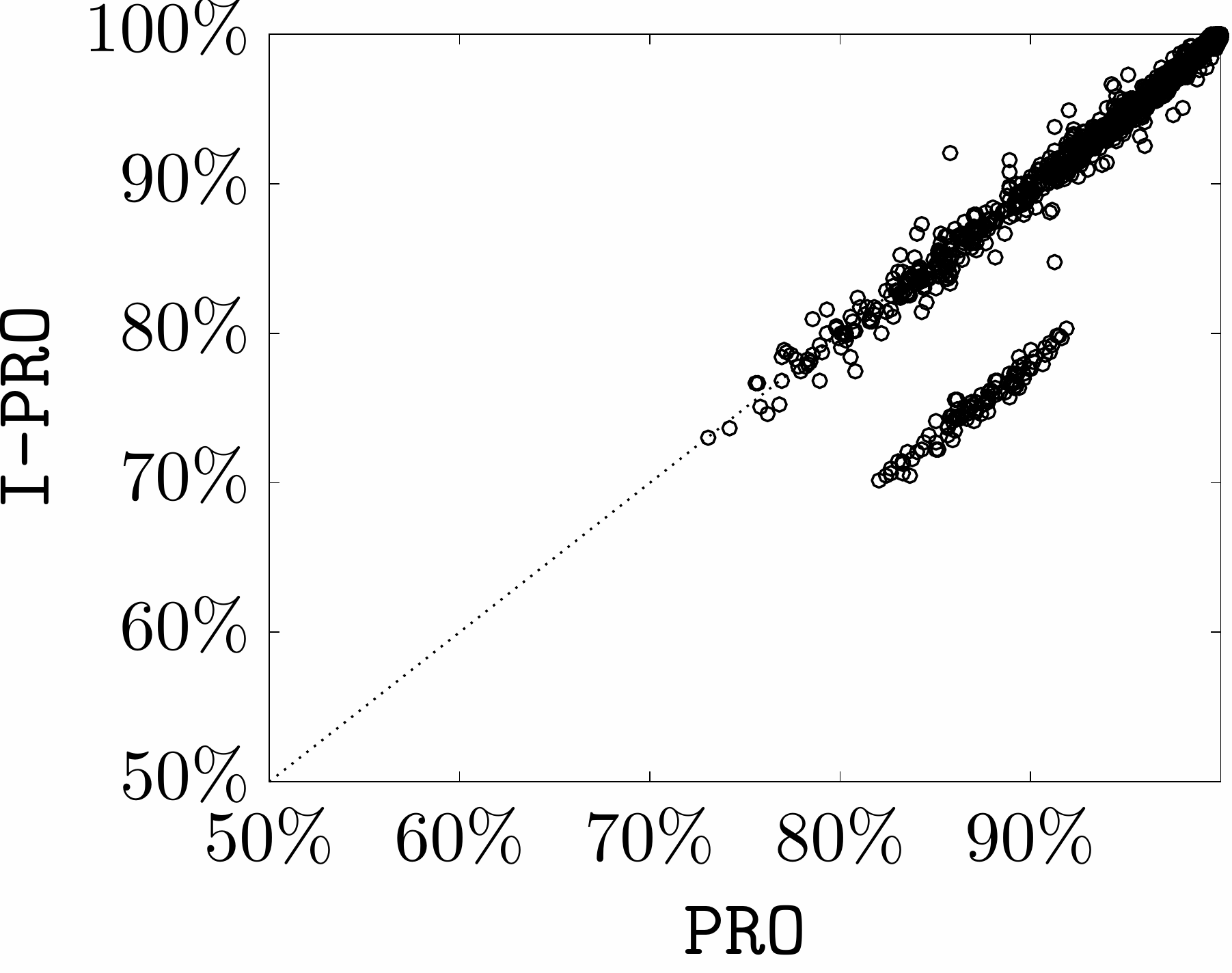}&
\includegraphics[width=0.32\textwidth]{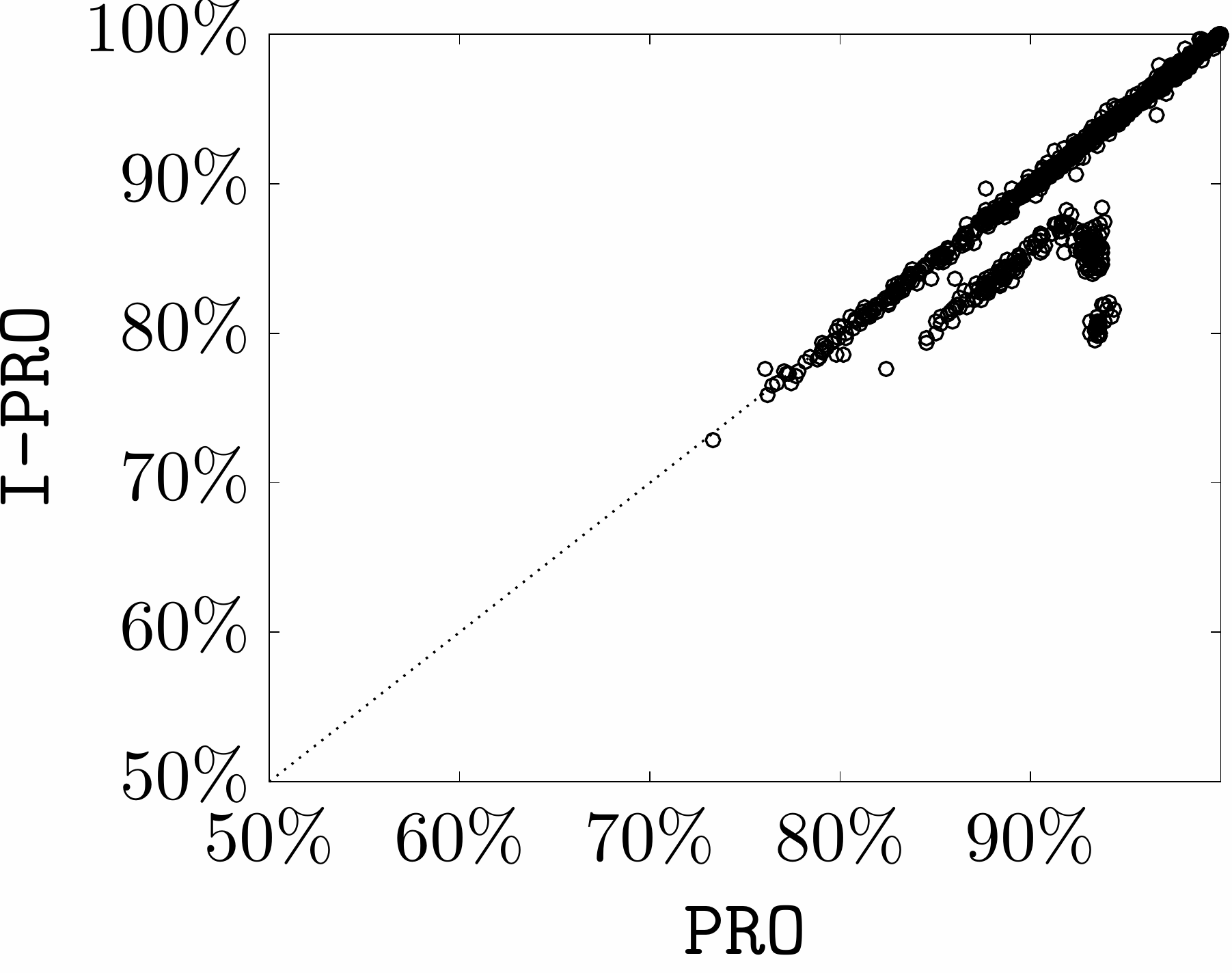}&
\includegraphics[width=0.32\textwidth]{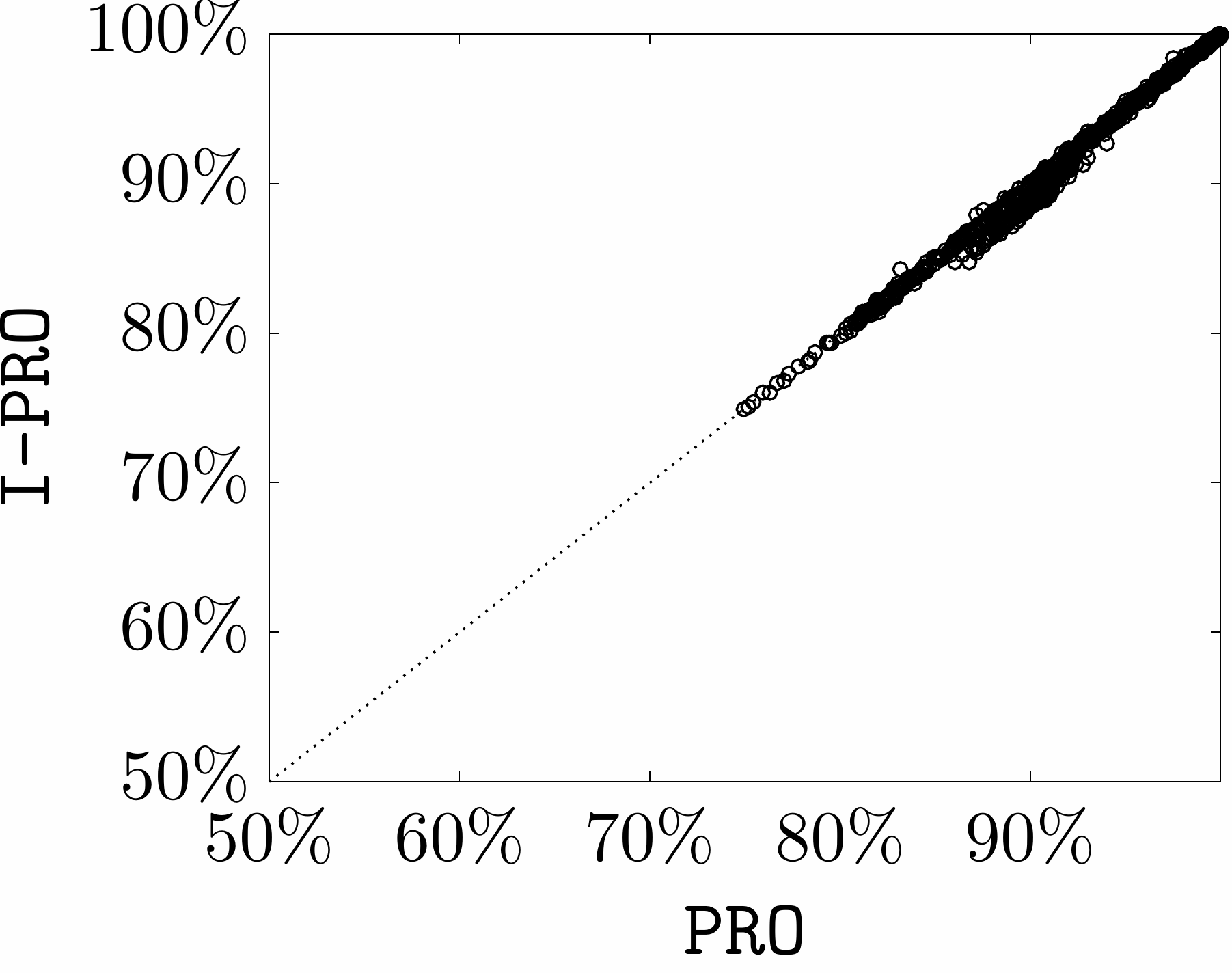}\\[4pt]
(a) $n=64$ & (b) $n=256$ & (c) $n=1024$
\end{tabular}
\caption{Scatter plots of the efficiency of \texttt{PRO} and \texttt{I-PRO}.
The top, middle and bottom rows are for $\xi=10$, $\xi=20$ and $\xi=40$, respectively.\label{scatter-plots-ipro}}
\end{figure}

In Fig. \ref{scatter-plots-ipro}, we show the scattered plots of the efficiency of \texttt{PRO} and
\texttt{I-PRO}. The \texttt{I-PRO} estimates are mostly close to that of \texttt{PRO}, except outliers for \texttt{I-PRO}
on the well-conditioned problem, i.e., \textit{heat(5)}, which  concurs with the observation from Tables \ref{table:1}
and \ref{table:2}. \texttt{I-PRO} can also provide an estimation on SNR, which is generally very close to the
true values. The precise mechanism is unclear but it essentially underpins the reliability of
\texttt{I-PRO} for parameter choice when $\sigma^2$ is unknown.

\subsection{Large-scale problems}

Now we illustrate the proposed \texttt{PRO} and \texttt{I-PRO} on three tomography examples from public package AIR tools, i.e.,
{\it paralleltomo} (parallel beam tomography), {\it fanbeamtomo} (fan beam tomography) and {\it seismictomo} (seismic tomography)
(available at \url{http://people.compute.dtu.dk/pcha/AIRtoolsII/index.html},
accessed on July 1, 2019). The problem parameters are taken as follows: for {\it paralleltomo}
and {\it fanbeamtomo}, the number of discretization intervals in both dimensions is fixed at $\ell=128$ (i.e., the
domain consists of $\ell^2$ cells), $175$ parallel rays for each angle $\theta\in\{0,1,...,179\}$; For {\it seismictomo},
we consider the same discretization $\ell=128$ with $128$ sources and $175$ receivers (seismographs).
See \cite{Hansen20122167} for a complete description of these problems. For each example, we consider three noise
levels (with SNR $\xi=10,20,30$), each with $50$ realizations $\eta_i$, $i=1,\ldots,50$, i.i.d. Gaussian noise
with mean zero and variance $\sigma^2$. See Fig. \ref{three_data} for exemplary noisy sinograms.
In the implementation, for \texttt{PRO}, we take $100$ $\alpha$ values equidistributed in a
logarithmic scale over the interval $(10^{-8}, 10^{-2})\|A\|^2/2$ (as prescribed by Proposition
\ref{lemma}) to find the minimizer of the function $T_{(\rho^2,\sigma^2)}(\alpha)$ (over the sample points), and finally compute
the reconstructions using LSQR \cite{paige1982lsqr}. Meanwhile, \texttt{I-PRO} implements
the fixed point iteration given in the \texttt{I-PRO} algorithm, initialized at the value $\alpha_1:=\alpha^{(50)}$ (from the samples).
In all experiments below, the algorithm converges in less than 10 iterations.

\begin{figure}[H]
\setlength{\tabcolsep}{2pt}
\begin{tabular}{ccc}
\includegraphics[width=0.32\textwidth]{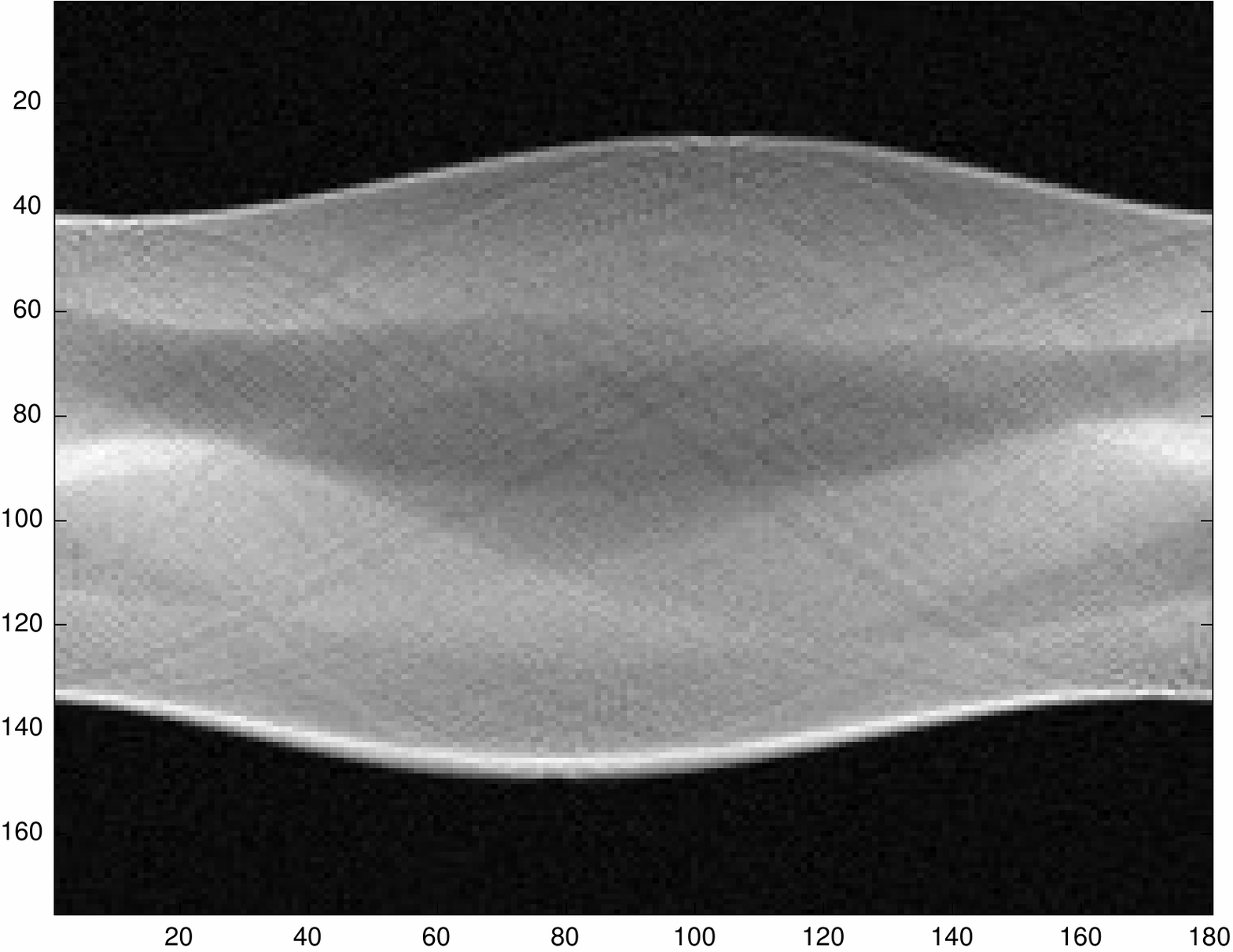}&
\includegraphics[width=0.32\textwidth]{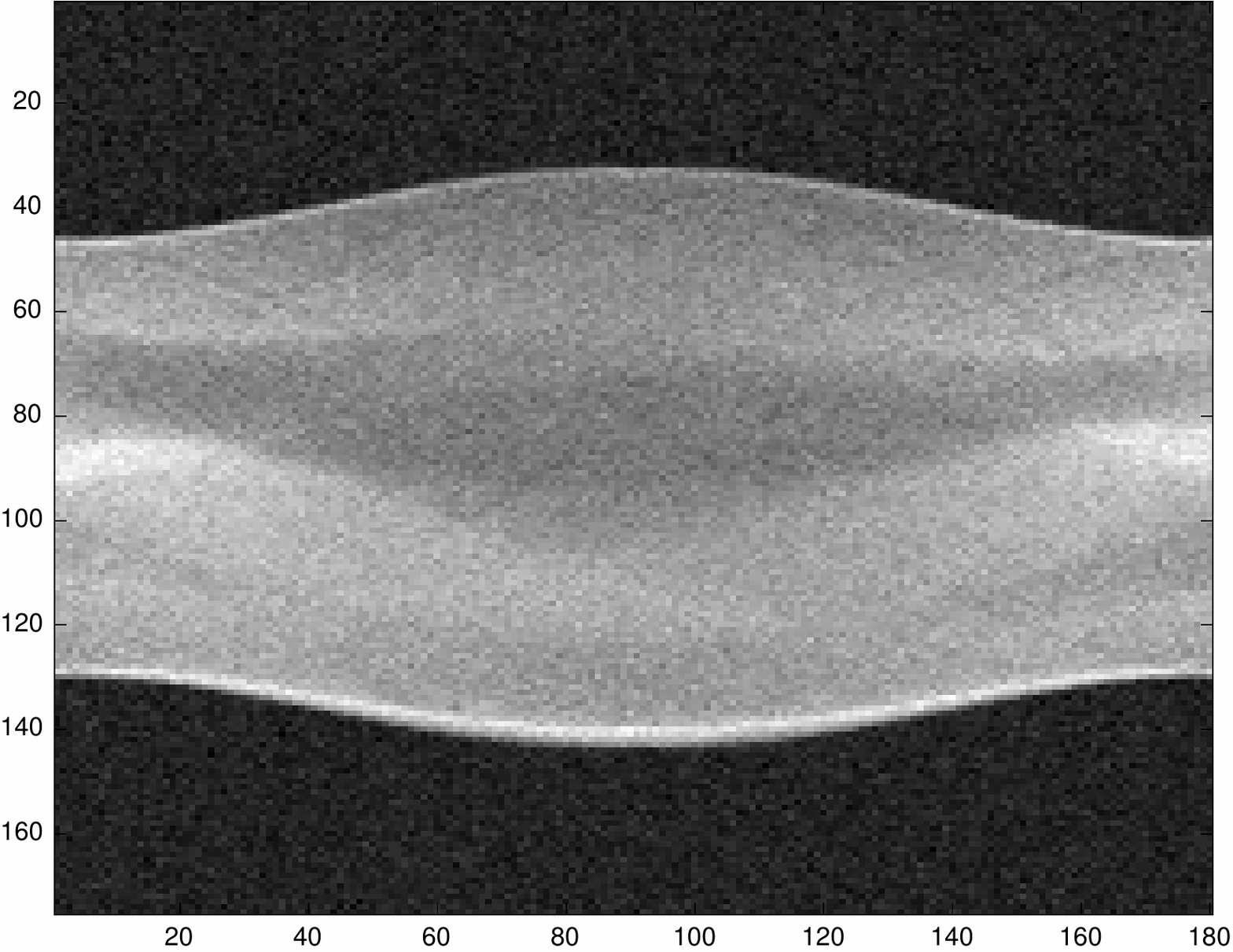}&
\includegraphics[width=0.315\textwidth]{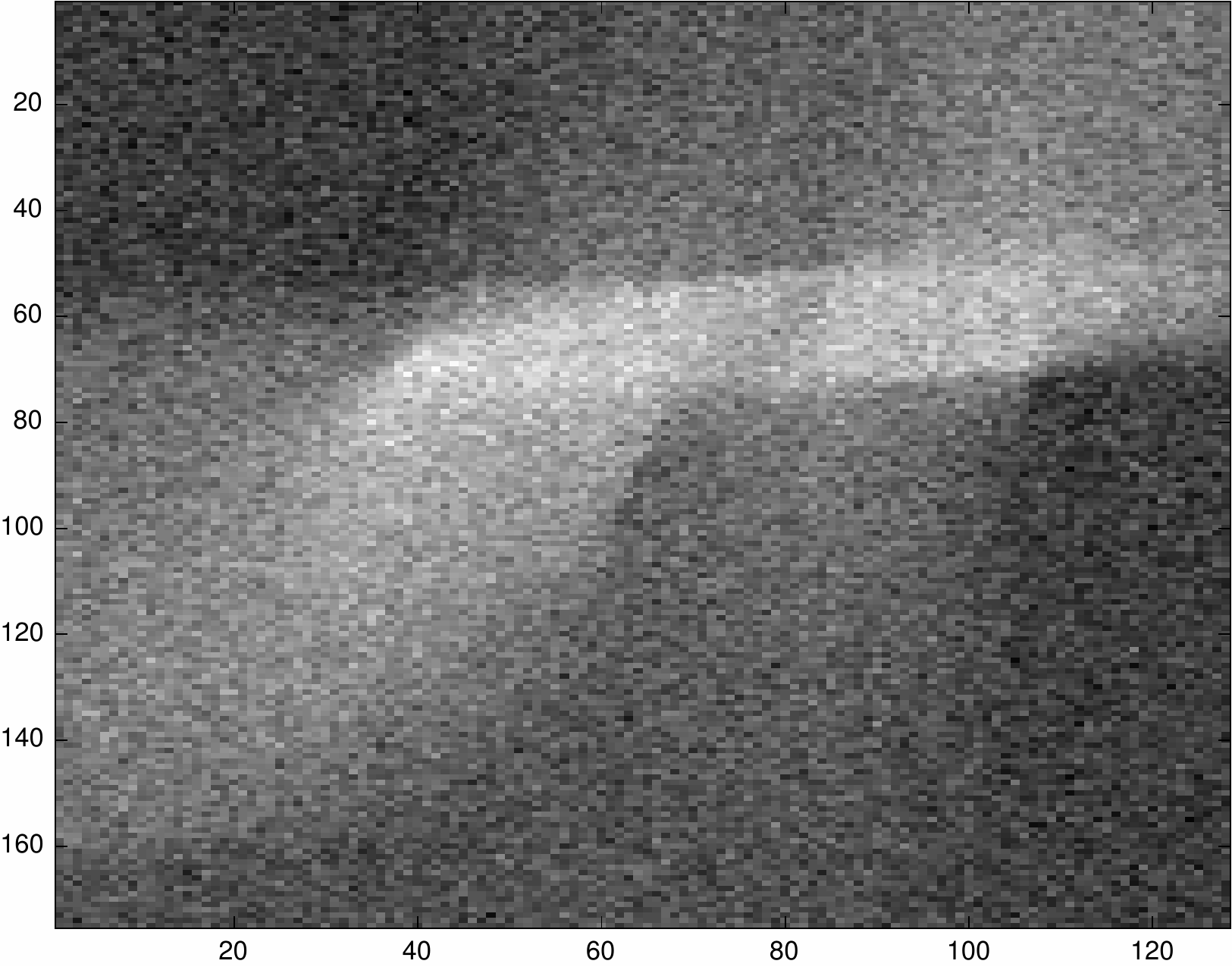}\\
(a) \textit{fanbeamtomo}, $\xi=30$ & (b) \textit{paralleltomo}, $\xi=20$ & (c) \textit{seismictomo}, $\xi=10$
\end{tabular}
\caption{Sinograms of tomography examples at various SNRs. \label{three_data}}
\end{figure}

Since both terms of $T_{(\rho^2,\sigma^2)}(\alpha)$ depend only on the matrix $A$, they can be precomputed.
In Fig. \ref{reg_path_approximation_amplification_error}(a), we show the two terms $s_n(I-X_\alpha)^2$ and
$\|X_\alpha\|_F^2$ as a function of the regularization parameter $\alpha$ for \textit{paralleltomo}, and in Fig.
\ref{reg_path_approximation_amplification_error}(b),  the sequence of four intermediate objective functions
$T_{\hat\rho^2,\hat\sigma^2}(\alpha)$ by the \texttt{I-PRO} algorithm. The iterative process returns an $\alpha$ value close to the \texttt{PRO} value.

\begin{figure}
\centering
\begin{tabular}{cc}
\includegraphics[width=0.44\textwidth]{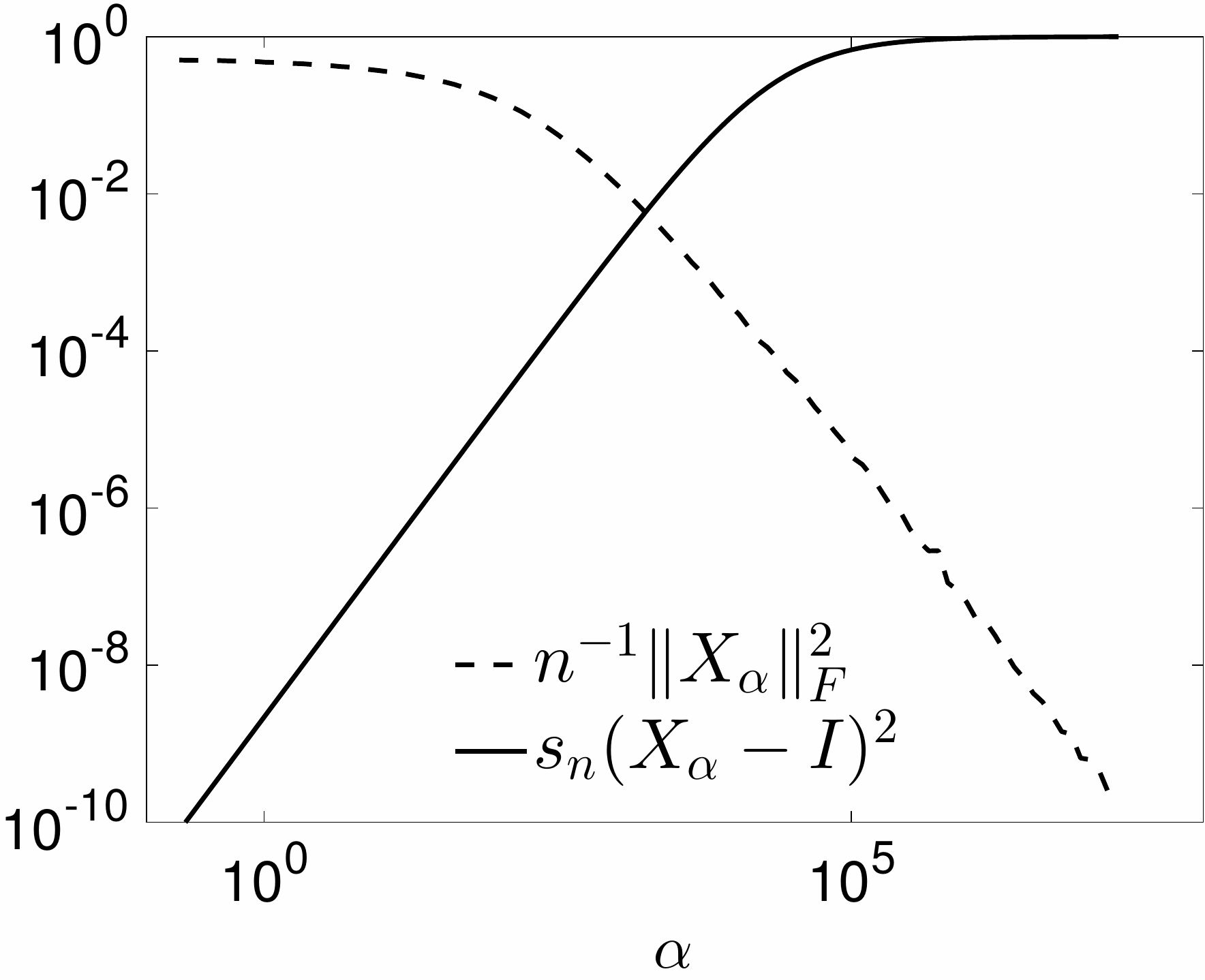}&
\includegraphics[width=0.44\textwidth]{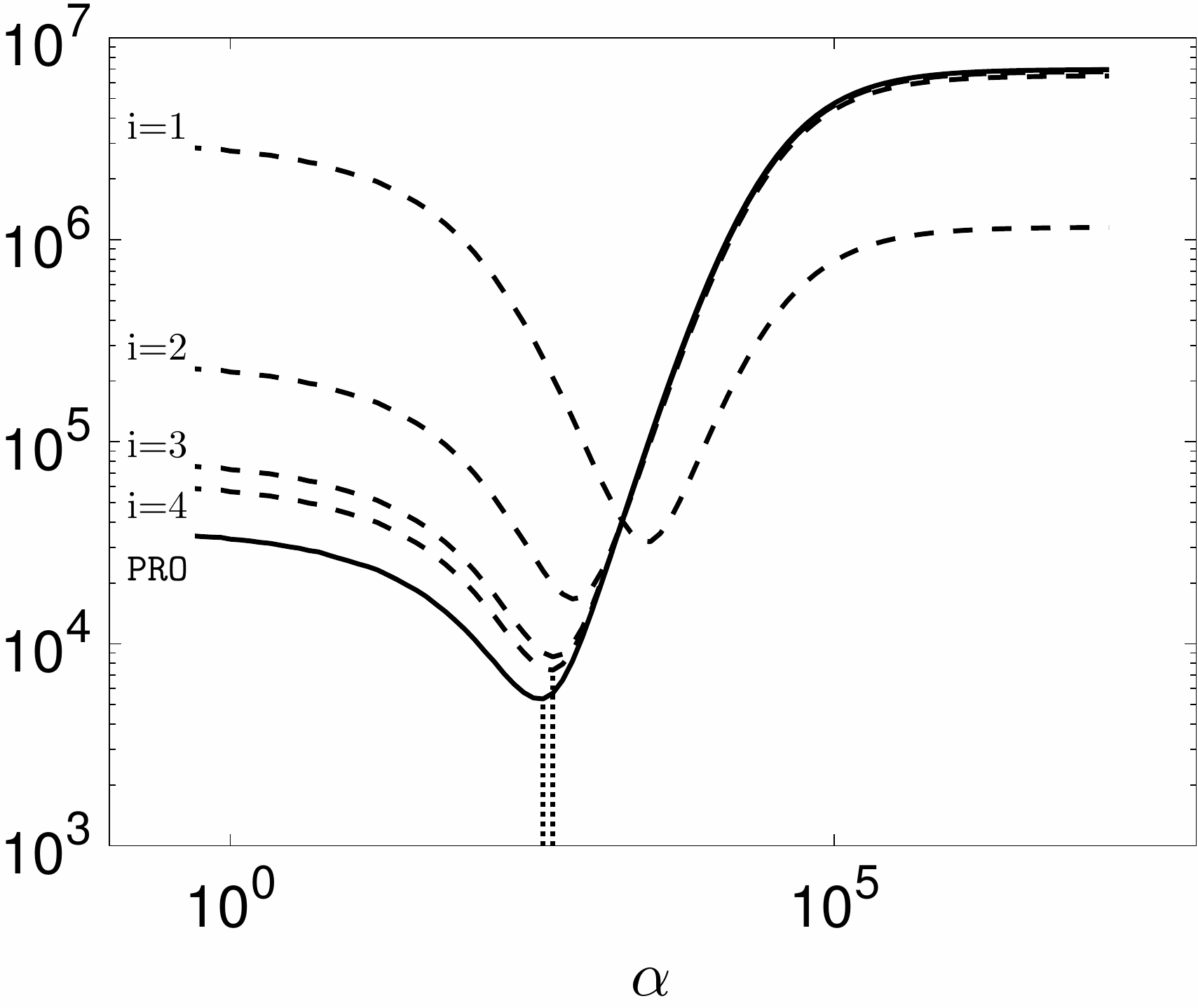}\\
(a) & (b)
\end{tabular}
\caption{(a) The noise amplification error $n^{-1}\|X_\alpha\|_F^2$ and approximation error $s_n(I-X_\alpha)^2$ (lower bound) versus $\alpha$ for \textit{paralleltomo}.
(b): $T_{(\rho^2,\sigma^2)}(\alpha)$ (solid line) and its successive approximations $T_{(\hat{\rho^2},\hat{\sigma^2})}(\alpha)$ with \texttt{I-PRO}. \texttt{I-PRO} converges after four iterations (i.e., $i=4$). The minimizers of these two functions
provide the {\tt PRO} and {\tt I-PRO} solutions showed in the left panels of Fig. \ref{six_reconstructions}.}
\label{reg_path_approximation_amplification_error}
\end{figure}

\begin{figure}[!h]
\includegraphics[width=0.24\textwidth]{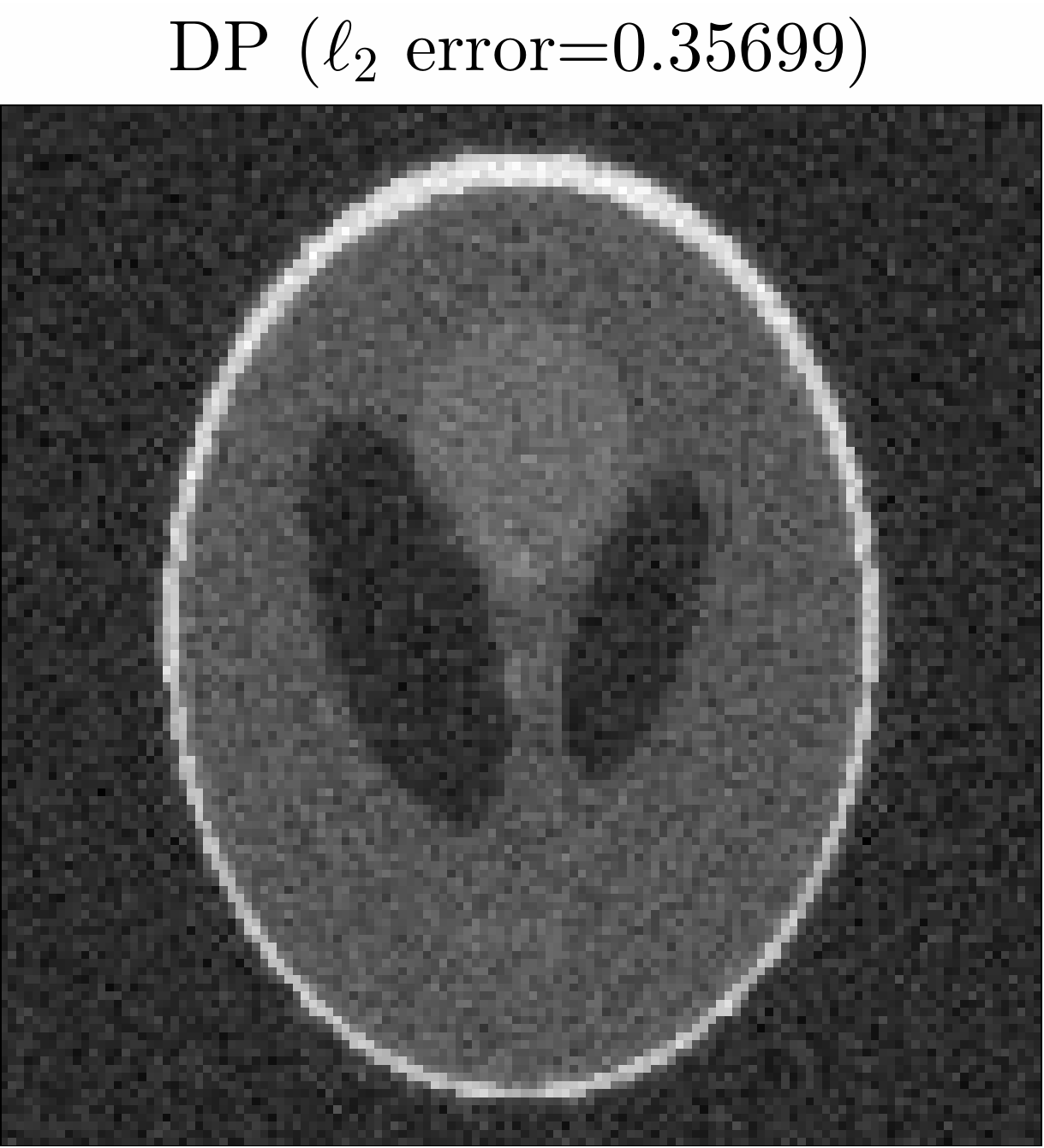}
\includegraphics[width=0.24\textwidth]{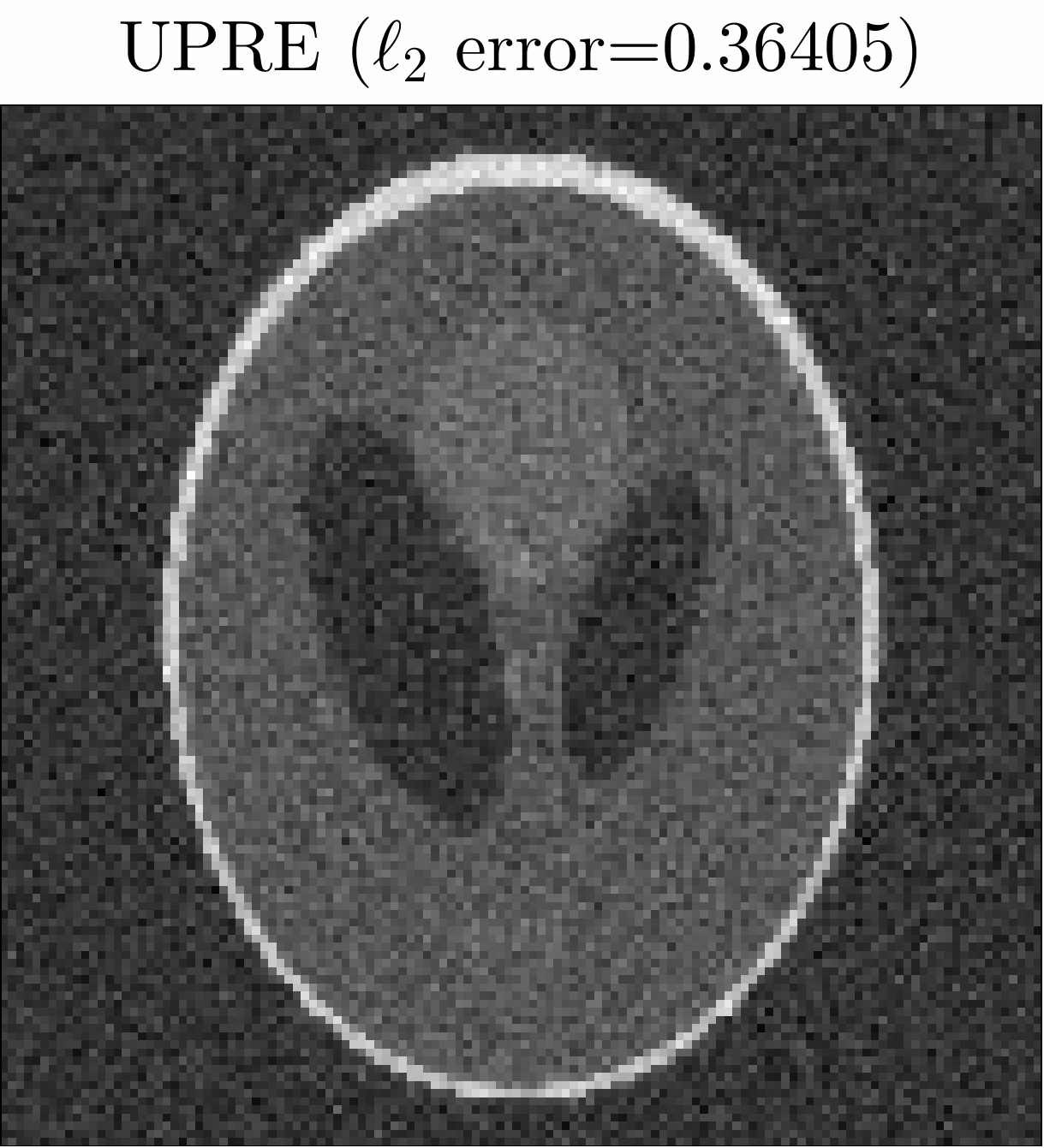}
\includegraphics[width=0.24\textwidth]{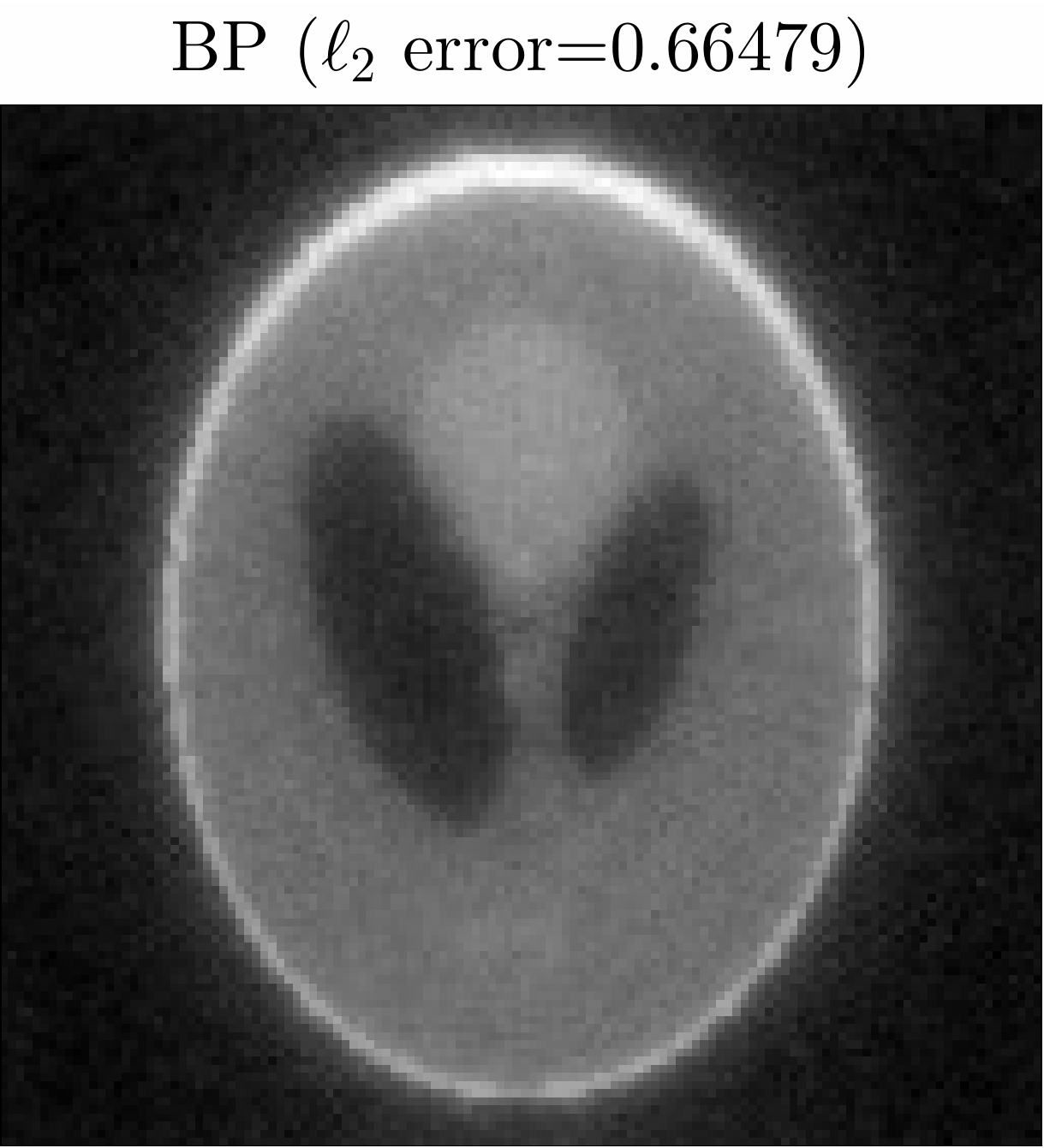}
\includegraphics[width=0.24\textwidth]{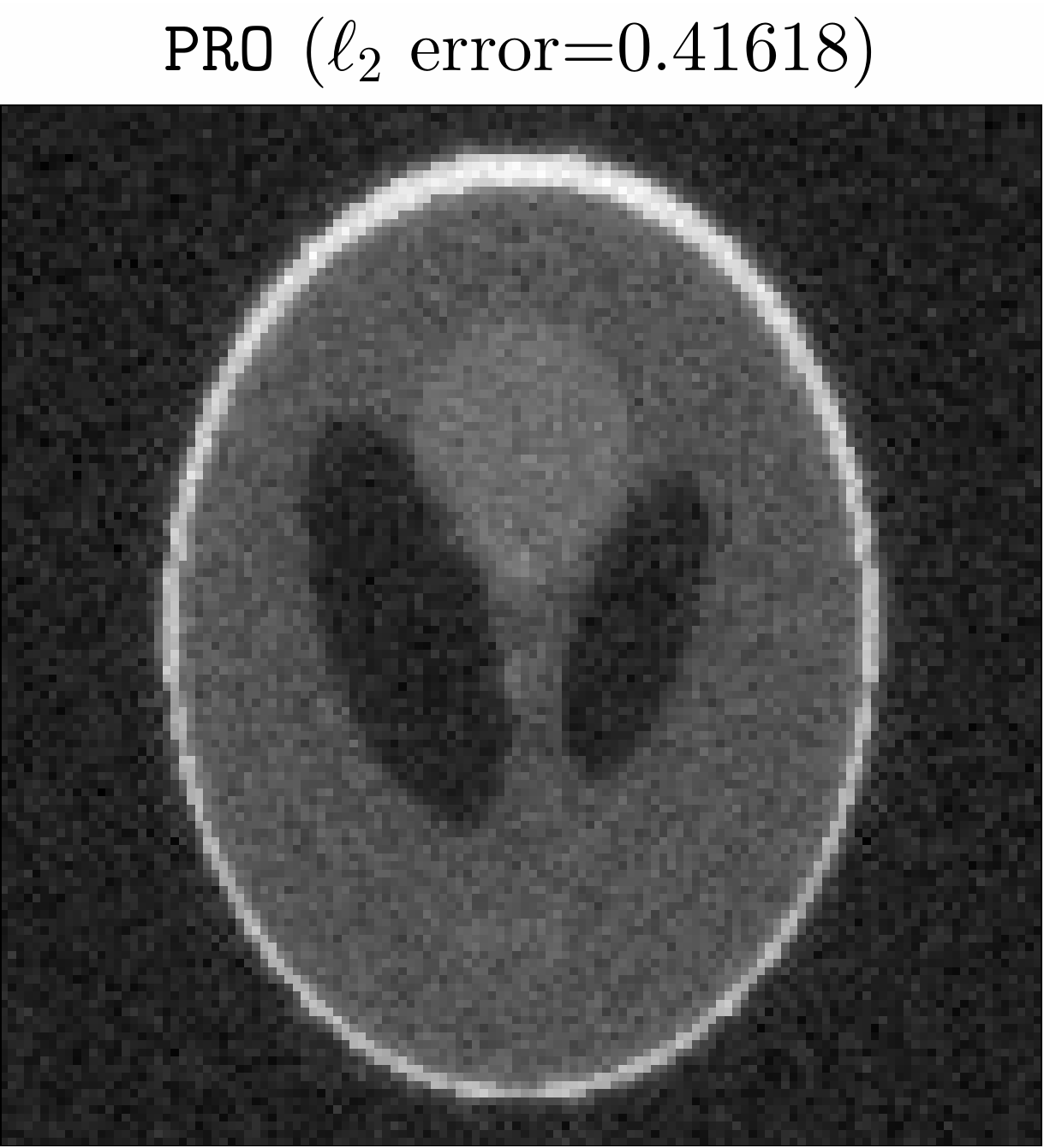}\\
\includegraphics[width=0.24\textwidth]{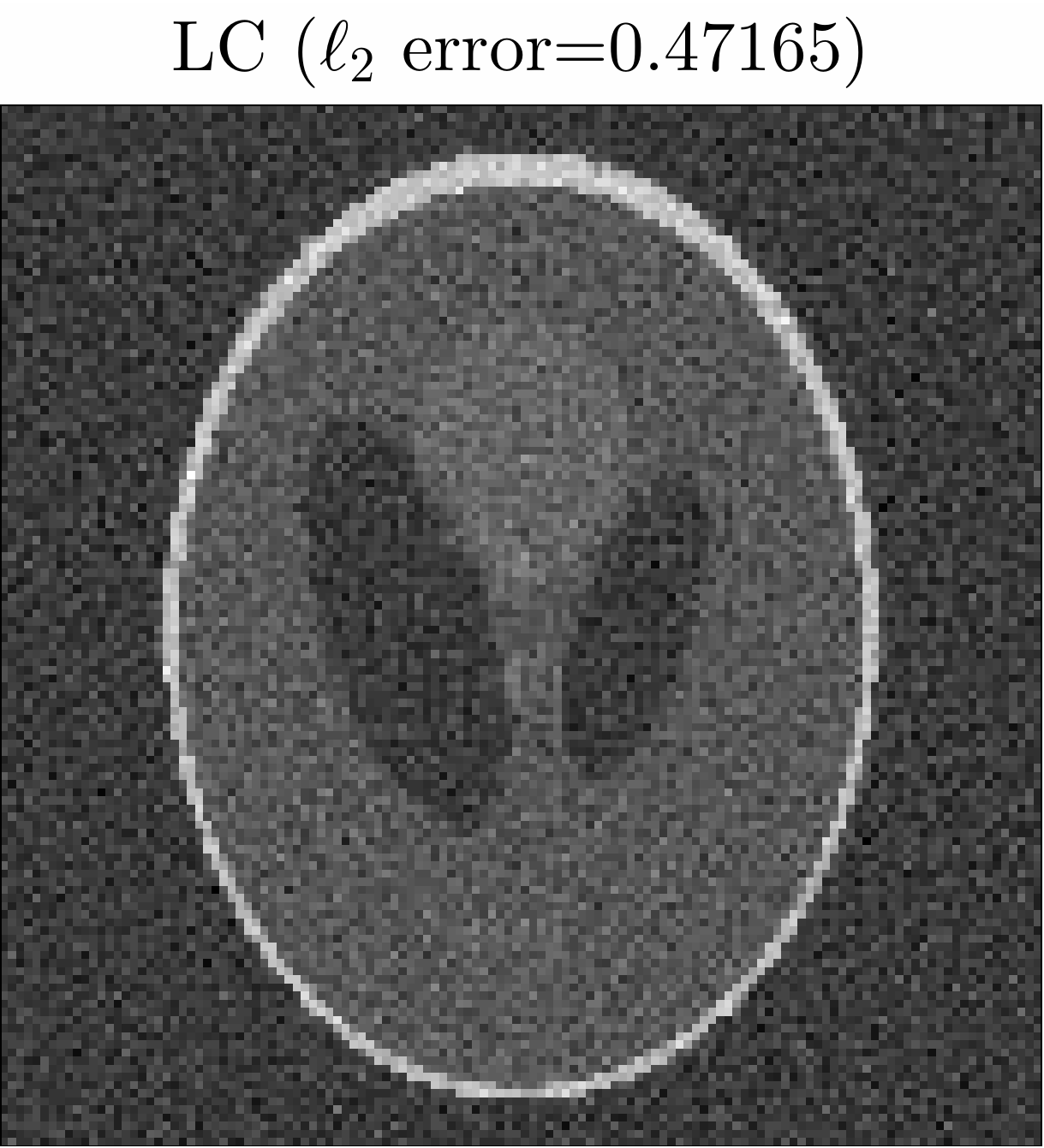}
\includegraphics[width=0.24\textwidth]{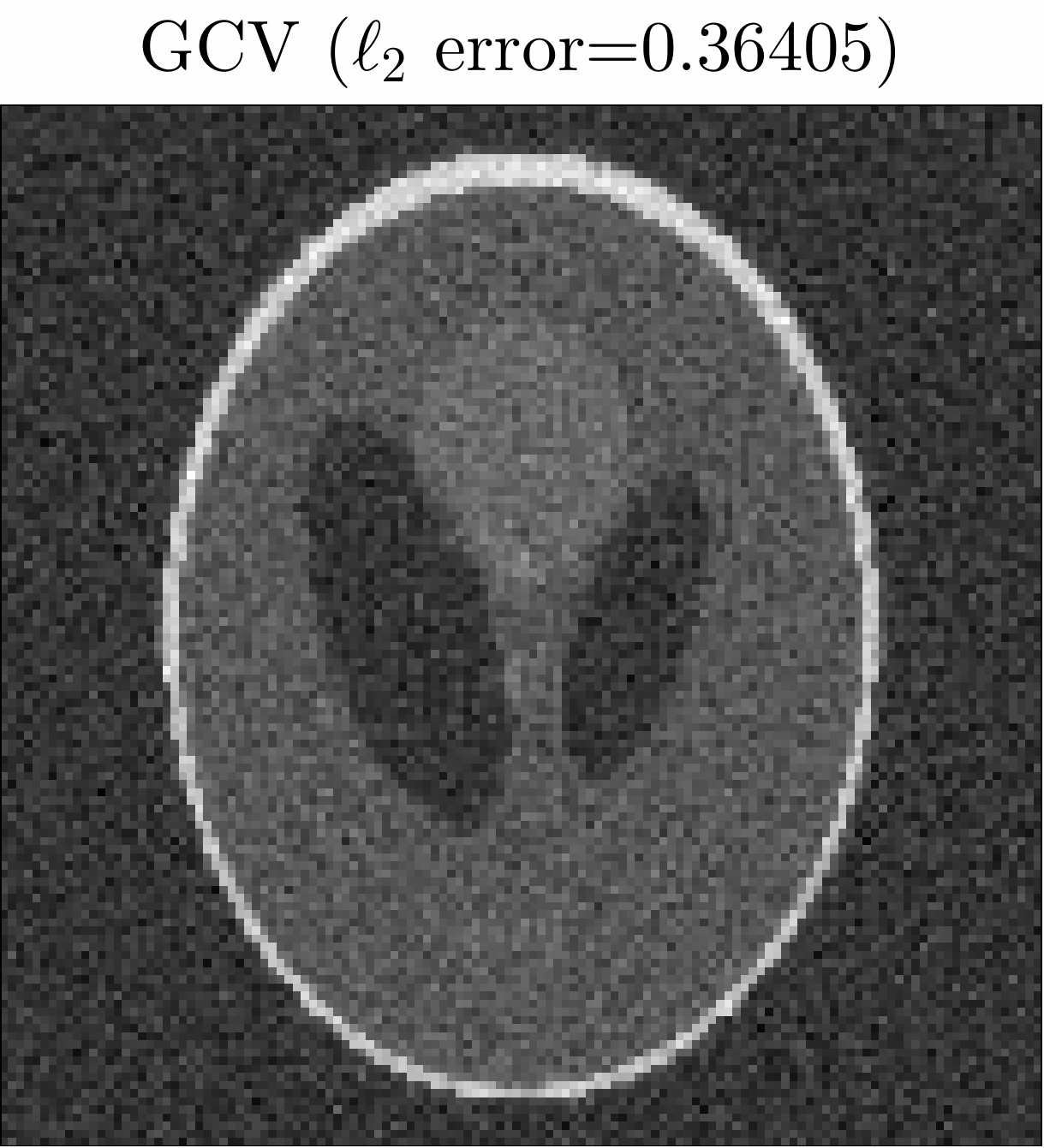}
\includegraphics[width=0.24\textwidth]{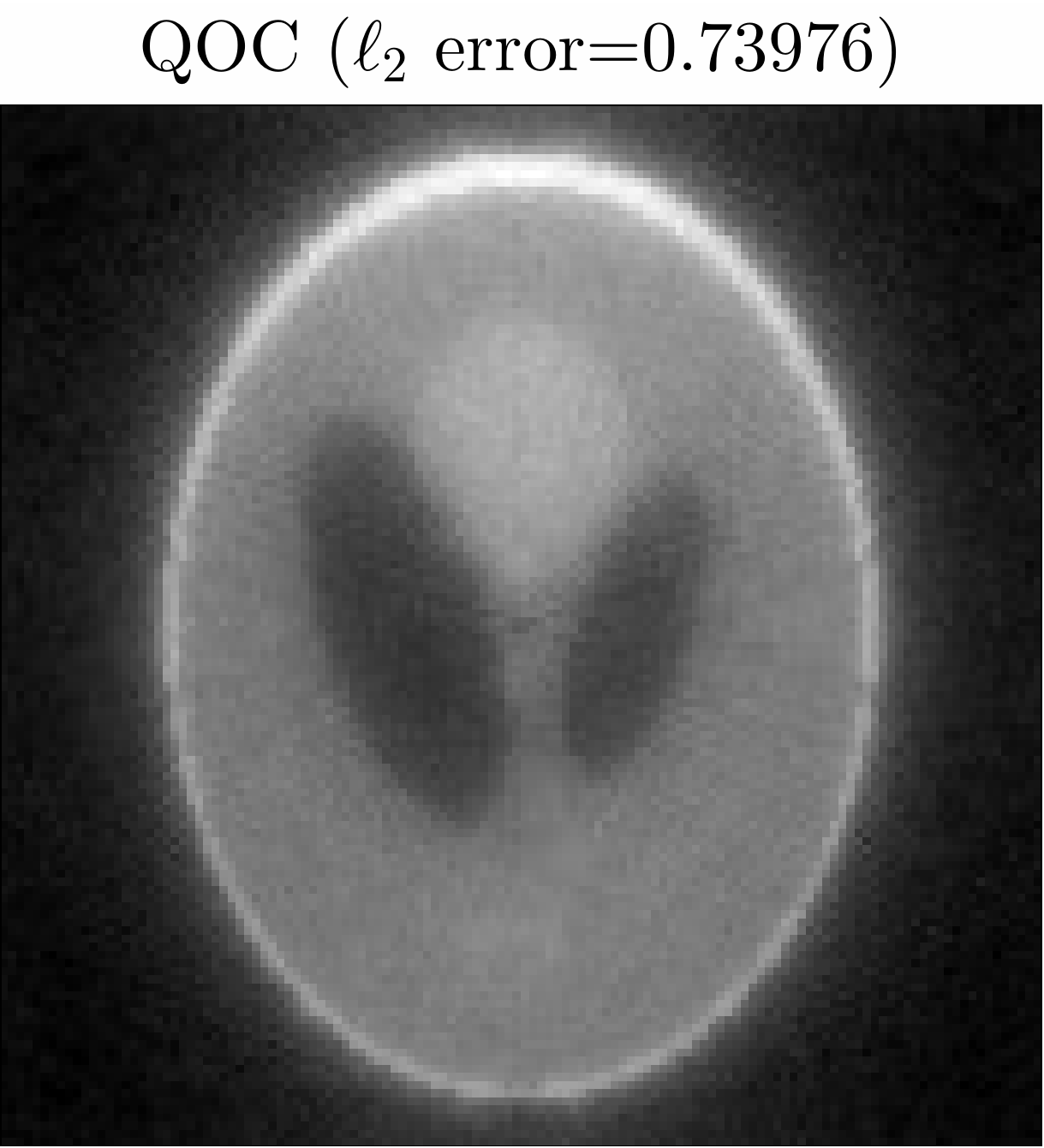}
\includegraphics[width=0.24\textwidth]{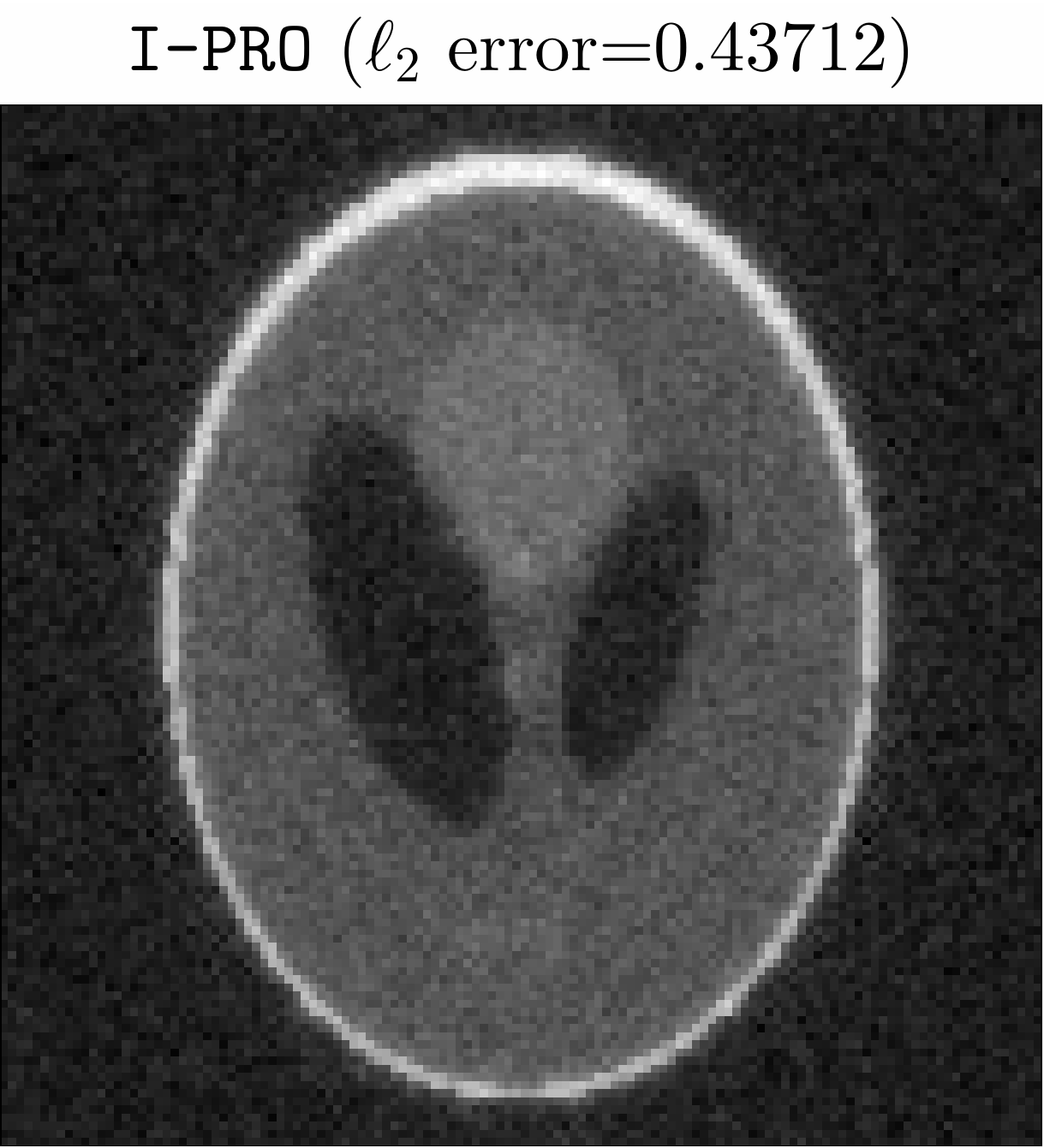}
\caption{Reconstructions for \textit{paralleltomo}, for the noisy sinogram showed in Fig. \ref{three_data}(b). \label{six_reconstructions}}
\end{figure}

Due to the large number of cells ($\ell^2 \simeq 1.6\times10^4$), the variance of the reconstructions with respect to
the noise realization is very small. In Table \ref{large-scale-table}, we show the efficiency of the methods.
For all three examples, all eight rules can choose a suitable regularization parameter,
and in terms of efficiency, \texttt{PRO} and \texttt{I-PRO} perform better than existing methods when SNR
is low and slightly worse when SNR is high, for {\it paralleltomo} and {\it fanbeamtomo}.
This is attributed to the fact that by construction, \texttt{PRO} and \texttt{I-PRO} provide conservative
estimates of the optimal parameter (with respect to predictive risk). In Fig. \ref{six_reconstructions}, we
show exemplary reconstructions for \textit{paralleltomo}.

\begin{table}
\centering
\resizebox{\columnwidth}{!}{%
\begin{tabular}{ l c | r  rrrr rrrr }
\toprule
problem &  stat & \multicolumn{1}{c}{$\ell_2$ error}  &  \multicolumn{4}{c}{$\sigma$ known, $\rho$ unknown}  & \multicolumn{4}{c}{$\sigma$ and $\rho$ unknowns} \\
\cmidrule(lr){3-3} \cmidrule(lr){4-7} \cmidrule(lr){8-11}
 & $\xi$ & \multicolumn{1}{c}{Oracle} & \multicolumn{1}{c}{DP} &  \multicolumn{1}{c}{UPRE} & \multicolumn{1}{c}{BP} &  \multicolumn{1}{c}{\tt{PRO}}  &  \multicolumn{1}{c}{{LC}}  & \multicolumn{1}{c}{GCV} &  \multicolumn{1}{c}{QOC}  & \multicolumn{1}{c}{\tt{I-PRO}} \\
\midrule
seismic & $10$ & 0.35 & 99.8$\%$,  & 74.5$\%$,  & 41.2$\%$,  & 93.5$\%$,  & 87.9$\%$,  & 74.5$\%$,  & 93.2$\%$,  & 93.5$\%$\\
& $20$ & 0.24 & 99.4$\%$,  & 80.0$\%$,  & 74.8$\%$,  & 99.4$\%$,  & 48.5$\%$,  & 80.0$\%$,  & 99.1$\%$,  & 99.4$\%$\\
& $30$ & 0.17 & 100.0$\%$,  & 90.9$\%$,  & 93.9$\%$,  & 99.6$\%$,  & 43.5$\%$,  & 84.1$\%$,  & 99.6$\%$,  & 99.6$\%$\\
\midrule
fanbeam & $10$ & 0.55 & 99.8$\%$,  & 88.6$\%$,  & 89.2$\%$,  & 98.7$\%$,  & 99.0$\%$,  & 88.6$\%$,  & 72.3$\%$,  & 98.7$\%$\\
& $20$ & 0.37 & 98.8$\%$,  & 94.9$\%$,  & 55.9$\%$,  & 87.2$\%$,  & 75.2$\%$,  & 94.9$\%$,  & 51.6$\%$,  & 83.5$\%$\\
& $30$ & 0.24 & 96.6$\%$,  & 97.8$\%$,  & 33.9$\%$,  & 76.9$\%$,  & 82.9$\%$,  & 97.8$\%$,  & 96.5$\%$,  & 61.5$\%$\\
\midrule
parallel & $10$ & 0.54 & 99.4$\%$,  & 90.0$\%$,  & 90.2$\%$,  & 97.9$\%$,  & 99.4$\%$,  & 90.0$\%$,  & 70.5$\%$,  & 97.9$\%$\\
& $20$ & 0.34 & 97.4$\%$,  & 95.8$\%$,  & 51.9$\%$,  & 83.3$\%$,  & 73.8$\%$,  & 95.8$\%$,  & 46.6$\%$,  & 79.2$\%$\\
& $30$ & 0.19 & 94.1$\%$,  & 93.9$\%$,  & 27.0$\%$,  & 68.2$\%$,  & 71.3$\%$,  & 93.9$\%$,  & 98.2$\%$,  & 52.1$\%$\\
\bottomrule
\end{tabular}}
\caption{Performance comparison between the choice rules on the AIR tools dataset.
The first column indicates the problem, the second gives the SNR of the data, the third shows the oracle $\ell_2$-norm error, $\varepsilon_o$, and the 4th--11th columns: the median efficiency of the methods.\label{large-scale-table}}
\end{table}

\section{Conclusion}\label{sec:concl}
In this work, we have proposed a new criterion, termed as \texttt{PRO}, to choose the crucial Tikhonov regularization
parameter for discrete linear inverse problems. It is based on minimizing a lower bound of the predictive risk, and can handle
effectively both known and unknown noise levels. In the latter case, we proposed an iterative scheme, i.e.,
\texttt{I-PRO}, which alternates between estimating the noise level and minimizing the predictive risk. Extensive
numerical simulations show that both \texttt{PRO} and \texttt{I-PRO} are not only competitive with six existing choice
rules including discrepancy principle, unbiased predictive risk estimator, balancing principle, generalized cross
validation, L-curve criterion and quasi-optimality criterion) in terms of accuracy, but also can yield more stable results
in terms of reliability for small-sized samples. Moreover, the methods apply also to large-scale inverse problems, and
can produce solutions with improved accuracy for data with low signal-to-noise ratio. One preliminary
theoretical analysis was provided to show several interesting properties of \texttt{PRO} and \texttt{I-PRO}, and
the compelling empirical results promote further analysis, especially regularizing property and convergence rates,
as well as developing extensions to other linear regularization techniques.

\section*{Acknowledgements}
The research leading to these results has received funding from the European Union's Horizon2020 research and innovation
programme under grant agreement  no.~640216. Federico Benvenuto thanks the National Group of Scientific Computing
(GNCS-INDAM) that supported this research. The work of B. Jin is partially supported by UK EPSRC EP/T000864/1.

\bibliographystyle{abbrv}
\bibliography{pro-3}

\end{document}